\documentclass[ejs,preprint]{imsart}
\usepackage{amsfonts}
\usepackage{amsmath}
\usepackage[OT1]{fontenc}
\usepackage{graphicx}

\arxiv{math.ST/1202.1399}
\newtheorem{remark}{Remark}
\newtheorem{condition}{Condition}
\newtheorem{theorem}{Theorem}

\newtheorem{corollary}[theorem]{Corollary}

\newtheorem{definition}{Definition}
\newtheorem{example}[theorem]{Example}

\newtheorem{lemma}[theorem]{Lemma}

\newtheorem{proposition}[theorem]{Proposition}

\newenvironment{proof}[1][Proof]{\textbf{#1.} }{\ \rule{0.5em}{0.5em}}
\begin{document}

\begin{frontmatter}
\title{Needlet-Whittle Estimates on the Unit Sphere}
\runtitle{Needlet-Whittle Estimates on the Unit Sphere}

\begin{aug}
  \author{Claudio Durastanti\ead[label=e1]{%
durastan@mat.uniroma2.it}},
\address{University of Tor Vergata, Roma, and
University of Pavia,\\
           \printead{e1}}
  \author{Xiaohong Lan%
\ead[label=e2]{xhlan@ustc.edu.cn}}
\address{ University of Science and
Technology of China (Hefei, Anhui, China 230026) and University of
Connecticut,\\
           \printead{e2}}
  \author{Domenico Marinucci%
\thanksref{t2}\ead[label=e3]{marinucc@mat.uniroma2.it}}
  \address{%
University of Tor Vergata, Roma,\\
          \printead{e3}}

\thankstext{t2}{Corresponding author.}

  \runauthor{Durastanti et al.}

\end{aug}

\begin{abstract}
We study the asymptotic behaviour of
needlets-based approximate maximum
likelihood estimators for the spectral
parameters of Gaussian and isotropic
spherical random fields. We prove
consistency and asymptotic Gaussianity, in
the high-frequency limit, thus
generalizing earlier results by Durastanti et
al. (2011) based upon
standard Fourier analysis on the sphere. The
asymptotic results are then
illustrated by an extensive Monte Carlo study.
\end{abstract}

\begin{%
keyword}[class=AMS]
\kwd{62M15}
\kwd{62M30}
\kwd{60G60}
\kwd{42C40}
\end{keyword}


\tableofcontents

\end{frontmatter}


\section{Introduction\label{Sec: Introduction}}

In a recent paper, (see \cite{dlm}, \cite{durastanti}) we investigated the
asymptotic behaviour of a Whittle-like approximate maximum likelihood
procedure for the estimation of the spectral parameters (e.g., the \emph{%
spectral index}) of isotropic Gaussian random fields defined on the unit
sphere $\mathbb{S}^{2}.$ Under Gaussianity, it was indeed possible to
establish consistency and a central limit theorem allowing for feasible
inference, under broad conditions on the behaviour of the angular power
spectrum. These results, as many others concerning statistical inference on
spherical random fields, have recently found strong motivations arising from
applications, especially in a Cosmological framework (see for instance \cite%
{marpecbook} and the references therein). From the technical point of view,
the asymptotic framework is rather different from usual, as it is based on
observations collected at higher and higher frequencies on a fixed-domain
(the unit sphere): even consistency of the angular power-spectrum estimator
under such circumstances becomes non-trivial, and largely open for research,
see also \cite{marpec2}.

The procedure considered in \cite{dlm} is based upon spherical harmonics and
classical Fourier analysis on the sphere. Despite its appealing properties,
it must be stressed that in many practical circumstances spherical harmonics
may suffer serious drawbacks, due to their lack of localization in real
space. This can make their implementation infeasible, due to the presence of
unobserved regions on the sphere (as is commonly the case for Cosmological
applications), and it may exclude the possibility of separate estimation on
different hemispheres, as considered for instance by \cite{bkmpBer}, \cite%
{pietrobon2}. In view of these issues, it is natural to investigate the
possibility to extend Whittle-like procedures to a spherical wavelet
framework, so as to exploit the double-localization properties (in real and
harmonic space) of such constructions. This is the purpose of this paper,
where we shall focus in particular on spherical needlets.

Spherical needlets are a form of second-generation wavelets on the sphere,
introduced in 2006 by \cite{npw1} and \cite{npw2}, and very extensively
exploited both in the statistical literature and for astrophysical
applications in the last few years. Stochastic properties of needlets when
used to estimate spherical random fields are developed in \cite{bkmpAoS},
\cite{bkmpBer}, \cite{bkmpAoSb} \cite{ejslan}, \cite{spalan} and \cite%
{mayeli}. Needlets have been generalized either to an unbounded support in
the frequency domain (Mexican needlets) by \cite{gm1}, \cite{gm2} and \cite%
{gm3}, and to the case of spin fiber bundles (spin needlets, see \cite%
{gelmar}, and mixed needlets \cite{gelmar2010}), again in view of
Cosmological applications such as weak gravitational lensing and the
so-called polarization of the Cosmic Microwave Background (CMB) radiation.
Concerning the latter, applications to CMB temperature and polarization data
are presented for instance by \cite{bkmpAoS}, \cite{cama}, \cite{dmg}, \cite%
{fay08}, \cite{glm}, \cite{ghmkp}, \cite{mpbb08}, \cite{pietrobon1}, \cite%
{pietrobon2}, \cite{rudjord2}. Indeed, as described for instance in \cite%
{dode2004} and \cite{cama}, satellite missions such \emph{WMAP } and \emph{%
Planck} (see http://map.gsfc.nasa.gov/) are providing huge datasets on CMB,
usually assumed to be a realization of an isotropic, Gaussian spherical
random field. Parameter estimation has been considered by many applied
papers (see \cite{hamann} for a review), but in our knowledge until now no
rigorous asymptotic result has so far been produced on these procedures. We
refer also to \cite{bkmpAoS}, \cite{fay08}, \cite{glm}, \cite{pbm06}, \cite%
{pietrobon1}, \cite{marpec2} for further theoretical and applied results on
angular power spectrum estimation, in a purely nonparametric setting, and to
\cite{kerkyphampic}, \cite{kim}, \cite{kimkoo}, \cite{kookim}, \cite{jin},
\cite{leonenko1}, \cite{leonenko2}, \cite{ejslan} and \cite{marpecbook} for
further results on statistical inference for spherical random fields or
wavelets applied to CMB.

As mentioned earlier, the asymptotic framework we are considering here is
rather different from usual: we assume we are observing a single realization
of an isotropic field, the asymptotics being implemented with respect to the
higher and higher resolution level data becoming available. In view of this,
our paper is to some extent related to the growing area of fixed-domain
asymptotics (see for instance \cite{anderes}, \cite{guo}, \cite{loh}, \cite%
{stein}); on the other hand, as for \cite{dlm} some of the techniques
exploited here are close to those adopted by \cite{Robinson}, where
semiparametric estimates of the long memory parameter for covariance
stationary processes are analyzed. In terms of angular power spectrum
behaviour, we shall also allow for semiparametric models where only the
high-frequency/small-scale behaviour of the random field is constrained. In
particular, we consider both full-band and narrow-band estimates, the latter
entailing a slower rate of convergence but allowing for unbiased estimation
under more general circumstances.

The plan of the paper is as follows: in Section \ref{subsec:
gaussrandomfiels}, we will recall briefly some well-known background
material on needlet analysis for spherical isotropic random fields; in
Section \ref{sec: needwhittle} we will introduce and motivate the
Whittle-like minimum contrast estimators. In Section \ref{asproperties} we
shall establish the asymptotic properties of these estimators, in particular
weak consistency and Gaussianity, while in Section \ref{narrow} we present
results on narrow band estimates. Some Monte Carlo evidence on performance
and comparisons with the procedures in \cite{dlm} are collected in Section %
\ref{numerical}, while some auxiliary technical results are collected in the
Appendix.

\section{Spherical Random Fields and Angular Power Spectrum \label{subsec:
gaussrandomfiels}}

It is a well-known fact in Fourier analysis that the set of spherical
harmonics $\left\{ Y_{lm}:l\geq 0,m=-l,...,l\right\} $ represents an
orthonormal basis for the space $L^{2}\left( \mathbb{S}^{2}\right) $, the
class of square-integrable functions on the unit sphere (see for instance
\cite{adler}, \cite{steinweiss}, \cite{leonenko1}, \cite{marpecbook}, for
more details, and \cite{leosa}, \cite{mal} for extensions). Spherical
harmonics are defined as the eigenfunctions of the spherical Laplacian $%
\Delta _{S^{2}},$ e.g. $\Delta _{S^{2}}Y_{lm}=-l(l+1)Y_{lm}$, see again \cite%
{steinweiss}, \cite{VMK} and \cite{marpecbook} for more discussion and
analytic expressions. The spherical needlets (\cite{npw1}, \cite{npw2}) are
defined as
\begin{equation}
\psi _{jk}\left( x\right) =\sqrt{\lambda _{jk}}\sum_{l}b\left( \frac{l}{B^{j}%
}\right) \sum_{m=-l}^{l}\overline{Y}_{lm}\left( x\right) Y_{lm}\left( \xi
_{jk}\right) \text{ ,}  \label{need-def}
\end{equation}%
where $\left\{ \xi _{jk}\right\} $ is a set of cubature points on the
sphere, indexed by $j$, the resolution level index, and $k$, the cardinality
of the point over the fixed resolution level, while $\lambda _{jk}>0$ is the
weight associated to any $\xi _{jk}$ (see also e.g. \cite{bkmpAoSb}) and
\cite{marpecbook}). Let $N_{j}$ denote the number of cubature points for a
given level $j$; as discussed by (\cite{npw1}, \cite{npw2}), cubature points
and weights can be chosen to satisfy
\begin{equation}
\lambda _{jk}\approx B^{-2j}\text{ },\text{ }N_{j}\approx B^{2j}\text{ ,}
\label{Njdef}
\end{equation}%
where by $a\approx b$, we mean that there exists $c_{1},c_{2}>0$ such that $%
c_{1}a\leq b\leq c_{2}a$. In the sequel, for notational simplicity we shall
assume that there exists a positive constant $c_{B}$ such that $%
N_{j}=c_{B}B^{2j},$ for all scales $j$. In practice, cubature points and
weights can be identified with those evaluated by common packages such as
HealPix (see for instance \cite{bkmpAoS}, \cite{GLESP}, \cite{HEALPIX}).

Viewing $L_{l}(\left\langle x,y\right\rangle )=\sum_{m=-l}^{l}\overline{Y}%
_{lm}\left( x\right) Y_{lm}\left( y\right) $ as a projection operator, (\ref%
{need-def}) can be considered a weighted convolution with a weight function $%
b\left( \cdot \right) ,$ chosen so that the following properties holds (see
\cite{npw1}, \cite{npw2}): for fixed $B>1$, $b\left( \cdot \right) $ has
compact support in $\left[ B^{-1},B\right] $ and therefore $b\left( \frac{l}{%
B^{j}}\right) $ has compact support in $l\in \left[ B^{j-1},B^{j+1}\right] $%
; this implies that needlets have compact support in the harmonic domain.
Moreover, $b\left( \cdot \right) \in C^{\infty }\left( 0,\infty \right) $,
which is pivotal to prove the following quasi-exponential localization
property (see \cite{npw1}): for any $M=1,2,...$ there exists $c_{M}>0$ such
that for any $x\in \mathbb{S}^{2}$,
\begin{equation*}
\left\vert \psi _{jk}\right\vert \leq \frac{c_{M}B^{j}}{\left(
1+B^{j}\arccos \left( \left\langle x,\xi _{jk}\right\rangle \right) \right)
^{M}}\text{ .}
\end{equation*}%
Finally, we have the so-called \emph{partition of unity} property: for $l>B$%
\begin{equation*}
\sum_{j\geq 0}b^{2}\left( \frac{l}{B^{j}}\right) =1\text{ ,}
\end{equation*}%
which allows the establishment of the following reconstruction formula (see
again \cite{npw1}): for $f\in L^{2}\left( \mathbb{S}^{2}\right) $, we have,
in the $L^{2}$ sense:
\begin{equation*}
f(x)=\sum_{j,k}\beta _{jk}\psi _{jk}(x)\text{ ,}
\end{equation*}%
where%
\begin{equation}
\beta _{jk}=\left\langle f,\psi _{jk}\right\rangle _{L_{2}\left( \mathbb{S}%
^{2}\right) }=\int_{\mathbb{S}^{2}}\overline{\psi }_{jk}\left( x\right)
f\left( x\right) dx\text{ .}  \label{needcoeffic}
\end{equation}

Consider now a zero-mean, isotropic Gaussian random field $T:\mathbb{S}%
^{2}\times \Omega \rightarrow \mathbb{R}$; we recall also that for every $%
g\in SO\left( 3\right) $ and $x\in \mathbb{S}^{2}$, a field $T\left( \cdot
\right) $ is isotropic if and only if%
\begin{equation*}
T\left( x\right) \overset{d}{=}T\left( gx\right) \text{ ,}
\end{equation*}%
where the equality holds in the sense of processes. It is again a standard
fact (see e.g. \cite{leonenko1}, \cite{marpecbook}) that the following
spectral representation holds:
\begin{eqnarray}
T\left( x\right) &=&\sum_{l\geq 0}\sum_{m=-l}^{l}a_{lm}Y_{lm}\left( x\right)
\text{ ,}  \label{specrap1} \\
a_{lm} &=&\int_{\mathbb{S}^{2}}T\left( x\right) \overline{Y}_{lm}\left(
x\right) dx\text{ .}  \notag
\end{eqnarray}%
Note that this equality holds in both the $L^{2}\left( \mathbb{S}^{2}\times
\Omega ,dx\otimes \mathbb{P}\right) $ and $L^{2}\left( \mathbb{P}\right) $
senses for every fixed $x\in \mathbb{S}^{2}$. For an isotropic Gaussian
field, the spherical harmonics coefficients $a_{lm}$ are Gaussian complex
random variables such that%
\begin{equation*}
\mathbb{E}\left( a_{lm}\right) =0\text{ , }\mathbb{E}\left( a_{lm}\overline{a%
}_{l_{1}m_{1}}\right) =\delta _{l}^{l_{1}}\delta _{m}^{m_{1}}C_{l}\text{ ,}
\end{equation*}%
where the angular power spectrum $\left\{ C_{l}\text{ , }l=1,2,3,...\right\}
$ fully characterizes the dependence structure under Gaussianity.
Characterizations of the spherical harmonics coefficients under Gaussianity
and isotropy are discussed for instance by \cite{bm}, \cite{marpecbook};
here we simply recall that:
\begin{equation*}
\sum_{m=-l}^{l}\left\vert a_{lm}\right\vert ^{2}\sim C_{l}\times \chi
_{2l+1}^{2}\text{ .}
\end{equation*}%
Hence, given a realization of the random field, an estimator of the angular
power spectrum can be defined as:%
\begin{equation*}
\widehat{C}_{l}=\frac{1}{2l+1}\sum_{m=-l}^{l}\left\vert a_{lm}\right\vert
^{2}\text{ ,}
\end{equation*}%
the so-called empirical angular power spectrum. It is immediately observed
that%
\begin{equation}
\mathbb{E}\left( \widehat{C}_{l}\right) =\frac{1}{2l+1}%
\sum_{m=-l}^{l}C_{l}=C_{l}\text{ , }Var\left( \frac{\widehat{C}_{l}}{C_{l}}%
\right) =\frac{2}{2l+1}\rightarrow 0\text{ for }l\rightarrow +\infty \text{ .%
}  \label{powest2}
\end{equation}

Now recall that the needlet coefficients can be written as%
\begin{equation}
\beta _{jk}=\sqrt{\lambda _{jk}}\sum_{B^{j-1}<l<B^{j+1}}b\left( \frac{l}{%
B^{j}}\right) \sum_{m=-l}^{l}a_{lm}Y_{lm}\left( \xi _{jk}\right) \text{ ,}
\label{needcoeff}
\end{equation}%
where%
\begin{equation*}
\mathbb{E}\left( \beta _{jk}\right) =\sqrt{\lambda _{jk}}%
\sum_{B^{j-1}<l<B^{j+1}}b\left( \frac{l}{B^{j}}\right)
\sum_{m=-l}^{l}Y_{lm}\left( \xi _{jk}\right) \mathbb{E}\left( a_{lm}\right)
=0\text{ .}
\end{equation*}

\begin{remark}
It should be noted that in this paper we consider as observations
directly the needlet coefficients, rather than their measurements
on actual data. While this is clearly a simplifying assumption, we
believe it can be heuristically justified, at least as a first
order approximation, as follows.

The data collection procedure for astrophysical experiments can be
described as consisting of continuous averages around pointing
directions, obtained as

\begin{equation*}
H\left( y\right) =\int_{\mathbb{S}^{2}}T\left( x\right) K\left(
\left\langle x,y\right\rangle \right) dx\,\ \text{.}
\end{equation*}%
Heuristically, the kernel $K(.,.)$ represents the spatial effect
of the measuring antenna; in the astrophysical literature, it is
usually labelled a beam function, which we take to be radially
symmetric, so that it
can be expanded in terms of Legendre polynomials as%
\begin{equation*}
K\left( u\right) =\sum_{l}h_{l}P_{l}\left( u\right) \text{ , }u\in
\left[ -1,1\right] \text{ .}
\end{equation*}%
Standard computations then show that the observed spherical
harmonic coefficients are related to the intrinsic ones by the
simple modulation factor $a^{obs}_{l,m}=h_{l}a_{l,m}$. For our
purposes, the factor $h_{l}$ can be incorporated in the asymptotic
behaviour of the angular power spectrum $C_{l}$, for which we
allow some flexibility, see Conditions
\ref{REGULNEED0}-\ref{REGULNEED3} below.

A more realistic experimental set-up would allow also for the
presence of masked or unobserved regions, e.g. we can consider the
observed field
\begin{equation*}
\widetilde{T}\left( x\right) :=T\left( x\right) M\left( x\right) \text{ ,}
\end{equation*}%
where $M\left( x\right) $ is the mask function, e.g. the indicator
function of the set where observations are actually collected.
However, this more general setting does not really pose
extra-difficulties; indeed, defining
\begin{equation*}
\widetilde{\beta }_{jk}:=\int_{\mathbb{S}^{2}}\overline{\psi }_{jk}\left(
x\right) \widetilde{T}\left( x\right) dx\text{ ;}
\end{equation*}%
it was proven in \cite{bkmpBer} that
\begin{equation*}
\frac{\widetilde{\beta }_{jk}-\beta _{jk}}{\sqrt{Var\left( \beta
_{jk}\right) }}=o_{p}\left(1\right) \text{ ,}
\end{equation*}
for all coefficients outside an arbitrarily small neighbourhood
around the masked regions. Heuristically, this result is stating
that for coefficients centred outside the masked regions, the
presence of missing observations is asymptotically negligible.
This is indeed a major advantage of the needlet analysis, when
compared to standard spherical harmonics transforms.

\end{remark}

As in \cite{bkmpAoS}, we introduce the following regularity condition on the
angular power spectrum:

\begin{condition}[Regularity]
\label{REGULNEED0}The random field $T(x)$ is Gaussian and isotropic with
angular power spectrum such that:%
\begin{equation}
C_{l}=l^{-\alpha _{0}}G(l)>0\text{,}  \label{Cl-reg}
\end{equation}%
where $c_{0}^{-1}\leq G(l)\leq c_{0}$, $\alpha _{0}>2$, for all $l\in
\mathbb{N}$ ,and for every $r\in \mathbb{N}$ there exist $c_{r}>0$ such that:%
\begin{equation*}
\left\vert \frac{d^{r}}{du^{r}}G\left( u\right) \right\vert \leq c_{r}u^{-r}%
\text{,}
\end{equation*}%
for $u\in (0,\infty )$ .
\end{condition}

Condition \ref{REGULNEED0} requires some form of regular variation on the
tail behaviour of the angular power spectrum $C_{l}$. For instance, in the
CMB framework the so-called \emph{Sachs-Wolfe} power spectrum (i.e. the
leading model for fluctuations of the primordial gravitational potential)
takes the form (\ref{Cl-reg}), the spectral index $\alpha _{0}$ capturing
the scale invariance properties of the field itself ($\alpha _{0}$ is
expected to be close to $2$ from theoretical considerations, a prediction so
far in good agreement with observations, see for instance \cite{dode2004}
and \cite{Larson}). In particular, this Condition will be necessary to prove
needlet coefficients (\ref{needcoeffic}) to be asymptotically uncorrelated\
(see \cite{bkmpAoS}). For asymptotic results below, we shall need to
strengthen Condition \ref{REGULNEED0} as in \cite{dlm}, imposing in
particular

\begin{condition}
\label{REGULNEED} Condition \ref{REGULNEED0} holds, and moreover
\begin{equation*}
G\left( l\right) =G_{0}\left( 1+O\left( l^{-1}\right) \right) \text{ ,}
\end{equation*}%
whence
\begin{equation*}
\mathbb{E}\left( \widehat{C}_{l}\right) =G_{0}l^{-\alpha _{0}}\left(
1+O\left( l^{-1}\right) \right) \text{ .}
\end{equation*}
\end{condition}

As we shall show, Condition \ref{REGULNEED} is sufficient to establish
consistency for the estimator we are going to define. We shall also consider
two further assumptions, \ref{REGULNEED2} (which implies \ref{REGULNEED}),
to derive asymptotic Gaussianity, and \ref{REGULNEED3} (which implies \ref%
{REGULNEED2}) to provide a centered limiting distribution, see also \cite%
{dlm} for related assumptions.

\begin{condition}
\label{REGULNEED2} Condition \ref{REGULNEED0} holds and moreover%
\begin{equation*}
G\left( l\right) =G_{0}\left( 1+\kappa l^{-1}+O\left( l^{-2}\right) \right)
\text{ .}
\end{equation*}
\end{condition}

\begin{condition}
\label{REGULNEED3} Condition \ref{REGULNEED0} holds and moreover
\begin{equation*}
G\left( l\right) =G_{0}\left( 1+o\left( l^{-1}\right) \right) \text{ .}
\end{equation*}
\end{condition}

\begin{example}
Condition \ref{REGULNEED0} is satisfied for instance by%
\begin{equation*}
G\left( l\right) =\left\{ \log l\right\} ^{\delta }\frac{P\left( l\right) }{%
Q\left( l\right) }\text{ ,}
\end{equation*}%
where%
\begin{eqnarray*}
P\left( l\right) &=&p_{0}+p_{1}l+p_{2}l^{2}+p_{m}l^{m}\text{ ;} \\
Q\left( l\right) &=&q_{0}+q_{1}l+q_{2}l^{2}+q_{m}l^{m}\text{ }
\end{eqnarray*}%
are two finite order polynomials such that $P\left( l\right) ,Q\left(
l\right) >0$. Condition \ref{REGULNEED2} is then fulfilled for $\delta =0$, $%
\kappa =p_{m-1}/p_{m}-q_{m-1}/q_{m}$, while Condition \ref{REGULNEED3} holds
for $\kappa =0$ (see also \cite{bkmpAoS},\cite{dlm},\cite{mayeli}).
\end{example}

Under Condition \ref{REGULNEED0} we have:%
\begin{equation}
c_{0}B^{\left( 2-\alpha \right) j}\leq \sum_{l}b^{2}\left( \frac{l}{B^{j}}%
\right) C_{l}\frac{2l+1}{4\pi }L_{l}\left( \left\langle x,y\right\rangle
\right) \leq c_{1}B^{\left( 2-\alpha \right) j}.  \label{boundden}
\end{equation}

Indeed%
\begin{equation*}
\frac{1}{B^{\left( 2-\alpha \right) j}}\sum_{l}b^{2}\left( \frac{l}{B^{j}}%
\right) C_{l}\frac{2l+1}{4\pi }L_{l}\left( \left\langle x,y\right\rangle
\right)
\end{equation*}%
\begin{equation*}
=c\int_{\mathbb{S}^{2}}b^{2}\left( x\right) g\left( x\right) x^{1-\alpha
}dx+o_{j}(1)\approx B^{\left( 2-\alpha \right) j}\text{;}
\end{equation*}%
more details can be found the Appendix, Proposition \ref{propKj}.

As mentioned before, in \cite{bkmpAoS} it is proven, in view of (\ref%
{boundden}), that:

\begin{lemma}
\label{Lemmacor}Under Condition \ref{REGULNEED0}, there exists $M>0$ such
that:
\begin{equation*}
\left\vert Cor\left( \beta _{jk},\beta _{jk^{\prime }}\right) \right\vert
\leq \frac{C_{M}}{\left( 1+B^{j}d\left( \left\langle \xi _{jk},\xi
_{jk^{\prime }}\right\rangle \right) \right) ^{M}}\text{ .}
\end{equation*}
\end{lemma}

As a direct consequence of this lemma, needlets coefficients at any finite
distance are asymptotically uncorrelated, and hence asymptotically
independent in the Gaussian case.

Following (\ref{needcoeff}) and \cite{npw1}, we have easily that:%
\begin{eqnarray*}
\sum_{k}\beta _{jk}^{2} &=&\sum_{B^{j-1}<l<B^{j+1}}b^{2}\left( \frac{l}{B^{j}%
}\right) \sum_{m=-l}^{l}\left\vert a_{lm}\right\vert ^{2} \\
&=&\sum_{B^{j-1}<l<B^{j+1}}b^{2}\left( \frac{l}{B^{j}}\right) \left(
2l+1\right) \widehat{C}_{l}\text{ ,}
\end{eqnarray*}%
where, by equation (\ref{powest2}):

\begin{equation}
\mathbb{E}\left( \sum_{k}\beta _{jk}^{2}\right)
=\sum_{B^{j-1}<l<B^{j+1}}b^{2}\left( \frac{l}{B^{j}}\right) \left(
2l+1\right) C_{l}\left( 1+O\left( l^{-1}\right) \right) \text{ .}
\label{expbeta2}
\end{equation}

The following Lemma provides the asymptotic behaviour of $Var\left(
\sum_{k}\beta _{jk}^{2}\right) $.

\begin{lemma}
\label{gavarini} Under Condition \ref{Cl-reg}, we have%
\begin{equation*}
\lim_{j\rightarrow \infty }\frac{1}{B^{2\left( 1-\alpha _{0}\right) j}}%
Var\left\{ \sum_{k=1}^{N_{j}}\beta _{jk}^{2}\right\} =G_{0}^{2}\sigma
^{2}(\alpha _{0},B)\text{ ,}
\end{equation*}%
and%
\begin{eqnarray*}
\lim_{j\rightarrow \infty }\frac{1}{B^{2\left( 1-\alpha _{0}\right) j}}%
Cov\left\{ \sum_{k=1}^{N_{j}}\beta _{jk}^{2},\sum_{k_{1}=1}^{N_{j_{1}}}\beta
_{j_{1},k_{1}}^{2}\right\} &=&G_{0}^{2}\tau _{+}^{2}(\alpha _{0},B)\text{ ,
for }j_{1}=j+1\text{ ,} \\
&=&G_{0}^{2}\tau _{-}^{2}(\alpha _{0},B)\text{ , for }j_{1}=j-1\text{ ,} \\
&=&0\text{ , for }\left\vert j_{1}-j\right\vert \geq 2\text{ ,}
\end{eqnarray*}%
where%
\begin{eqnarray*}
\sigma ^{2}(\alpha _{0},B) &:&=4\int_{B^{-1}}^{B}b^{4}(x)x^{1-2\alpha _{0}}dx%
\text{ , } \\
\tau _{+}^{2}(\alpha _{0},B)\text{ } &:&=4\int_{1}^{B}b^{2}(x)b^{2}(\frac{x}{%
B})x^{1-2\alpha _{0}}dx\text{ ,} \\
\tau _{-}^{2}(\alpha _{0},B)\text{ } &:&=\frac{4}{B^{2-2\alpha _{0}}}%
\int_{1}^{B}b^{2}(x)b^{2}(\frac{x}{B})x^{1-2\alpha _{0}}dx\text{ .}
\end{eqnarray*}
\end{lemma}

For the sake of notational simplicity, in the sequel we shall write $\sigma
^{2},\tau _{+}^{2},\tau _{-}^{2}$ (omitting the dependence on $\alpha
_{0},B) $ whenever this does not entail any risk of confusion.

\begin{proof}
Simple calculations based on (\ref{powest2}) lead to:
\begin{equation*}
Var\left\{ \sum_{k=1}^{N_{j}}\beta _{jk}^{2}\right\}
=2\sum_{l=B^{j-1}}^{B^{j+1}}b^{4}(\frac{l}{B^{j}})\left( 2l+1\right)
C_{l}^{2}\text{ ,}
\end{equation*}%
and, for $j_{1}<j_{2}$,%
\begin{equation*}
Cov\left\{ \sum_{k_{1}=1}^{N_{j}}\beta
_{j_{1},k_{1}}^{2},\sum_{k_{2}=1}^{N_{j}}\beta _{j_{2},k_{2}}^{2}\right\}
\end{equation*}%
\begin{equation*}
=\sum_{l_{1}=B^{j_{1}-1}}^{B^{j_{1}+1}}%
\sum_{l_{2}=B^{j_{2}-1}}^{B^{j_{2}+1}}b^{2}\left( \frac{l_{1}}{B^{j_{1}}}%
\right) b^{2}\left( \frac{l_{2}}{B^{j_{2}}}\right) \left( 2l_{1}+1\right)
\left( 2l_{2}+1\right) Cov\left\{ \widehat{C}_{l_{1}},\widehat{C}%
_{l_{2}}\right\}
\end{equation*}%
\begin{equation*}
=\left\{
\begin{array}{c}
\sum_{l=B^{j_{2}-1}}^{B^{j_{1}+1}}b^{2}(\frac{l}{B^{j_{1}}})b^{2}(\frac{l}{%
B^{j_{2}}})\left( 2l+1\right) 2C_{l}^{2}\text{ , for }j_{2}=j_{1}\pm 1 \\
0\text{ , otherwise}%
\end{array}%
\right. \text{ .}
\end{equation*}

Finally, we have:%
\begin{equation*}
\lim_{j\rightarrow \infty }\frac{2}{B^{j(2-2\alpha _{0})}}%
\sum_{l=B^{j-1}}^{B^{j+1}}b^{4}\left( \frac{l}{B^{j}}\right) \left(
2l+1\right) C_{l}^{2}
\end{equation*}%
\begin{eqnarray*}
&=&4G_{0}^{2}\lim_{j\rightarrow \infty
}\sum_{l=B^{j-1}}^{B^{j+1}}b^{4}\left( \frac{l}{B^{j}}\right) \frac{l}{B^{j}}%
\frac{l^{-2\alpha _{0}}}{B^{-2\alpha _{0}}}\frac{1}{B^{j}} \\
&=&G_{0}^{2}\left( 4\int_{B^{-1}}^{B}b^{4}\left( x\right) x^{1-2\alpha
_{0}}dx\right) =:G_{0}^{2}\sigma ^{2}\text{ ,}
\end{eqnarray*}%
and, if $j_{2}=j+1$
\begin{eqnarray*}
&&\lim_{j\rightarrow \infty }\frac{2}{B^{j\left( 2-2\alpha _{0}\right) }}%
\sum_{l=B^{j}}^{B^{j+1}}b^{2}\left( \frac{l}{B^{j}}\right) b^{2}\left( \frac{%
l}{B^{j+1}}\right) \left( 2l+1\right) C_{l}^{2} \\
&=&G_{0}^{2}\left( 4\int_{1}^{B}b^{2}\left( x\right) b^{2}\left( \frac{x}{B}%
\right) x^{1-2\alpha _{0}}dx\right) \text{ }=:G_{0}^{2}\tau _{+}^{2}\text{ ,}
\end{eqnarray*}%
while if $j_{2}=j-1$%
\begin{eqnarray*}
&&\lim_{j\rightarrow \infty }\frac{2}{B^{j\left( 2-2\alpha _{0}\right) }}%
\sum_{l=B^{j-1}}^{B^{j}}b^{2}\left( \frac{l}{B^{j}}\right) b^{2}\left( \frac{%
l}{B^{j-1}}\right) \left( 2l+1\right) C_{l}^{2} \\
&=&G_{0}^{2}\left( 4\int_{B^{-1}}^{1}b^{2}\left( x\right) b^{2}\left(
Bx\right) x^{1-2\alpha _{0}}dx\right) \text{ }= \\
&=&G_{0}^{2}\text{ }\left( \frac{4}{B^{2-2\alpha _{0}}}\int_{1}^{B}b^{2}%
\left( x\right) b^{2}\left( \frac{x}{B}\right) x^{1-2\alpha _{0}}dx\right)
=:G_{0}^{2}\tau _{-}^{2},
\end{eqnarray*}%
as claimed.
\end{proof}

\section{ A Needlet Whittle-like approximation to the likelihood function
\label{sec: needwhittle}}

Our aim in this Section is to discuss heuristically a needlet Whittle-like
approximation for the log-likelihood of isotropic spherical Gaussian fields,
and to derive the corresponding estimator. We start from the assumption that
needlet coefficients can be evaluated exactly, i.e. without observational or
numerical error, up to resolution level $J_{L}$. This is clearly a
simplified picture, analogous to what we assumed in \cite{dlm} for the case
of spherical harmonic coefficients; however in the wavelet case the
assumption can be considered much more realistic. Indeed, it is shown for
instance in \cite{bkmpAoS} that the effect of masked or unobserved regions
is asymptotically negligible, in view of the localization properties of the
needlet transform. Hence we believe our results provide a useful guidance
also for realistic experimental situations. Needless to say, the maximal
observed scale $J_{L}$ grows larger and larger when more sophisticated
experiments are undertaken: indeed $J_{L}$ is a monotonically increasing
function of the maximal observed multipole $L$. The latter is for instance
in the order of 500/600 for data collected from $WMAP$ and 1500/2000 for
those from $Planck$. In terms of our following discussion, it is harmless to
envisage that $B^{J_{L}+1}=L$. The analysis of frequency-domain approximate
maximum likelihood estimators based on spherical harmonics is described in
\cite{dlm}, while narrow-band, wavelet-based maximum likelihood estimators
over $\mathbb{R}$ can be found in \cite{Mrt}.

To motivate heuristically our objective function, consider the vector of
coefficients%
\begin{equation*}
\overrightarrow{\beta }_{j}=\left( \beta _{j1},\beta _{j2},...,\beta
_{jN_{j}}\right) \text{ .}
\end{equation*}%
Under the hypothesis of isotropy and Gaussianity for $T$, we have that $%
\overrightarrow{\beta }_{j}\sim N\left( 0,\Gamma \right) $, where%
\begin{equation*}
\Gamma =\left[ Cov\left( \beta _{jk},\beta _{jk^{\prime }}\right) \right]
_{k,k^{\prime }}=\sqrt{\lambda _{jk}\lambda _{jk^{\prime }}}%
\sum_{l}b^{2}\left( \frac{l}{B^{j}}\right) \frac{\left( 2l+1\right) }{4\pi }%
C_{l}P_{l}\left( \left\langle \xi _{jk},\xi _{jk\prime }\right\rangle
\right) \text{ .}
\end{equation*}%
In view of Lemma \ref{Lemmacor} and equation (\ref{Njdef}), it is to some
extent natural to consider the approximation
\begin{equation*}
\Gamma \simeq \frac{4\pi }{N_{j}}\left( \sum_{l}b^{2}\left( \frac{l}{B^{j}}%
\right) \frac{\left( 2l+1\right) }{4\pi }C_{l}\right) I_{N_{j}}\text{ ,}
\end{equation*}%
where $I_{N_{j}}$ denotes the $N_{j}\times N_{j}$ identity matrix. We
stress, however, that the present argument is merely heuristic - indeed, for
instance, elements on the first diagonal do not converge to zero. The
approximation however motivates the introduction of the pseudo-likelihood
function:%
\begin{equation*}
\mathcal{L}\left( \vartheta ;\overrightarrow{\beta }_{j}\right) =\left( 2\pi
\right) ^{-N_{j}}\left( \det \Gamma \right) ^{-\frac{1}{2}}\exp \left( -%
\frac{1}{2}\overrightarrow{\beta }^{T}\Gamma ^{-1}\overrightarrow{\beta }%
\right)
\end{equation*}%
and the corresponding log-likelihood as:%
\begin{equation*}
-2\log \mathcal{L}\left( \vartheta ;\overrightarrow{\beta }_{j}\right)
\end{equation*}%
\begin{equation*}
\simeq \sum_{k}\left[ \frac{\beta _{jk}^{2}}{N_{j}^{-1}\sum_{l}b^{2}\left(
\frac{l}{B^{j}}\right) \left( 2l+1\right) C_{l}(\vartheta )}-\log \left(
\frac{\beta _{jk}^{2}}{N_{j}^{-1}\sum_{l}b^{2}\left( \frac{l}{B^{j}}\right)
\left( 2l+1\right) C_{l}(\vartheta )}\right) \right] \text{ ,}
\end{equation*}%
up to an additive constant. The full (pseudo-)likelihood is obtained by
combining together all scales $j$, so that%
\begin{equation*}
-2\log \mathcal{L}\left( \vartheta ;...\overrightarrow{\beta }_{j},...%
\overrightarrow{\beta }_{J_{L}}\right)
\end{equation*}%
\begin{equation*}
\simeq \sum_{j=1}^{J_{L}}\left[ \frac{\sum_{k}\beta _{jk}^{2}}{%
N_{j}^{-1}\sum_{l}b^{2}\left( \frac{l}{B^{j}}\right) \left( 2l+1\right)
C_{l}(\vartheta )}-\sum_{k}\log \left( \frac{\beta _{jk}^{2}}{%
N_{j}^{-1}\sum_{l}b^{2}\left( \frac{l}{B^{j}}\right) \left( 2l+1\right)
C_{l}(\vartheta )}\right) \right] \text{.}
\end{equation*}%
Let us now introduce the following:

\begin{definition}
\label{Kfunctions}For $\alpha \in (2,+\infty \equiv A$, define the function
\begin{equation*}
K_{j}\left( \alpha \right) =\frac{1}{N_{j}}\sum_{l}b^{2}\left( \frac{l}{B^{j}%
}\right) \left( 2l+1\right) l^{-\alpha }\text{,}
\end{equation*}%
with derivatives $K_{j,u}\left( \alpha \right) $ given by%
\begin{equation*}
K_{j,u}\left( \alpha \right) =\frac{d}{d\alpha }K_{j}\left( \alpha \right) =%
\frac{\left( -1\right) ^{u}}{N_{j}}\sum_{l}b^{2}\left( \frac{l}{B^{j}}%
\right) \left( 2l+1\right) l^{-\alpha }\left( \log l\right) ^{u}.
\end{equation*}
\end{definition}

Our objective function will hence be written compactly as:
\begin{eqnarray*}
\mathcal{R}_{J_{L}}(G,\alpha ) &:&=-2\log \mathcal{L}(G,\alpha ;%
\overrightarrow{\beta }_{j}) \\
&=&\sum_{j=1}^{J_{L}}\left[ \frac{\sum_{k}\beta _{jk}^{2}}{GK_{j}\left(
\alpha \right) }-\sum_{k}\log \left( \frac{\beta _{jk}^{2}}{GK_{j}\left(
\alpha \right) }\right) \right] \\
&=&\sum_{j=1}^{J_{L}}\left[ \frac{1}{G}\frac{\sum_{k}\beta _{jk}^{2}}{%
K_{j}\left( \alpha \right) }+N_{j}\log G-\sum_{k}\log \left( \frac{\beta
_{jk}^{2}}{K_{j}\left( \alpha \right) }\right) \right] \text{ .}
\end{eqnarray*}%
More precisely, in view of Condition \ref{REGULNEED} and the discussion in
the previous Section, the following Definition seems rather natural:

\begin{definition}
The \emph{Needlet Spherical Whittle estimator }for the parameters $(\alpha
_{0},G_{0})$ is provided by%
\begin{equation*}
\left( \widehat{\alpha }_{J_{L}},\widehat{G}_{J_{L}}\right) :=\arg
\min_{\alpha \in A,G\in \Gamma }\mathcal{R}_{J_{L}}\left( G,\alpha \right)
\text{ .}
\end{equation*}
\end{definition}

\begin{remark}
To ensure that the estimator exists, as usual we shall assume throughout
this paper that the parameter space is a compact subset of $\mathbb{R}^{2};$
more precisely we take $\alpha \in A=\left[ a_{1},a_{2}\right] ,$ $%
2<a_{1}<a_{2}<\infty ,$ and $G\in \Gamma =\left[ \gamma _{1},\gamma _{2}%
\right] ,$ $0<\gamma _{1}<\gamma _{2}<\infty .$ This is little more than a
formal requirement that is standard in the literature on (pseudo-)maximum
likelihood estimation.
\end{remark}

We can rewrite in a more transparent form the previous estimator following
an argument analogous to \cite{Robinson}, i.e. \textquotedblleft
concentrating out\textquotedblright\ the parameter $G$. Indeed, the previous
minimization problem is equivalent to consider%
\begin{equation*}
\left( \widehat{\alpha }_{J_{L}},\widehat{G}_{J_{L}}\right) :=\arg
\min_{\alpha ,G}\mathcal{R}_{J_{L}}\left( G,\alpha \right) \text{ .}
\end{equation*}%
It is readily seen that:%
\begin{equation*}
\frac{\partial }{\partial G}\mathcal{R}_{J_{L}}\left( G,\alpha \right)
=\sum_{j=1}^{J_{L}}\left[ -\frac{1}{G^{2}}\frac{\sum_{k}\beta _{jk}^{2}}{%
K_{j}\left( \alpha \right) }+\frac{N_{j}}{G}\right]
\end{equation*}%
\begin{equation}
\frac{\partial }{\partial G}\mathcal{R}_{J_{L}}\left( G,\alpha \right)
=0\Longleftrightarrow G=\widehat{G}\left( \alpha \right) :=\frac{1}{%
\sum_{j=1}^{J_{L}}N_{j}}\sum_{j=1}^{J_{L}}\frac{\sum_{k=1}^{N_{j}}\beta
_{jk}^{2}}{K_{j}\left( \alpha \right) }  \label{G_est}
\end{equation}%
and because
\begin{equation*}
\frac{\partial ^{2}}{\partial G^{2}}\mathcal{R}_{J_{L}}(G,\alpha )=\frac{1}{%
G^{2}}\sum_{j=1}^{J_{L}}\sum_{k=1}^{N_{j}}\left[ \frac{2}{G}\frac{\beta
_{jk}^{2}}{K_{j}\left( \alpha \right) }-1\right] \text{,}
\end{equation*}%
we obtain%
\begin{eqnarray*}
\left. \frac{\partial ^{2}}{\partial G^{2}}\mathcal{R}_{J_{L}}(G,\alpha
)\right\vert _{G=\widehat{G}\left( \alpha \right) } &=&\frac{1}{\widehat{G}%
\left( \alpha \right) ^{2}}\sum_{j=1}^{J_{L}}\sum_{k=1}^{N_{j}}\left[ \frac{2%
}{\widehat{G}\left( \alpha \right) }\frac{\beta _{jk}^{2}}{K_{j}\left(
\alpha \right) }-1\right] \\
&=&\frac{\sum_{j=1}^{J_{L}}N_{j}}{\widehat{G}\left( \alpha \right) ^{2}}>0%
\text{ .}
\end{eqnarray*}%
Hence $\widehat{G}\left( \alpha \right) $ maximizes $\mathcal{R}%
_{J_{L}}(G,\alpha )$ for any given value of $\alpha $. It remains to compute
\begin{equation}
\widehat{\alpha }=\arg \min_{\alpha \in A}\left\{ \emph{R}_{J_{L}}\left(
\alpha \right) \right\}  \label{alphaneed}
\end{equation}%
where%
\begin{eqnarray*}
\emph{R}_{J_{L}}\left( \alpha \right) &:&=\frac{\mathcal{R}_{J_{L}}\left(
\widehat{G}\left( \alpha \right) ,\alpha \right) }{\sum_{j}N_{j}}-1 \\
&=&\log \sum_{j=1}^{J_{L}}\frac{\sum_{k}\beta _{jk}^{2}}{K_{j}\left( \alpha
\right) }+\frac{1}{\sum_{j=1}^{J_{L}}N_{j}}\sum_{j=1}^{J_{L}}N_{j}\log
K_{j}\left( \alpha \right) \text{.}
\end{eqnarray*}

\section{Asymptotic properties \label{asproperties}}

In this Section we investigate the asymptotic properties of the estimators $%
\widehat{\alpha }_{J_{L}}$ and $\widehat{G}_{J_{L}}$. We start by studying
their asymptotic consistency: in order to achieve this result, we will apply
a technique developed by \cite{brillinger} and \cite{Robinson}, see also
\cite{dlm}.

\begin{theorem}
\label{consistenza} Under Condition \ref{REGULNEED}, as $J_{L}\rightarrow
\infty $ we have:
\begin{equation*}
\left( \widehat{\alpha }_{J_{L}},\widehat{G}_{J_{L}}\right) \rightarrow
_{p}\left( \alpha _{0},G_{0}\right) \text{ .}
\end{equation*}
\end{theorem}

\begin{proof}
Let us write:%
\begin{equation*}
\Delta \emph{R}_{J_{L}}\left( \alpha ,\alpha _{0}\right) :=\emph{R}%
_{J_{L}}\left( \alpha \right) -\emph{R}_{J_{L}}\left( \alpha _{0}\right)
\end{equation*}%
\begin{equation*}
=\log \frac{\widehat{G}\left( \alpha \right) }{G\left( \alpha \right) }-\log
\frac{\widehat{G}\left( \alpha _{0}\right) }{G\left( \alpha _{0}\right) }+%
\frac{1}{\sum_{j=1}^{J_{L}}N_{j}}\sum_{j=1}^{J_{L}}N_{j}\log \frac{%
K_{j}\left( \alpha \right) }{K_{j}\left( \alpha _{0}\right) }+\log \frac{%
G\left( \alpha \right) }{G\left( \alpha _{0}\right) }
\end{equation*}%
\begin{equation*}
=U_{J_{L}}\left( \alpha ,\alpha _{0}\right) -T_{J_{L}}\left( \alpha ,\alpha
_{0}\right) \text{ ,}
\end{equation*}%
where%
\begin{eqnarray*}
G\left( \alpha \right) &:&=\frac{1}{\sum_{j}N_{j}}\sum_{j=1}^{J_{L}}N_{j}%
\frac{G_{0}K_{j}\left( \alpha _{0}\right) }{K_{j}\left( \alpha \right) }%
\text{ ,} \\
T_{J_{L}}\left( \alpha ,\alpha _{0}\right) &:&=\log \frac{\widehat{G}\left(
\alpha _{0}\right) }{G\left( \alpha _{0}\right) }-\log \frac{\widehat{G}%
\left( \alpha \right) }{G\left( \alpha \right) } \\
U_{J_{L}}\left( \alpha ,\alpha _{0}\right) &:&=\frac{1}{%
\sum_{j=1}^{J_{L}}N_{j}}\sum_{j=1}^{J_{L}}N_{j}\log \frac{K_{j}\left( \alpha
\right) }{K_{j}\left( \alpha _{0}\right) }+\log \frac{G\left( \alpha \right)
}{G\left( \alpha _{0}\right) }\text{ .}
\end{eqnarray*}%
It is easy to see that:
\begin{equation*}
G\left( \alpha _{0}\right) =G_{0}\text{ , }\log \frac{G\left( \alpha \right)
}{G\left( \alpha _{0}\right) }=\log \frac{1}{\sum_{j}N_{j}}%
\sum_{j=1}^{J_{L}}N_{j}\frac{K_{j}\left( \alpha _{0}\right) }{K_{j}\left(
\alpha \right) }\text{ .}
\end{equation*}

The proof is then completed with the aid of the auxiliary Lemmas \ref%
{consistency1}, \ref{consistency2} which we shall discuss below. In
particular%
\begin{eqnarray*}
\Pr \left( \left\vert \widehat{\alpha }_{J_{L}}-\alpha _{0}\right\vert
>\varepsilon \right) &\leq &\Pr \left( \inf_{\left\vert \alpha -\alpha
_{0}\right\vert >\varepsilon }\Delta \emph{R}_{J_{L}}(\alpha ,\alpha
_{0})\leq 0\right) \\
&\leq &\Pr \left( \inf_{\left\vert \alpha -\alpha _{0}\right\vert
>\varepsilon }\left[ U_{J_{L}}(\alpha ,\alpha _{0})-T_{J_{L}}(\alpha ,\alpha
_{0})\right] \leq 0\right) \text{ .}
\end{eqnarray*}%
The previous probability is bounded by, for any $\delta >0$%
\begin{equation*}
\Pr \left( \inf_{\left\vert \alpha -\alpha _{0}\right\vert >\varepsilon
}U_{J_{L}}\left( \alpha ,\alpha _{0}\right) \leq \delta \right) +\Pr \left(
\sup_{\left\vert \alpha -\alpha _{0}\right\vert >\varepsilon
}T_{J_{L}}\left( a,\alpha _{0}\right) >0\right) \text{ ; }
\end{equation*}%
for $\alpha _{0}-\alpha <2,$ it is sufficient to note that%
\begin{equation*}
\lim_{L\rightarrow \infty }\Pr \left( \sup_{\left\vert \alpha -\alpha
_{0}\right\vert >\varepsilon }T_{J_{L}}\left( a,\alpha _{0}\right) >0\right)
=0
\end{equation*}%
from Lemma \ref{consistency2}, while from Lemma \ref{consistency1} there
exists $\delta _{\varepsilon }=\frac{B^{2}}{B^{2+\varepsilon }-1}+\frac{%
B^{2}\varepsilon }{B^{2}-1}\log B>0$ such that
\begin{equation*}
\lim_{L\rightarrow \infty }\Pr \left( \inf_{\left\vert \alpha -\alpha
_{0}\right\vert >\varepsilon }U_{J_{L}}\left( \alpha ,\alpha _{0}\right)
\leq \delta _{\varepsilon }\right) =0\text{ .}
\end{equation*}%
For $\alpha _{0}-\alpha =2$ or $\alpha _{0}-\alpha >2$ the same result is
obtained by dividing $\Delta \emph{R}_{J_{L}}(\alpha ,\alpha _{0})$ by,
respectively $\log \log B^{J_{L}}$ or $\log B^{J_{L}}$ and then resorting
again to Lemmas \ref{consistency1}, \ref{consistency2}. Thus $\widehat{%
\alpha }_{J_{L}}\rightarrow _{p}\alpha _{0}$ is established.

Now note that%
\begin{eqnarray*}
\left\vert \widehat{G}(\widehat{\alpha }_{J_{L}})-G_{0}\right\vert &=&\frac{1%
}{\sum_{j}N_{j}}\sum_{j=1}^{J_{L}}\frac{\sum_{k}\beta _{jk}^{2}}{K_{j}\left(
\widehat{\alpha }_{J_{L}}\right) }-\frac{1}{\sum_{j}N_{j}}\sum_{j=1}^{J_{L}}%
\frac{G_{0}K_{j}\left( \alpha _{0}\right) }{K_{j}\left( \alpha _{0}\right) }
\\
&=&\frac{1}{\sum_{j}N_{j}}\sum_{j=1}^{J_{L}}G_{0}\frac{K_{j}\left( \alpha
_{0}\right) }{K_{j}\left( \widehat{\alpha }_{J_{L}}\right) }\left( \frac{%
\sum_{k}\beta _{jk}^{2}}{G_{0}K_{j}\left( \alpha _{0}\right) }-\frac{%
K_{j}\left( \widehat{\alpha }_{J_{L}}\right) }{K_{j}\left( \alpha
_{0}\right) }\right) \text{ .}
\end{eqnarray*}%
By adding and subtracting $I\left( B,\alpha _{0},\widehat{\alpha }%
_{J_{L}}\right) -1$, where $I\left( B,\alpha _{0},\widehat{\alpha }%
_{J_{L}}\right) $ is defined as (\ref{I_def}) in Proposition \ref{propKj},
we obtain:%
\begin{eqnarray*}
\left\vert \widehat{G}(\widehat{\alpha }_{J_{L}})-G_{0}\right\vert &\leq
&\left\vert \frac{1}{\sum_{j}N_{j}}\sum_{j=1}^{J_{L}}G_{0}\frac{K_{j}\left(
\alpha _{0}\right) }{K_{j}\left( \widehat{\alpha }_{J_{L}}\right) }\left(
\frac{\sum_{k}\beta _{jk}^{2}}{G_{0}K_{j}\left( \alpha _{0}\right) }%
-1\right) \right\vert \\
&&+\left\vert \frac{1}{\sum_{j}N_{j}}\sum_{j=1}^{J_{L}}G_{0}\left( \frac{%
K_{j}\left( \alpha _{0}\right) }{K_{j}\left( \widehat{\alpha }%
_{J_{L}}\right) }-I\left( B,\alpha _{0},\widehat{\alpha }_{J_{L}}\right)
\right) \right\vert \\
&&+\left\vert \frac{1}{\sum_{j}N_{j}}\sum_{j=1}^{J_{L}}G_{0}\left( I\left(
B,\alpha _{0},\widehat{\alpha }_{J_{L}}\right) -1\right) \right\vert \\
&=&\left\vert G_{A}\right\vert +\left\vert G_{B}\right\vert +\left\vert
G_{C}\right\vert \text{ .}
\end{eqnarray*}

By Proposition \ref{propKj} we have that
\begin{equation*}
\frac{K_{j}\left( \alpha _{0}\right) }{K_{j}\left( \widehat{\alpha }%
_{J_{L}}\right) }=B^{j\left( \widehat{\alpha }_{J_{L}}-\alpha _{0}\right)
}I\left( B,\alpha _{0},\widehat{\alpha }_{J_{L}}\right) +o_{J_{L}}\left(
1\right) \text{ .}
\end{equation*}

Clearly
\begin{equation*}
\Pr \left\{ \left\vert G_{A}\right\vert \geq \frac{\varepsilon }{3}\right\}
\leq \Pr \left\{ \left[ \left\vert G_{A}\right\vert \geq \frac{\varepsilon }{%
3}\right] \cap \left[ \left\vert \alpha _{0}-\widehat{\alpha }%
_{J_{L}}\right\vert <\frac{1}{3}\right] \right\} +\Pr \left\{ \left\vert
\alpha _{0}-\widehat{\alpha }_{J_{L}}\right\vert \geq \frac{1}{3}\right\}
\end{equation*}

\begin{equation*}
\leq \Pr \left\{ \left[ \frac{G_{0}}{\sum_{j}N_{j}}\sum_{j}I\left( B,\alpha
_{0},\widehat{\alpha }_{J_{L}}\right) B^{j\left( \widehat{\alpha }%
_{J_{L}}-\alpha _{0}\right) }\left\vert \frac{\sum_{k}\beta _{jk}^{2}}{%
G_{0}K_{j}\left( \alpha _{0}\right) }-1\right\vert \geq \varepsilon \right]
\right\} +o_{J_{L}}(1)
\end{equation*}%
\begin{eqnarray*}
&\leq &\frac{1}{\varepsilon }\frac{G_{0}}{\sum_{j}N_{j}}\sum_{j}\sqrt{N_{j}}%
I\left( B,\alpha _{0},\widehat{\alpha }_{J_{L}}\right) B^{j\left( \widehat{%
\alpha }_{J_{L}}-\alpha _{0}\right) }\mathbb{E}\left\vert \frac{1}{\sqrt{%
N_{j}}}\frac{\sum_{k}\beta _{jk}^{2}}{G_{0}K_{j}\left( \alpha _{0}\right) }%
-1\right\vert +o_{J_{L}}(1) \\
&\leq &\frac{1}{\varepsilon }\frac{G_{0}}{\sum_{j}N_{j}}\sum_{j}\sqrt{N_{j}}%
I\left( B,\alpha _{0},\widehat{\alpha }_{J_{L}}\right) B^{j\left( \widehat{%
\alpha }_{J_{L}}-\alpha _{0}\right) }\left[ \mathbb{E}\left\vert \frac{1}{%
\sqrt{N_{j}}}\frac{\sum_{k}\beta _{jk}^{2}}{G_{0}K_{j}\left( \alpha
_{0}\right) }-1\right\vert ^{2}\right] ^{1/2}+o_{J_{L}}(1) \\
&=&\frac{C}{\varepsilon }\frac{G_{0}I\left( B,\alpha _{0},\widehat{\alpha }%
_{J_{L}}\right) }{B^{2J_{L}}}B^{\left( \left( \widehat{\alpha }%
_{J_{L}}-\alpha _{0}\right) +1\right) J_{L}}+o_{J_{L}}(1) \\
&=&\frac{CI\left( B,\alpha _{0},\widehat{\alpha }_{J_{L}}\right) }{%
\varepsilon }G_{0}B^{\left( \left( \widehat{\alpha }_{J_{L}}-\alpha
_{0}\right) -1\right) J_{L}}+o_{J_{L}}(1)=o_{J_{L}}(1)\text{ ,}
\end{eqnarray*}%
in view of the consistency of $\widehat{\alpha }_{J_{L}}$. As far as $G_{B}$
is concerned, we obtain, for a sufficiently small $\delta >0$:

\begin{eqnarray*}
\Pr \left\{ \left\vert G_{B}\right\vert \geq \frac{\varepsilon }{3}\right\}
&=&\Pr \left\{ \left[ \left\vert G_{B}\right\vert \geq \frac{\varepsilon }{3}%
\right] \cap \left[ \log B^{j}\left\vert \alpha _{0}-\widehat{\alpha }%
_{J_{L}}\right\vert <\delta \right] \right\} \\
&&+\Pr \left\{ \left[ \left\vert G_{B}\right\vert \geq \frac{\varepsilon }{3}%
\right] \cap \left[ \log B^{j}\left\vert \alpha _{0}-\widehat{\alpha }%
_{J_{L}}\right\vert \geq \delta \right] \right\} \\
&=&\Pr \left\{ \left[ \left\vert G_{B}\right\vert \geq \frac{\varepsilon }{3}%
\right] \cap \left[ \log B^{j}\left\vert \alpha _{0}-\widehat{\alpha }%
_{J_{L}}\right\vert <\delta \right] \right\} +o_{J_{L}}\left( 1\right) \text{
.}
\end{eqnarray*}%
Because for $0\leq x\leq 1$, we have $\left\vert e^{-x}-1\right\vert \leq x$%
, we have:%
\begin{eqnarray*}
\left\vert B^{-j\left( \alpha _{0}-\widehat{\alpha }_{J_{L}}\right)
}-1\right\vert &=&\left\vert \exp \left( -j\left( \alpha _{0}-\widehat{%
\alpha }_{J_{L}}\right) \log B\right) -1\right\vert \\
&\leq &\left( \alpha _{0}-\widehat{\alpha }_{J_{L}}\right) \log B^{j}\text{ ,%
}
\end{eqnarray*}%
and hence, in view of
\begin{equation*}
\Pr \left\{ \left[ \left\vert G_{B}\right\vert \geq \frac{\varepsilon }{3}%
\right] \cap \left[ \log B^{j}\left\vert \alpha _{0}-\widehat{\alpha }%
_{J_{L}}\right\vert <\delta \right] \right\}
\end{equation*}%
\begin{eqnarray*}
&\leq &\Pr \left\{ \left[ \frac{1}{\sum_{j}N_{j}}\sum_{j=1}^{J_{L}}G_{0}%
\left\vert \frac{K_{j}\left( \alpha _{0}\right) }{K_{j}\left( \widehat{%
\alpha }_{J_{L}}\right) }-1\right\vert \geq \frac{\varepsilon }{3}\right]
\cap \left[ \log B^{j}\left\vert \alpha _{0}-\widehat{\alpha }%
_{J_{L}}\right\vert <\delta \right] \right\} \\
&\leq &\frac{C}{\varepsilon }\frac{I\left( B,\alpha _{0},\widehat{\alpha }%
_{J_{L}}\right) }{\sum_{j}N_{j}}\sum_{j=1}^{J_{L}}G_{0}\log B^{j}\mathbb{E}%
\left\vert \alpha _{0}-\widehat{\alpha }_{J_{L}}\right\vert =\frac{C}{%
\varepsilon }\frac{\delta J_{L}}{\sum_{j}N_{j}}=o_{J_{L}}(1)\text{ .}
\end{eqnarray*}%
Finally, in view of (\ref{I-ratio}),%
\begin{equation*}
\left\vert \frac{1}{\sum_{j}N_{j}}\sum_{j=1}^{J_{L}}G_{0}\left( I\left(
B,\alpha _{0},\widehat{\alpha }_{J_{L}}\right) -1\right) \right\vert \leq
\frac{J_{L}C_{I}\left\vert \alpha _{0}-\widehat{\alpha }_{J_{L}}\right\vert
}{\sum_{j}N_{j}}\text{ .}
\end{equation*}%
Hence, because for a sufficiently small $\delta >0$:%
\begin{eqnarray*}
\Pr \left\{ \left[ \left\vert G_{C}\right\vert \geq \frac{\varepsilon }{3}%
\right] \right\} &\leq &\Pr \left\{ \left[ \left\vert G_{C}\right\vert \geq
\frac{\varepsilon }{3}\right] \cap \left[ \left\vert \alpha _{0}-\widehat{%
\alpha }_{J_{L}}\right\vert <\delta \right] \right\} +\Pr \left\{ \left\vert
\alpha _{0}-\widehat{\alpha }_{J_{L}}\right\vert \geq \delta \right\} \\
&=&\Pr \left\{ \left[ \left\vert G_{C}\right\vert \geq \frac{\varepsilon }{3}%
\right] \cap \left[ \left\vert \alpha _{0}-\widehat{\alpha }%
_{J_{L}}\right\vert <\delta \right] \right\} +o_{J_{L}}\left( 1\right)
\end{eqnarray*}%
\begin{equation*}
\Pr \left\{ \left[ \left\vert G_{C}\right\vert \geq \frac{\varepsilon }{3}%
\right] \cap \left[ \left\vert \alpha _{0}-\widehat{\alpha }%
_{J_{L}}\right\vert <\delta \right] \right\} \leq \frac{J_{L}C_{I}\delta }{%
\sum_{j}N_{j}}=o_{J_{L}}\left( 1\right)
\end{equation*}

as claimed.
\end{proof}

Here we present the auxiliary results we shall need on $\widehat{G}$ and its
derivatives. We introduce%
\begin{eqnarray}
\widehat{G}_{0}\left( \alpha \right) &:&=\widehat{G}\left( \alpha \right) =%
\frac{1}{\sum_{j=1}^{J_{L}}N_{j}}\sum_{j=1}^{J_{L}}\frac{\sum_{k}\beta
_{jk}^{2}}{K_{j}\left( \alpha \right) }U_{0,j}\left( \alpha \right) \text{;}
\notag \\
\widehat{G}_{1}\left( \alpha \right) &:&=\frac{d}{d\alpha }\widehat{G}\left(
\alpha \right) =\frac{1}{\sum_{j=1}^{J_{L}}N_{j}}\sum_{j=1}^{J_{L}}\frac{%
\sum_{k}\beta _{jk}^{2}}{K_{j}\left( \alpha \right) }U_{1,j}\left( \alpha
\right) \text{;}  \label{defGest} \\
\widehat{G}_{2}\left( \alpha \right) &:&=\frac{d^{2}}{d\alpha ^{2}}\widehat{G%
}\left( \alpha \right) =\frac{1}{\sum_{j=1}^{J_{L}}N_{j}}\sum_{j=1}^{J_{L}}%
\frac{\sum_{k}\beta _{jk}^{2}}{K_{j}\left( \alpha \right) }U_{2,j}\left(
\alpha \right) \text{,}  \notag
\end{eqnarray}%
where:%
\begin{equation*}
U_{0,j}\left( \alpha \right) =1\text{ , }U_{1,j}\left( \alpha \right) =-%
\frac{K_{j,1}\left( \alpha \right) }{K_{j}\left( \alpha \right) }\text{ , }%
U_{2,j}\left( \alpha \right) =2\left( \frac{K_{j,1}\left( \alpha \right) }{%
K_{j}\left( \alpha \right) }\right) ^{2}-\frac{K_{j,2}\left( \alpha \right)
}{K_{j}\left( \alpha \right) }\text{ .}
\end{equation*}%
Also, let:%
\begin{eqnarray}
G_{0}\left( \alpha \right) &:&=G\left( \alpha \right) =\frac{1}{%
\sum_{j=1}^{J_{L}}N_{j}}\sum_{j=1}^{J_{L}}N_{j}\frac{G_{0}K_{j}\left( \alpha
_{0}\right) }{K_{j}\left( \alpha \right) }U_{0,j}\left( \alpha \right) \text{%
;}  \notag \\
G_{1}\left( \alpha \right) &:&=\frac{d}{d\alpha }G\left( \alpha \right) =%
\frac{1}{\sum_{j=1}^{J_{L}}N_{j}}\sum_{j=1}^{J_{L}}N_{j}\frac{%
G_{0}K_{j}\left( \alpha _{0}\right) }{K_{j}\left( \alpha \right) }%
U_{1,j}\left( \alpha \right) \text{;}  \label{defGfun} \\
G_{2}\left( \alpha \right) &:&=\frac{d^{2}}{d\alpha ^{2}}G\left( \alpha
\right) =\frac{1}{\sum_{j=1}^{J_{L}}N_{j}}\sum_{j=1}^{J_{L}}N_{j}\frac{%
G_{0}K_{j}\left( \alpha _{0}\right) }{K_{j}\left( \alpha \right) }%
U_{2,j}\left( \alpha \right) \text{.}  \notag
\end{eqnarray}
The first result concerns the behaviour of expected value and variance of
the estimator $\widehat{G}\left( \alpha _{0}\right) $ computed in $\alpha
_{0}$, the second regards the uniform convergence in probability of the
ratio between $\widehat{G}$ and $G$, and their $k$-th order derivatives $%
\widehat{G}_{k}$, $G_{k}$.

\begin{lemma}
\label{lemmaG}Let $\widehat{G}\left( \alpha \right) $ be as in (\ref{G_est}%
). Under Condition \ref{Cl-reg}, we have%
\begin{eqnarray*}
\mathbb{E}\left( \widehat{G}\left( \alpha _{0}\right) \right) &=&G_{0}\text{
;} \\
\lim B^{2J_{L}}Var\left( \widehat{G}\left( \alpha _{0}\right) \right)
&=&G_{0}^{2}\rho ^{2}(\alpha _{0},B)\frac{B^{2}-1}{B^{2}}\text{ ,}
\end{eqnarray*}%
where $I_{0}\left( B\right) $ is defined by (\ref{I0_def}) in Proposition %
\ref{propKj} and%
\begin{equation*}
\rho ^{2}\left( \alpha _{0};B\right) =\frac{\sigma ^{2}\left( \alpha
_{0};B\right) +B^{-\alpha _{0}}\tau ^{2}\left( \alpha _{0};B\right) }{%
I_{0}^{2}\left( B\right) }\text{ , }\tau ^{2}\left( \alpha _{0};B\right)
:=\tau _{+}^{2}\left( \alpha _{0};B\right) +\tau _{-}^{2}\left( \alpha
_{0};B\right) \text{ .}
\end{equation*}
\end{lemma}

\begin{proof}
By (\ref{expbeta2}), we obtain that
\begin{eqnarray*}
\mathbb{E}\left( \widehat{G}\left( \alpha \right) \right) &=&\frac{1}{%
\sum_{j=1}^{J_{L}}N_{j}}\sum_{j=1}^{J_{L}}\frac{\mathbb{E}\left(
\sum_{k=1}^{N_{j}}\beta _{jk}^{2}\right) }{K_{j}\left( \alpha \right) } \\
&=&\frac{1}{\sum_{j=1}^{J_{L}}N_{j}}G_{0}\sum_{j=1}^{J_{L}}\frac{%
\sum_{l}b^{2}\left( \frac{l}{B^{j}}\right) \frac{2l+1}{4\pi }l^{-\alpha
_{0}}\left( 1+O\left( l^{-1}\right) \right) }{K_{j}\left( \alpha _{0}\right)
} \\
&=&G_{0}+o_{J_{L}}\left( 1\right) \text{ ,}
\end{eqnarray*}%
while from Lemma \ref{gavarini} and Proposition \ref{propKj} (see also the
proof of Lemma \ref{lemmaS0}), we have

\begin{equation*}
Var\left( \widehat{G}\left( \alpha _{0}\right) \right) =G_{0}^{2}\rho
^{2}(\alpha _{0},B)\frac{B^{2}-1}{B^{2}}B^{-2J_{L}}+o\left(
B^{-2J_{L}}\right) \text{ ,}
\end{equation*}%
as claimed.
\end{proof}

\begin{lemma}
\label{lemmagiovneed} Under Condition \ref{REGULNEED} we have for $n=0,1,2$:%
\begin{equation*}
\sup \left| \frac{\widehat{G}_{n}\left( \alpha \right) }{G_{n}\left( \alpha
\right) }-1\right| \longrightarrow _{p}0\text{ .}
\end{equation*}
\end{lemma}

\begin{proof}
Under Condition \ref{REGULNEED2}, we observe that:

\begin{equation*}
\frac{\widehat{G}_{n}\left( \alpha \right) }{G_{n}\left( \alpha \right) }-1=%
\frac{\sum_{j=1}^{J_{L}}\frac{\sum_{k}\beta _{jk}^{2}}{K_{j}\left( \alpha
\right) }U_{j,n}\left( \alpha \right) }{\sum_{j=1}^{J_{L}}N_{j}\frac{%
G_{0}K_{j}\left( \alpha _{0}\right) }{K_{j}\left( \alpha \right) }%
U_{j,n}\left( \alpha \right) }-1
\end{equation*}%
\begin{equation*}
=\frac{\sum_{j=1}^{J_{L}}\sqrt{N_{j}}\frac{K_{j}\left( \alpha _{0}\right) }{%
K_{j}\left( \alpha \right) }U_{j,n}\left( \alpha \right) \left[ \left( \frac{%
1}{\sqrt{N_{j}}}\sum_{k}\left( \frac{\beta _{jk}^{2}}{G_{0}K_{j}\left(
\alpha _{0}\right) }-1\right) \right) \right] }{\sum_{j=1}^{J_{L}}N_{j}\frac{%
K_{j}\left( \alpha _{0}\right) }{K_{j}\left( \alpha \right) }U_{j,n}\left(
\alpha \right) }\text{ .}
\end{equation*}

Then we have:%
\begin{equation*}
\mathbb{P}\left\{ \left\vert \frac{\sum_{j=1}^{J_{L}}\sqrt{N_{j}}\frac{%
K_{j}\left( \alpha _{0}\right) }{K_{j}\left( \alpha \right) }U_{j,n}\left(
\alpha \right) \left[ \left( \frac{1}{\sqrt{N_{j}}}\sum_{k}\left( \frac{%
\beta _{jk}^{2}}{G_{0}K_{j}\left( \alpha _{0}\right) }-1\right) \right) %
\right] }{\sum_{j=1}^{J_{L}}N_{j}\frac{K_{j}\left( \alpha _{0}\right) }{%
K_{j}\left( \alpha \right) }U_{j,n}\left( \alpha \right) }\right\vert
>\delta _{\varepsilon }\right\}
\end{equation*}%
\begin{equation*}
\leq \mathbb{P}\left( J_{L}^{2}\left\vert \frac{\sum_{j}\sqrt{N_{j}}\frac{%
K_{j}\left( \alpha _{0}\right) }{K_{j}\left( \alpha \right) }U_{j,n}\left(
\alpha \right) }{\sum_{j}N_{j}\frac{K_{j}\left( \alpha _{0}\right) }{%
K_{j}\left( \alpha \right) }U_{j,n}\left( \alpha \right) }\right\vert \frac{%
\sup_{j}\left\vert \frac{1}{\sqrt{N_{j}}}\sum_{k}\left( \frac{\beta _{jk}^{2}%
}{G_{0}K_{j}\left( \alpha _{0}\right) }-1\right) \right\vert }{J_{L}^{2}}%
>\delta _{\varepsilon }\right) \text{ .}
\end{equation*}

In view of Proposition \ref{propKj} and equations (\ref{K1lim}) and (\ref%
{K2lim}) in Corollary \ref{Kjcoroll}, described in the Appendix, we have
\begin{eqnarray*}
U_{j,1}\left( \alpha \right) &=&\left( -\frac{K_{j,1}\left( \alpha \right) }{%
K_{j}\left( \alpha \right) }\right) =\log B^{j}+o_{j}\left( 1\right) \\
U_{j,2}\left( \alpha \right) &=&\left( 2\frac{\left( K_{j,1}\left( \alpha
\right) \right) ^{2}}{\left( K_{j}\left( \alpha \right) \right) ^{2}}-\frac{%
K_{j,2}\left( \alpha \right) }{K_{j}\left( \alpha \right) }\right) =\left(
\log B^{j}\right) ^{2}+o_{j^{2}}\left( 1\right) \text{.}
\end{eqnarray*}

Then,
\begin{equation*}
\frac{\sum_{j=1}^{J_{L}}\sqrt{N_{j}}\frac{K_{j}\left( \alpha _{0}\right) }{%
K_{j}\left( \alpha \right) }U_{j,n}\left( \alpha \right) }{%
\sum_{j=1}^{J_{L}}N_{j}\frac{K_{j}\left( \alpha _{0}\right) }{K_{j}\left(
\alpha \right) }U_{j,n}\left( \alpha \right) }=\frac{%
\sum_{j=1}^{J_{L}}B^{j}B^{\left( \alpha -\alpha _{0}\right) j}j^{n}}{%
\sum_{j=1}^{J_{L}}B^{2j}B^{\left( \alpha -\alpha _{0}\right) j}j^{n}}%
=O(B^{-J_{L}})\,\text{,}
\end{equation*}%
so that
\begin{equation*}
\sup_{L}\left\vert J_{L}^{2}\frac{\sum_{j=1}^{J_{L}}\sqrt{N_{j}}\frac{%
K_{j}\left( \alpha _{0}\right) }{K_{j}\left( \alpha \right) }U_{j,n}\left(
\alpha \right) }{\sum_{j=1}^{J_{L}}N_{j}\frac{K_{j}\left( \alpha _{0}\right)
}{K_{j}\left( \alpha \right) }U_{j,n}\left( \alpha \right) }\right\vert
<+\infty \text{ .}
\end{equation*}

Also, by Markov inequality and Lemma \ref{gavarini} , we have that, for all $%
j$:
\begin{eqnarray*}
\mathbb{P}\left( \left\vert \frac{1}{\sqrt{N_{j}}}\sum_{k}\frac{\beta
_{jk}^{2}}{G_{0}K_{j}\left( \alpha _{0}\right) }-1\right\vert
>CJ_{L}^{2}\right) &\leq &\frac{1}{CJ_{L}^{2}}Var\left( \frac{1}{\sqrt{N_{j}}%
}\sum_{k}\frac{\beta _{jk}^{2}}{G_{0}K_{j}\left( \alpha _{0}\right) }%
-1\right) \\
&=&\frac{1}{CJ_{L}^{2}}\frac{1}{G_{0}^{2}N_{j}}Var\left( \sum_{k}\frac{\beta
_{jk}^{2}}{K_{j}\left( \alpha _{0}\right) }\right) \\
&=&\frac{1}{CJ_{L}^{2}}\rho ^{2}(\alpha _{0},B)=O\left( J_{L}^{-2}\right)
\text{ ,}
\end{eqnarray*}%
whence%
\begin{eqnarray*}
&&\mathbb{P}\left\{ \sup_{j=1,...,J_{L}}\left\vert \frac{1}{\sqrt{N_{j}}}%
\sum_{k}\frac{\beta _{jk}^{2}}{G_{0}K_{j}\left( \alpha _{0}\right) }%
-1\right\vert >CJ_{L}\right\} \\
&\leq &J_{L}\sup_{j=1,...,J_{L}}\mathbb{P}\left\{ \frac{1}{\sqrt{N_{j}}}%
\sum_{k}\left\vert \frac{\beta _{jk}^{2}}{G_{0}K_{j}\left( \alpha
_{0}\right) }-1\right\vert >CJ_{L}\right\} \\
&\leq &J_{L}\times O\left( J_{L}^{-2}\right) \text{ }%
=O(J_{L}^{-1})=o_{J_{L}}(1)\text{ .}
\end{eqnarray*}%
Under Condition \ref{REGULNEED} we have:%
\begin{eqnarray*}
\frac{\widehat{G}_{n}\left( \alpha \right) }{G_{n}\left( \alpha \right) }-1
&=&\frac{\sum_{j=1}^{J_{L}}\sqrt{N_{j}}\frac{K_{j}\left( \alpha _{0}\right)
}{K_{j}\left( \alpha \right) }U_{j,n}\left( \alpha \right) \left[ \left(
\frac{1}{\sqrt{N_{j}}}\sum_{k}\left( \frac{\beta _{jk}^{2}}{G_{0}K_{j}\left(
\alpha _{0}\right) }-\frac{\mathbb{E}\left[ \sum \beta _{jk}^{2}\right] }{%
N_{j}G_{0}K_{j}\left( \alpha _{0}\right) }\right) \right) \right] }{%
\sum_{j=1}^{J_{L}}N_{j}\frac{K_{j}\left( \alpha _{0}\right) }{K_{j}\left(
\alpha \right) }U_{j,n}\left( \alpha \right) } \\
&&+\frac{\sum_{j=1}^{J_{L}}N_{j}\frac{K_{j}\left( \alpha _{0}\right) }{%
K_{j}\left( \alpha \right) }U_{j,n}\left( \alpha \right) \left( \frac{%
\mathbb{E}\left[ \sum_{k}\beta _{jk}^{2}\right] }{N_{j}G_{0}K_{j}\left(
\alpha _{0}\right) }-1\right) }{\sum_{j=1}^{J_{L}}N_{j}\frac{K_{j}\left(
\alpha _{0}\right) }{K_{j}\left( \alpha \right) }U_{j,n}\left( \alpha
\right) }\text{ .}
\end{eqnarray*}%
From (\ref{expvaluebetabias}), it is easy to see that the second
term in the last equation is $O\left( B^{-J_{L}}\right) $, while
for the first term
we follow the same procedure already described, also considering that, by (%
\ref{expvaluebetabias}):%
\begin{equation*}
\left\vert \frac{1}{\sqrt{N_{j}}}\sum_{k}\left( \frac{\beta _{jk}^{2}}{%
G_{0}K_{j}\left( \alpha _{0}\right) }-\frac{\mathbb{E}\left[ \beta _{jk}^{2}%
\right] }{G_{0}K_{j}\left( \alpha _{0}\right) }\right) \right\vert
=\left\vert \frac{1}{\sqrt{N_{j}}}\frac{\mathbb{E}\left[ \beta _{jk}^{2}%
\right] }{G_{0}K_{j}\left( \alpha _{0}\right) }\left( \frac{\beta _{jk}^{2}}{%
\mathbb{E}\left[ \beta _{jk}^{2}\right] }-1\right) \right\vert
\end{equation*}%
\begin{equation*}
=\sqrt{N_{j}}\left( 1+O\left( B^{-j}\right) \right) \left\vert \left( \frac{%
\sum_{k}\beta _{jk}^{2}}{\mathbb{E}\left[ \sum_{k}\beta _{jk}^{2}\right] }%
-1\right) \right\vert \text{ .}
\end{equation*}
\end{proof}

We can establish now the asymptotic behaviour of $U_{J_{L}}\left( \alpha
,\alpha _{0}\right) $, for which we have the following

\begin{lemma}
\label{consistency1}For all $\varepsilon <\alpha _{0}-\alpha <2$%
\begin{equation*}
\lim_{J_{L}\rightarrow \infty }U_{J_{L}}\left( \alpha ,\alpha _{0}\right)
\end{equation*}%
\begin{equation*}
=\lim_{J_{L}\rightarrow \infty }\log \frac{1}{\sum_{j}N_{j}}%
\sum_{j=1}^{J_{L}}N_{j}\frac{K_{j}\left( \alpha _{0}\right) }{K_{j}\left(
\alpha \right) }-\frac{1}{\sum_{j=1}^{J_{L}}N_{j}}\sum_{j=1}^{J_{L}}N_{j}%
\log \frac{K_{j}\left( \alpha _{0}\right) }{K_{j}\left( \alpha \right) }
\end{equation*}%
\begin{equation*}
=\log \frac{B^{2}-1}{B^{2+(\alpha -\alpha _{0})}-1}+\frac{B^{2}(\alpha
-\alpha _{0})}{B^{2}-1}\log B>\delta _{\varepsilon }>0\text{ .}
\end{equation*}%
Moreover, if $\alpha -\alpha _{0}\,<-2$, we have
\begin{equation*}
\lim_{J_{L}\rightarrow \infty }\frac{1}{\log B^{J_{L}}}\left\{ \log \frac{1}{%
\sum_{j}N_{j}}\sum_{j=1}^{J_{L}}N_{j}\frac{K_{j}\left( \alpha _{0}\right) }{%
K_{j}\left( \alpha \right) }-\frac{1}{\sum_{j=1}^{J_{L}}N_{j}}%
\sum_{j=1}^{J_{L}}N_{j}\log \frac{K_{j}\left( \alpha _{0}\right) }{%
K_{j}\left( \alpha \right) }\right\}
\end{equation*}%
\begin{equation*}
=\frac{\alpha _{0}-\alpha }{2}-1>0\text{ ,}
\end{equation*}%
and if $\alpha -\alpha _{0}\,-2$%
\begin{equation*}
\lim_{J_{L}\rightarrow \infty }\frac{1}{\log J_{L}}\left\{ \log \frac{1}{%
\sum_{j}N_{j}}\sum_{j=1}^{J_{L}}N_{j}\frac{K_{j}\left( \alpha _{0}\right) }{%
K_{j}\left( \alpha \right) }-\frac{1}{\sum_{j=1}^{J_{L}}N_{j}}%
\sum_{j=1}^{J_{L}}N_{j}\log \frac{K_{j}\left( \alpha _{0}\right) }{%
K_{j}\left( \alpha \right) }\right\} =1\text{ .}
\end{equation*}
\end{lemma}

\begin{proof}
By recalling $N_{j}=c_{B}B^{2j}$ (see (\ref{Njdef})) and (\ref{K0ratio}), we
observe that,%
\begin{equation*}
\frac{1}{\sum_{j}N_{j}}\sum_{j=1}^{J_{L}}N_{j}\log \left\{ \frac{K_{j}\left(
\alpha _{0}\right) }{K_{j}\left( \alpha \right) }\right\}
\end{equation*}%
\begin{eqnarray*}
&=&\frac{B^{2}-1}{B^{2J_{L}}B^{2}}\sum_{j=1}^{J_{L}}B^{2j}\log \left\{ \frac{%
B^{\left( 2-\alpha _{0}\right) j}}{B^{\left( 2-\alpha \right) j}}\right\}
+\log \left( I\left( B,\alpha _{0},\alpha \right) \right) +o_{J_{L}}(1) \\
&=&\frac{B^{2}-1}{B^{2J_{L}}B^{2}}(\alpha -\alpha
_{0})\sum_{j=1}^{J_{L}}B^{2j}\log B^{j}+o_{J_{L}}(1)
\end{eqnarray*}%
\begin{equation*}
=\frac{B^{2}-1}{B^{2J_{L}}B^{2}}(\alpha -\alpha _{0})\log B\left\{ B^{2J_{L}}%
\frac{B^{2}}{B^{2}-1}\left[ J_{L}-\frac{1}{B^{2}-1}\right] \right\}
+o_{J_{L}}(1)\text{ ,}
\end{equation*}%
\begin{equation}
=(\alpha -\alpha _{0})\log B\left\{ \left[ J_{L}-\frac{1}{B^{2}-1}\right]
\right\} +o_{J_{L}}(1)\text{ ,}  \label{xiaohong2}
\end{equation}%
using (\ref{Jint0}) in Proposition \ref{propsumB}, described in the
Appendix, with $J_{1}=1$ and $s=2$. Now, for $\alpha -\alpha _{0}>-2$:%
\begin{equation*}
\log \left\{ \frac{1}{\sum_{j}N_{j}}\sum_{j=1}^{J_{L}}N_{j}\frac{K_{j}\left(
\alpha _{0}\right) }{K_{j}\left( \alpha \right) }\right\}
\end{equation*}%
\begin{eqnarray*}
&=&\log \left\{ \frac{1}{\sum_{j=1}^{J_{L}}B^{2j}}\sum_{j=1}^{J_{L}}B^{2j}%
\frac{B^{\left( 2-\alpha _{0}\right) j}}{B^{\left( 2-\alpha \right) j}}%
\right\} +\log \left( I\left( B,\alpha _{0},\alpha \right) \right)
+o_{J_{L}}(1) \\
&=&\log \left\{ \frac{B^{2}-1}{B^{2J_{L}}B^{2}}\frac{B^{(2+\alpha -\alpha
_{0})J_{L}}B^{2+\alpha -\alpha _{0}}}{B^{2+\alpha -\alpha _{0}}-1}\right\}
+\log \left( I\left( B,\alpha _{0},\alpha \right) \right) +o_{J_{L}}(1) \\
&=&\log \left\{ \frac{B^{2}-1}{B^{2+(\alpha -\alpha _{0})}-1}B^{\left(
\alpha -\alpha _{0}\right) J_{L}}B^{\alpha -\alpha _{0}}\right\} +\log
\left( I\left( B,\alpha _{0},\alpha \right) \right) +o_{J_{L}}(1)
\end{eqnarray*}%
\begin{equation}
=\log \left\{ \frac{B^{2}-1}{B^{2+(\alpha -\alpha _{0})}-1}\right\} +\left(
\alpha -\alpha _{0}\right) \left\{ J_{L}+1\right\} \log B+\log \left(
I\left( B,\alpha _{0},\alpha \right) \right) +o_{J_{L}}(1)\text{ .}
\label{xiaohong1}
\end{equation}

Hence, combining (\ref{xiaohong1}) and (\ref{xiaohong2}) we obtain%
\begin{equation*}
\log \left\{ \frac{1}{\sum_{j}N_{j}}\sum_{j=1}^{J_{L}}N_{j}\frac{K_{j}\left(
\alpha _{0}\right) }{K_{j}\left( \alpha \right) }\right\} -\frac{1}{%
\sum_{j}N_{j}}\sum_{j=1}^{J_{L}}N_{j}\log \left\{ \frac{K_{j}\left( \alpha
_{0}\right) }{K_{j}\left( \alpha \right) }\right\}
\end{equation*}%
\begin{equation*}
\log \left\{ \frac{B^{2}-1}{B^{2+(\alpha -\alpha _{0})}-1}\right\} +\left(
\alpha -\alpha _{0}\right) \left\{ J_{L}+1\right\} \log B-(\alpha -\alpha
_{0})\left\{ J_{L}-\frac{1}{B^{2}-1}\right\} \log B+o_{J_{L}}(1)
\end{equation*}

\begin{equation*}
=\log \left\{ \frac{B^{2}-1}{B^{2+(\alpha -\alpha _{0})}-1}\right\} +\left(
\alpha -\alpha _{0}\right) \log B+\frac{(\alpha -\alpha _{0})}{B^{2}-1}\log
B+o_{J_{L}}(1)
\end{equation*}%
\begin{equation*}
=\log \left\{ \frac{B^{2}-1}{B^{2+(\alpha -\alpha _{0})}-1}\right\} +\left(
\alpha -\alpha _{0}\right) \left\{ \frac{B^{2}}{B^{2}-1}\right\} \log
B+o_{J_{L}}(1)\text{ .}
\end{equation*}%
Now consider the function%
\begin{equation*}
l(x):=\log \left\{ \frac{B^{2}-1}{B^{2+x}-1}\right\} +x\left\{ \frac{B^{2}}{%
B^{2}-1}\right\} \log B\text{ ;}
\end{equation*}%
it is readily seen that for $x>-2,$ $l(x)$ is a continuous function such that%
\begin{eqnarray*}
l^{\prime }(x) &=&-\frac{B^{2+x}\log B}{B^{2+x}-1}+\frac{B^{2}}{B^{2}-1}\log
B\text{ , } \\
l^{\prime }(0) &=&0\text{ , }l^{\prime }(x)<0\text{ for }x<0\text{ , }%
l^{\prime }(x)>0\text{ for }x>0\text{ ,}
\end{eqnarray*}%
whence $l(0)=0$ is the unique minimum, and $l(x)>0$ for all $x\neq 0.$ The
first part of the proof is hence concluded.

Take now $\alpha -\alpha _{0}<-2$; we have%
\begin{equation*}
\frac{1}{\log B^{2J_{L}}}\left\{ \log \left[ \frac{1}{\sum_{j}N_{j}}%
\sum_{j=1}^{J_{L}}N_{j}\frac{K_{j}\left( \alpha _{0}\right) }{K_{j}\left(
\alpha \right) }\right] -\frac{1}{\sum_{j=1}^{J_{L}}N_{j}}%
\sum_{j=1}^{J_{L}}N_{j}\log \frac{K_{j}\left( \alpha _{0}\right) }{%
K_{j}\left( \alpha \right) }\right\}
\end{equation*}%
\begin{equation*}
=\frac{1}{\log B^{2J_{L}}}\left\{ \log \left[ \sum_{j=1}^{J_{L}}B^{j\left\{
2+\alpha -\alpha _{0}\right\} }\right] -\log B^{2J_{L}}+O_{J_{L}}(1)-\frac{%
\left\{ \alpha -\alpha _{0}\right\} }{\sum_{j=1}^{J_{L}}N_{j}}%
\sum_{j=1}^{J_{L}}N_{j}\log B^{j}\right\}
\end{equation*}%
\begin{equation*}
=\frac{\alpha _{0}-\alpha }{2}-1+o_{J_{L}}(1)\text{ .}
\end{equation*}%
Finally, for $\alpha _{0}-\alpha =2$ we obtain%
\begin{equation*}
\frac{1}{\log J_{L}}\left\{ \log \left[ \frac{1}{\sum_{j}N_{j}}%
\sum_{j=1}^{J_{L}}N_{j}\frac{K_{j}\left( \alpha _{0}\right) }{K_{j}\left(
\alpha \right) }\right] -\frac{1}{\sum_{j=1}^{J_{L}}N_{j}}%
\sum_{j=1}^{J_{L}}N_{j}\log \frac{K_{j}\left( \alpha _{0}\right) }{%
K_{j}\left( \alpha \right) }\right\}
\end{equation*}%
\begin{equation*}
=\frac{1}{\log J_{L}}\left\{ -\log B^{2J_{L}}+\log J_{L}+O_{J_{L}}(1)-\frac{1%
}{\sum_{j=1}^{J_{L}}N_{j}}\sum_{j=1}^{J_{L}}N_{j}\left[ \log
B^{-2j}+O_{J_{L}}(1)\right] \right\}
\end{equation*}%
\begin{equation*}
=\frac{1}{\log J_{L}}\left\{ -\log B^{2J_{L}}+\log J_{L}+O_{J_{L}}(1)+\log
B^{2J_{L}}+O_{J_{L}}(1)\right\} \text{ ,}
\end{equation*}%
whence%
\begin{equation*}
\lim_{J_{L}\rightarrow \infty }\frac{1}{\log J_{L}}\left\{ \log \frac{1}{%
\sum_{j}N_{j}}\sum_{j=1}^{J_{L}}N_{j}\frac{K_{j}\left( \alpha _{0}\right) }{%
K_{j}\left( \alpha \right) }-\frac{1}{\sum_{j=1}^{J_{L}}N_{j}}%
\sum_{j=1}^{J_{L}}N_{j}\log \frac{K_{j}\left( \alpha _{0}\right) }{%
K_{j}\left( \alpha \right) }\right\} =1\text{ ,}
\end{equation*}%
as claimed.
\end{proof}

Now we look at $T_{J_{L}}\left( \alpha ,\alpha _{0}\right) $. From (\ref%
{lemmaG}), we can prove the following:

\begin{lemma}
\label{consistency2} As $J_{L}\rightarrow \infty ,$ we have%
\begin{equation*}
\sup_{\alpha }\left\vert T_{J_{L}}\left( \alpha ,\alpha _{0}\right)
\right\vert =o_{p}(1)\text{ .}
\end{equation*}
\end{lemma}

\begin{proof}
From (\ref{defGest}) and (\ref{defGfun}), we have that%
\begin{equation*}
\frac{\widehat{G}\left( \alpha _{0}\right) }{G\left( \alpha _{0}\right) }=%
\frac{1}{\sum_{j=1}^{J_{L}}N_{j}}\sum_{j=1}^{J_{L}}\frac{\sum_{k}\beta
_{jk}^{2}}{G_{0}K_{j}\left( \alpha _{0}\right) }\text{ .}
\end{equation*}%
From Lemma \ref{lemmaG}, it is immediate to see that%
\begin{equation*}
\mathbb{E}\left( \frac{\widehat{G}\left( \alpha _{0}\right) }{G\left( \alpha
_{0}\right) }-1\right) =0\text{ ,}
\end{equation*}%
and%
\begin{equation*}
Var\left( \frac{\widehat{G}\left( \alpha _{0}\right) }{G\left( \alpha
_{0}\right) }-1\right) =\rho ^{2}(\alpha _{0},B)\frac{B^{2}-1}{B^{2}}%
B^{-2J_{L}}+o\left( B^{-2J_{L}}\right) =O\left( B^{-2J_{L}}\right) \text{ .}
\end{equation*}%
By applying Chebichev's inequality, we have
\begin{equation*}
\frac{\widehat{G}\left( \alpha _{0}\right) }{G\left( \alpha _{0}\right) }%
-1\longrightarrow _{p}0\text{ ,}
\end{equation*}%
and from Slutzky's lemma:
\begin{equation*}
\log \left( \frac{\widehat{G}\left( \alpha _{0}\right) }{G\left( \alpha
_{0}\right) }\right) \longrightarrow _{p}0\text{ .}
\end{equation*}%
On the other hand, in view of Lemma \ref{lemmagiovneed},
\begin{equation*}
\sup \left\vert \frac{\widehat{G}\left( \alpha \right) }{G\left( \alpha
\right) }-1\right\vert \longrightarrow _{p}0\text{ ,}
\end{equation*}%
so the proof is complete.
\end{proof}

The second main result to be achieved is a Central Limit Theorem for the
estimator $\widehat{\alpha }_{J_{L}}$, which will be investigated by
exploiting some classical argument on asymptotic Gaussianity for extremum
estimates, as recalled for instance by \cite{NmcF}, see also \cite{dlm}. We
shall in fact establish the following

\begin{theorem}
\label{clt0} Let $\widehat{\alpha }_{J_{L}}=$ $\arg \min_{\alpha \in A}$ $%
\emph{R}_{L}(\alpha )$.

a) Under Condition \ref{REGULNEED} we have:%
\begin{equation}
B^{J_{L}}\text{ }\left( \widehat{\alpha }_{J_{L}}-\alpha _{0}\right)
=O_{p}\left( 1\right) \text{ ;}  \label{clt1}
\end{equation}

b) Under Condition \ref{REGULNEED2} we have:
\begin{equation}
\left( \widehat{\alpha }_{J_{L}}-\alpha _{0}\right) \longrightarrow _{p}m%
\text{ ,}  \label{clt2}
\end{equation}%
where%
\begin{equation*}
m=\kappa I\left( B,\alpha _{0}+1,\alpha _{0}\right) \frac{\log B}{\left(
B+1\right) }\text{ ;}
\end{equation*}

c) Under Condition \ref{REGULNEED3} we have:
\begin{equation}
B^{J_{L}}\left( \widehat{\alpha }_{J_{L}}-\alpha _{0}\right) \overset{d}{%
\longrightarrow }\mathcal{N}\left( 0,D\right) \text{ ,}  \label{clt3}
\end{equation}%
where%
\begin{equation}
D=D\left( \alpha _{0},B\right) =\rho ^{2}\left( \alpha _{0};B\right) \Psi
\left( B\right) \text{ , }\Psi \left( B\right) =\frac{\left( B^{2}-1\right)
^{3}}{B^{4}\log ^{2}B}\text{ .}  \label{neve}
\end{equation}
\end{theorem}

\begin{proof}
By a standard Mean Value Theorem argument and consistency, for each $J_{L}$
there exists $\overline{\alpha }_{J_{L}}\in \left( \alpha _{0}-\widehat{%
\alpha },\alpha _{0}+\widehat{\alpha }\right) $ such that, with probability
one:%
\begin{equation*}
\left( \widehat{\alpha }_{J_{L}}-\alpha _{0}\right) =-\frac{S_{J_{L}}(\alpha
_{0})}{Q_{J_{L}}(\overline{\alpha }_{L})}\text{ ,}
\end{equation*}%
where $S_{J_{L}}(\alpha )$ is the score function corresponding to $%
R_{J_{L}}\left( \alpha \right) $, given by:

\begin{eqnarray*}
&=&\frac{1}{\sum_{j}N_{j}}\sum_{j}\sum_{k}\frac{\beta _{jk}^{2}}{\widehat{G}%
\left( \alpha \right) K_{j}\left( \alpha \right) }\left( -\frac{%
K_{j,1}\left( \alpha \right) }{K_{j}\left( \alpha \right) }\right) +\frac{1}{%
\sum_{j}N_{j}}\sum_{j}\frac{K_{j,1}\left( \alpha \right) }{K_{j}\left(
\alpha \right) }N_{j} \\
&=&\frac{1}{\sum_{j}N_{j}}\frac{G\left( \alpha \right) }{\widehat{G}\left(
\alpha \right) }\sum_{j}\left( -\frac{K_{j,1}\left( \alpha \right) }{%
K_{j}\left( \alpha \right) }\right) \sum_{k}\left( \frac{G\left( \alpha
\right) }{\widehat{G}\left( \alpha \right) }\frac{\beta _{jk}^{2}}{G\left(
\alpha \right) K_{j}\left( \alpha \right) }-\frac{\widehat{G}\left( \alpha
\right) }{G\left( \alpha \right) }\right) \text{ ,}\,
\end{eqnarray*}%
and%
\begin{equation*}
Q_{J_{L}}\left( \alpha \right) =\frac{d^{2}}{d\alpha ^{2}}R_{J_{L}}\left(
\alpha \right) \text{ ,}
\end{equation*}%
i.e.%
\begin{equation*}
Q_{J_{L}}\left( \alpha \right) =\frac{G_{2}\left( \alpha \right) G\left(
\alpha \right) -\left( G_{1}\left( \alpha \right) \right) ^{2}}{\left(
G\left( \alpha \right) \right) ^{2}}+\frac{1}{\sum_{j}N_{j}}\sum_{j}N_{j}%
\frac{K_{j,2}\left( \alpha \right) K_{j}\left( \alpha \right) -\left(
K_{j,1}\left( \alpha \right) \right) ^{2}}{\left( K_{j}\left( \alpha \right)
\right) ^{2}}
\end{equation*}%
\begin{eqnarray*}
&=&\frac{\left( \sum_{j}N_{j}\frac{K_{j}\left( \alpha _{0}\right) }{%
K_{j}\left( \alpha \right) }\left( 2\left( \frac{K_{j,1}\left( \alpha
\right) }{K_{j}\left( \alpha \right) }\right) ^{2}-\frac{K_{j,2}\left(
\alpha \right) }{K_{j}\left( \alpha \right) }\right) \right) \left(
\sum_{j}N_{j}\frac{K_{j}\left( \alpha _{0}\right) }{K_{j}\left( \alpha
\right) }\right) }{\left( \sum_{j}N_{j}\frac{K_{j}\left( \alpha _{0}\right)
}{K_{j}\left( \alpha \right) }\right) ^{2}} \\
&&-\frac{\left( \sum_{j}N_{j}\frac{K_{j}\left( \alpha _{0}\right) }{%
K_{j}\left( \alpha \right) }\left( -\frac{K_{j,1}\left( \alpha \right) }{%
K_{j}\left( \alpha \right) }\right) \right) ^{2}}{\left( \sum_{j}N_{j}\frac{%
K_{j}\left( \alpha _{0}\right) }{K_{j}\left( \alpha \right) }\right) ^{2}}+%
\frac{1}{\sum_{j}N_{j}}\sum_{j}N_{j}\frac{K_{j,2}\left( \alpha \right)
K_{j}\left( \alpha \right) -\left( K_{j,1}\left( \alpha \right) \right) ^{2}%
}{\left( K_{j}\left( \alpha \right) \right) ^{2}}\text{ .}
\end{eqnarray*}%
where $\widehat{G}(\alpha )$, $\widehat{G}_{1}(\alpha )$, $\widehat{G}%
_{2}(\alpha )$ are respectively the estimate of $G$ and its first and second
derivatives, as in Lemma \ref{lemmagiovneed}. In order to establish the
Central Limit Theorem, we analyze the fourth order cumulants, observing that
this statistics belong to the second order Wiener chaos with respect to a
Gaussian white noise random measure (see \cite{nourdinpeccati}). Let%
\begin{equation*}
B^{J_{L}}S_{J_{L}}\left( \alpha _{0}\right) =\frac{1}{B^{J_{L}}}%
\sum_{j}\left( A_{j}+B_{j}\right) \text{ ,}
\end{equation*}%
where%
\begin{eqnarray}
A_{j} &=&B^{2j}\log B^{j}\left\{ \frac{\sum_{k}\beta _{jk}^{2}}{%
N_{j}G_{0}K_{j}\left( \alpha _{0}\right) }-1\right\} \text{ ,}
\label{Aj_cum} \\
B_{j} &=&B^{2j}\log B^{j}\left\{ \frac{\widehat{G}_{J_{L}}(\alpha _{0})}{%
G_{0}}-1\right\} \text{ .}  \label{Bj_cum}
\end{eqnarray}%
In the Appendix, Lemma \ref{cumulants} shows that:%
\begin{equation*}
\frac{1}{B^{4J_{L}}}cum\left\{
\sum_{l_{1}}(A_{j_{1}}+B_{j_{1}}),\sum_{l_{2}}(A_{j_{2}}+B_{j_{2}}),%
\sum_{l_{3}}(A_{j_{3}}+B_{j_{3}}),\sum_{l_{4}}(A_{j_{4}}+B_{j_{4}})\right\}
\end{equation*}%
\begin{equation*}
=O_{J_{L}}\left( \frac{J_{L}^{4}\log ^{4}B}{B^{2J_{L}}}\right) \text{ .}
\end{equation*}

Central Limit Theorem follows therefore from results in \cite{nourdinpeccati}%
. The proofs of (\ref{clt2}) and (\ref{clt3}) are completed by combining the
following Lemmas \ref{lemmaS0} and \ref{LemmaS1}. Observe that under
Condition \ref{REGULNEED2}, the only difference between (\ref{clt1}) and (%
\ref{clt2}) concerns the possibility to estimate analytically the bias term.
\end{proof}

The following result concerns the behaviour of $S_{J_{L}}\left( \alpha
\right) $.

\begin{lemma}
\label{lemmaS0}Under Condition \ref{REGULNEED2}, we have:%
\begin{equation*}
B^{J_{L}}S_{J_{L}}\left( \alpha _{0}\right) \longrightarrow _{p}\kappa
I\left( B,\alpha _{0}+1,\alpha _{0}\right) \frac{\log B}{\left( B+1\right) }%
\kappa \text{ ;}
\end{equation*}%
while under Condition \ref{REGULNEED3} we have:%
\begin{equation*}
B^{J_{L}}S_{J_{L}}\left( \alpha _{0}\right) \overset{d}{\longrightarrow }%
\mathcal{N}\left( 0,\rho _{B}^{2}\left( \alpha _{0}\right) \frac{\log ^{2}B}{%
(B^{2}-1)}\right)
\end{equation*}
\end{lemma}

\begin{proof}
We have that:%
\begin{equation*}
S_{J_{L}}\left( \alpha _{0}\right) =\frac{1}{\sum_{j}N_{j}}\sum_{j}\left( -%
\frac{K_{j,1}\left( \alpha _{0}\right) }{K_{j}\left( \alpha _{0}\right) }%
\right) \sum_{k}\left( \frac{G_{0}}{\widehat{G}\left( \alpha _{0}\right) }%
\frac{\beta _{jk}^{2}}{G_{0}K_{j}\left( \alpha _{0}\right) }-\frac{\widehat{G%
}\left( \alpha \right) }{G_{0}}\right) \text{ ,}
\end{equation*}%
where we recall that for Lemma \ref{lemmagiovneed}:%
\begin{equation*}
\frac{G_{0}}{\widehat{G}\left( \alpha _{0}\right) }\longrightarrow _{p}1%
\text{ .}
\end{equation*}%
Then we will study the behaviour of
\begin{equation*}
\overline{S}_{J_{L}}\left( \alpha _{0}\right) =\frac{1}{\sum_{j}N_{j}}%
\sum_{j}\left( -\frac{K_{j,1}\left( \alpha _{0}\right) }{K_{j}\left( \alpha
_{0}\right) }\right) \sum_{k}\left( \frac{\beta _{jk}^{2}}{G_{0}K_{j}\left(
\alpha _{0}\right) }-\frac{\widehat{G}\left( \alpha _{0}\right) }{G_{0}}%
\right) \text{ .}
\end{equation*}%
Under Condition \ref{REGULNEED2}, simple calculations, in view of (\ref%
{expvaluebetak}), (\ref{K1lim})\ in Corollary \ref{Kjcoroll} and (\ref{Jint1}%
) in Proposition \ref{propsumB}, lead to%
\begin{eqnarray*}
&&\lim_{J_{L}\rightarrow \infty }B^{J_{L}}\mathbb{E}\left( \overline{S}%
_{J_{L}}\left( \alpha _{0}\right) \right)  \\
&=&\lim_{J_{L}\rightarrow \infty }B^{J_{L}}%
\frac{1}{\sum_{j}N_{j}}\sum_{j}\log B^{j}N_{j}\left( \frac{%
\mathbb{E}\left( \sum_{k}\beta _{jk}^{2}\right) }{N_{j}G_{0}K_{j}\left(
\alpha _{0}\right) }-\mathbb{E}\left( \frac{\widehat{G}\left( \alpha
_{0}\right) }{G_{0}}\right) \right)  \\
&=&\lim_{J_{L}\rightarrow \infty }B^{J_{L}}\frac{I_{0}\left( B,\alpha
_{0}\right) }{I_{0}\left( B,\alpha \right) }\frac{\kappa }{%
\sum_{j=J_{0}}^{J_{L}}B^{2j}}\sum_{j=J_{0}}^{J_{L}}\log B^{j}\cdot
B^{2j}\left( B^{-j}-\frac{1}{\sum_{j=J_{0}}^{J_{L}}B^{2j}}%
\sum_{j=J_{0}}^{J_{L}}B^{j}\right) +o_{J_{L}}\left( 1\right)  \\
&=&\lim_{J_{L}\rightarrow \infty }\kappa I\left( B,\alpha _{0}+1,\alpha
_{0}\right) \frac{\log B}{\left( B+1\right) }+o_{J_{L}}\left( 1\right) \text{
,}
\end{eqnarray*}%
while under Condition \ref{REGULNEED3} we have
\begin{equation*}
\lim_{J_{L}\rightarrow \infty }\mathbb{E}\left( \overline{S}_{J_{L}}\left(
\alpha _{0}\right) \right) =0\text{ .}
\end{equation*}%
Moreover we obtain:%
\begin{equation*}
Var\left( \overline{S}_{J_{L}}\left( \alpha _{0}\right) \right)
\end{equation*}%
\begin{eqnarray*}
&=&Var\left( \frac{-1}{\sum_{j}N_{j}}\sum_{j=1}^{J_{L}}\frac{K_{j,1}\left(
\alpha _{0}\right) }{K_{j}\left( \alpha _{0}\right) }\left( \frac{%
\sum_{k}\beta _{jk}^{2}}{G_{0}K_{j}\left( \alpha _{0}\right) }-N_{j}\frac{%
\widehat{G}\left( \alpha _{0}\right) }{G_{0}}\right) \right)  \\
&=&A+B+C\text{ ,}
\end{eqnarray*}%
where%
\begin{eqnarray*}
A &=&\frac{1}{\left( \sum_{j}N_{j}\right) ^{2}}\sum_{j_{1},j_{2}}\left(
\frac{K_{j_{1},1}\left( \alpha _{0}\right) }{K_{j_{1}}\left( \alpha
_{0}\right) }\frac{K_{j_{2},1}\left( \alpha _{0}\right) }{K_{j_{2}}\left(
\alpha _{0}\right) }\right) Cov\left( \frac{\sum_{k_{1}}\left\vert \beta
_{j_{1}k_{1}}\right\vert ^{2}}{G_{0}K_{j_{1}}\left( \alpha _{0}\right) },%
\frac{\sum_{k_{2}}\left\vert \beta _{j_{2}k_{2}}\right\vert ^{2}}{%
G_{0}K_{j_{2}}\left( \alpha _{0}\right) }\right) \text{ ;} \\
B &=&\frac{1}{\left( \sum_{j}N_{j}\right) ^{2}}\sum_{j_{1},j_{2}}\left(
\frac{K_{j_{1},1}\left( \alpha _{0}\right) }{K_{j_{1}}\left( \alpha
_{0}\right) }\frac{K_{j_{2},1}\left( \alpha _{0}\right) }{K_{j_{2}}\left(
\alpha _{0}\right) }\right) N_{j_{1}}N_{j_{2}}Var\left( \frac{\widehat{G}%
\left( \alpha _{0}\right) }{G_{0}}\right) \text{ ;} \\
C &=&\frac{-2}{\left( \sum_{j}N_{j}\right) ^{2}}\sum_{j_{1},j_{2}}\left(
\frac{K_{j_{1},1}\left( \alpha _{0}\right) }{K_{j_{1}}\left( \alpha
_{0}\right) }\frac{K_{j_{2},1}\left( \alpha _{0}\right) }{K_{j_{2}}\left(
\alpha _{0}\right) }\right) Cov\left( \frac{\sum_{k}\left\vert \beta
_{j_{1}k}\right\vert ^{2}}{G_{0}K_{j_{1}}\left( \alpha _{0}\right) }%
,N_{j_{2}}\frac{\widehat{G}\left( \alpha _{0}\right) }{G_{0}}\right) \text{ .%
}
\end{eqnarray*}%
In view of Proposition \ref{propKj} and Lemma \ref{gavarini}, we obtain:%
\begin{eqnarray*}
A &=&\frac{1}{\left( \sum_{j}N_{j}\right) ^{2}}\sum_{j=1}^{J_{L}}\left(
\frac{K_{j,1}\left( \alpha _{0}\right) }{K_{j}\left( \alpha _{0}\right) }%
\right) ^{2}Var\left( \frac{\sum_{k}\beta _{jk}^{2}}{G_{0}K_{j}\left( \alpha
_{0}\right) }\right)  \\
&&+\sum_{j=1}^{J_{L}-1}\left( \frac{K_{j,1}\left( \alpha _{0}\right) }{%
K_{j}\left( \alpha _{0}\right) }\frac{K_{j+1,1}\left( \alpha _{0}\right) }{%
K_{j+1}\left( \alpha _{0}\right) }\right) Cov\left( \frac{\sum_{k}\beta
_{jk}^{2}}{G_{0}K_{j}\left( \alpha _{0}\right) },\frac{\sum_{k}\beta
_{j+1,k}^{2}}{G_{0}K_{j+1}\left( \alpha _{0}\right) }\right)  \\
&&+\sum_{j=1}^{J_{L}}\left( \frac{K_{j,1}\left( \alpha _{0}\right) }{%
K_{j}\left( \alpha _{0}\right) }\frac{K_{j-1}\left( \alpha _{0}\right) }{%
K_{j-1}\left( \alpha _{0}\right) }\right) Cov\left( \frac{\sum_{k}\beta
_{jk}^{2}}{G_{0}K_{j}\left( \alpha _{0}\right) },\frac{\sum_{k}\beta
_{j-1,k}^{2}}{G_{0}K_{j-1}\left( \alpha _{0}\right) }\right)  \\
&=&\frac{1}{\left( \sum_{j}N_{j}\right) ^{2}}\left( \sum_{j}\log ^{2}B^{j}%
\frac{c_{B}^{2}\left( \sigma ^{2}+B^{-\alpha _{0}}\tau ^{2}\right) }{%
I_{0}^{2}\left( B\right) }B^{2j}+o_{J_{L}}\left( 1\right) \right)  \\
&=&\rho ^{2}\left( \alpha _{0},B\right) \frac{1}{\left(
\sum_{j}B^{2j}\right) ^{2}}\sum_{j}B^{2j}\log ^{2}B^{j}+o_{J_{L}}\left(
1\right) \text{ .}
\end{eqnarray*}%
because%
\begin{eqnarray*}
&&\sum_{j=1}^{J_{L}-1}\left( \frac{K_{j,1}\left( \alpha _{0}\right) }{%
K_{j}\left( \alpha _{0}\right) }\frac{K_{j+1,1}\left( \alpha _{0}\right) }{%
K_{j+1}\left( \alpha _{0}\right) }\right) Cov\left( \frac{\sum_{k}\beta
_{jk}^{2}}{G_{0}K_{j}\left( \alpha _{0}\right) },\frac{\sum_{k}\beta
_{j+1,k}^{2}}{G_{0}K_{j+1}\left( \alpha _{0}\right) }\right)  \\
&=&\sum_{j=1}^{J_{L}-1}\left( \left( \frac{K_{j,1}\left( \alpha _{0}\right)
}{K_{j}\left( \alpha _{0}\right) }\frac{K_{j+1,1}\left( \alpha _{0}\right) }{%
K_{j+1}\left( \alpha _{0}\right) }\right) \frac{\tau _{+}^{2}}{%
I_{0}^{2}\left( \alpha _{0},B\right) B^{\alpha _{0}}}\frac{B^{\left(
2-2\alpha _{0}\right) \left( j+1\right) }}{B^{-2\alpha _{0}\left( j+1\right)
}}+o_{j}\left( 1\right) \right)  \\
&=&\text{ }\sum_{j=1}^{J_{L}}\left( \left( \frac{K_{j,1}\left( \alpha
_{0}\right) }{K_{j}\left( \alpha _{0}\right) }\right) ^{2}\frac{B^{-\alpha
_{0}}\tau _{+}^{2}}{I_{0}^{2}\left( \alpha _{0},B\right) }+o_{j}\left(
1\right) \right)
\end{eqnarray*}%
and, likewise,
\begin{eqnarray*}
&&\sum_{j=1}^{J_{L}}\left( \frac{K_{j,1}\left( \alpha _{0}\right) }{%
K_{j}\left( \alpha _{0}\right) }\frac{K_{j-1}\left( \alpha _{0}\right) }{%
K_{j-1}\left( \alpha _{0}\right) }\right) Cov\left( \frac{\sum_{k}\beta
_{jk}^{2}}{G_{0}K_{j}\left( \alpha _{0}\right) },\frac{\sum_{k}\beta
_{j-1,k}^{2}}{G_{0}K_{j-1}\left( \alpha _{0}\right) }\right)  \\
&=&\sum_{j=1}^{J_{L}}\left( \left( \frac{K_{j,1}\left( \alpha _{0}\right) }{%
K_{j}\left( \alpha _{0}\right) }\right) ^{2}\frac{B^{-\alpha _{0}}\tau
_{-}^{2}}{I_{0}^{2}\left( \alpha _{0},B\right) }+o_{j}\left( 1\right)
\right) \text{ .}
\end{eqnarray*}%
On the other hand, by Lemma \ref{lemmaG}, we have:%
\begin{eqnarray*}
B &=&\frac{1}{\left( \sum_{j}B^{2j}\right) ^{4}}\rho ^{2}\left( \alpha
_{0},B\right) \sum_{j_{1}=1}^{J_{L}}B^{2j_{1}}\log
B^{j_{1}}\sum_{j_{2}=1}^{J_{L}}B^{2j_{2}}\log
B^{j_{2}}\sum_{j_{3}=1}^{J_{L}}B^{2j_{3}}+o_{J_{L}}\left( 1\right)  \\
&=&\rho ^{2}\left( \alpha _{0},B\right) \frac{\left( \sum_{j}B^{2j}\log
B^{j}\right) ^{2}}{\left( \sum_{j}B^{2j}\right) ^{3}}\text{ }%
+o_{J_{L}}\left( 1\right) \text{.}
\end{eqnarray*}%
Finally we have:%
\begin{eqnarray*}
C &=&\frac{-2}{\left( \sum_{j}N_{j}\right) ^{3}}\sum_{j_{1},j_{2}}\left(
\frac{K_{j_{1},1}\left( \alpha _{0}\right) }{K_{j_{1}}\left( \alpha
_{0}\right) }\frac{K_{j_{2},1}\left( \alpha _{0}\right) }{K_{j_{2}}\left(
\alpha _{0}\right) }\right) N_{j_{2}}\sum_{j_{3}}^{J_{L}}Cov\left( \frac{%
\sum_{k}\beta _{j_{1}k}^{2}}{G\left( \alpha _{0}\right) K_{j_{1}}\left(
\alpha _{0}\right) },\frac{\sum_{k}\left\vert \beta _{j_{3}k}\right\vert ^{2}%
}{G_{0}K_{j_{3}}\left( \alpha _{0}\right) }\right)  \\
&=&\frac{-2}{\left( \sum_{j}N_{j}\right) ^{3}}\sum_{j_{1},j_{2}}\left( \frac{%
K_{j_{1},1}\left( \alpha _{0}\right) }{K_{j_{1}}\left( \alpha _{0}\right) }%
\frac{K_{j_{2},1}\left( \alpha _{0}\right) }{K_{j_{2}}\left( \alpha
_{0}\right) }\right) N_{j_{2}}\frac{c_{B}^{2}\left( \sigma ^{2}+B^{-\alpha
_{0}}\tau ^{2}\right) }{I_{0}^{2}\left( B\right) }B^{2j_{1}}+o_{J_{L}}\left(
1\right)  \\
&=&-2\rho ^{2}\left( \alpha _{0},B\right) \frac{\left( \sum_{j}B^{2j}\log
B^{j}\right) ^{2}}{\left( \sum_{j}B^{2j}\right) ^{3}}\text{ }%
+o_{J_{L}}\left( 1\right) \text{.}
\end{eqnarray*}%
Hence, following Corollary \ref{sumcorollary} in the Appendix and equation (%
\ref{Jint0}), we have:
\begin{eqnarray}
Var\left( \overline{S}_{J_{L}}\left( \alpha _{0}\right) \right)  &=&\rho
^{2}\left( \alpha _{0},B\right) \frac{1}{\left( \sum_{j}B^{2j}\right) ^{3}}
\notag \\
&&\times \left[ \sum_{j}B^{2j}\sum_{j}B^{2j}\log ^{2}B^{j}-\left(
\sum_{j}B^{2j}\log B^{j}\right) ^{2}+o\left( B^{-J_{L}}\right) \right]
\notag \\
&=&\rho ^{2}\left( \alpha _{0},B\right) \frac{1}{\left(
\sum_{j}B^{2j}\right) ^{3}}Z_{J_{L}}\left( 2\right) +o\left(
B^{-2J_{L}}\right)   \notag \\
&=&\rho ^{2}\left( \alpha _{0},B\right) \frac{\left( B^{2}-1\right) ^{3}}{%
B^{6J_{L}}B^{6}}B^{4J_{L}}\log ^{2}B\frac{B^{6}}{(B^{2}-1)^{4}}+o\left(
B^{-2J_{L}}\right)   \notag \\
&=&\rho ^{2}\left( \alpha _{0},B\right) \frac{\log ^{2}B}{(B^{2}-1)}%
B^{-2J_{L}}\text{ }+o\left( B^{-2J_{L}}\right) \text{.}  \label{VarSJ1}
\end{eqnarray}%
Finally we have:%
\begin{equation*}
\lim_{J_{L}\rightarrow \infty }Var\left( B^{J_{L}}S_{J_{L}}\left( \alpha
_{0}\right) \right) =\rho ^{2}\left( \alpha _{0},B\right) \frac{\log ^{2}B}{%
(B^{2}-1)}\text{ ;}
\end{equation*}%
as claimed.
\end{proof}

The following Lemma regards instead the behaviour of $Q_{J_{L}}(\alpha )$.

\begin{lemma}
\label{LemmaS1}Under Condition \ref{REGULNEED}, we have:%
\begin{equation*}
Q_{J_{L}}\left( \overline{\alpha }_{L}\right) \longrightarrow _{p}\frac{%
B^{2}\log ^{2}B}{\left( B^{2}-1\right) ^{2}}\text{ .}
\end{equation*}
\end{lemma}

\begin{proof}
By using \ref{lemmagiovneed}, we obtain:

\begin{equation*}
Q_{J_{L}}\left( \alpha \right) =\frac{G_{2}\left( \alpha \right) G\left(
\alpha \right) -\left( G_{1}\left( \alpha \right) \right) ^{2}}{\left(
G\left( \alpha \right) \right) ^{2}}+\frac{1}{\sum_{j}N_{j}}\sum_{j}N_{j}%
\frac{K_{j,2}\left( \alpha \right) K_{j}\left( \alpha \right) -\left(
K_{j,1}\left( \alpha \right) \right) ^{2}}{\left( K_{j}\left( \alpha \right)
\right) ^{2}}
\end{equation*}%
\begin{equation*}
=\frac{\left( \sum_{j}N_{j}\frac{K_{j}\left( \alpha _{0}\right) }{%
K_{j}\left( \alpha \right) }\left( 2\left( \frac{K_{j,1}\left( \alpha
\right) }{K_{j}\left( \alpha \right) }\right) ^{2}-\frac{K_{j,2}\left(
\alpha \right) }{K_{j}\left( \alpha \right) }\right) \right) \left(
\sum_{j}N_{j}\frac{K_{j}\left( \alpha _{0}\right) }{K_{j}\left( \alpha
\right) }\right) }{\left( \sum_{j}N_{j}\frac{K_{j}\left( \alpha _{0}\right)
}{K_{j}\left( \alpha \right) }\right) ^{2}}
\end{equation*}%
\begin{equation*}
-\frac{\left( \sum_{j}N_{j}\frac{K_{j}\left( \alpha _{0}\right) }{%
K_{j}\left( \alpha \right) }\left( -\frac{K_{j,1}\left( \alpha \right) }{%
K_{j}\left( \alpha \right) }\right) \right) ^{2}}{\left( \sum_{j}N_{j}\frac{%
K_{j}\left( \alpha _{0}\right) }{K_{j}\left( \alpha \right) }\right) ^{2}}+%
\frac{1}{\sum_{j}N_{j}}\sum_{j}N_{j}\frac{K_{j,2}\left( \alpha \right)
K_{j}\left( \alpha \right) -\left( K_{j,1}\left( \alpha \right) \right) ^{2}%
}{\left( K_{j}\left( \alpha \right) \right) ^{2}}\text{ .}
\end{equation*}%
$Q_{J_{L}}\left( \alpha \right) $ can be rewritten as the sum of three terms:%
\begin{equation*}
Q_{J_{L}}\left( \alpha \right) =Q_{1}\left( \alpha \right) +Q_{2}\left(
\alpha \right) +Q_{3}\left( \alpha \right) \text{ ,}
\end{equation*}%
where:%
\begin{equation*}
Q_{1}\left( \alpha \right) =\frac{Q_{1}^{num}\left( \alpha \right) }{%
Q_{1}^{den}\left( \alpha \right) }
\end{equation*}%
\begin{equation*}
=\frac{\left( \sum_{j}N_{j}\frac{K_{j}\left( \alpha _{0}\right) }{%
K_{j}\left( \alpha \right) }\left( \frac{K_{j,1}\left( \alpha \right) }{%
K_{j}\left( \alpha \right) }\right) ^{2}\right) \left( \sum_{j}N_{j}\frac{%
K_{j}\left( \alpha _{0}\right) }{K_{j}\left( \alpha \right) }\right) -\left(
\sum_{j}N_{j}\frac{K_{j}\left( \alpha _{0}\right) }{K_{j}\left( \alpha
\right) }\left( -\frac{K_{j,1}\left( \alpha \right) }{K_{j}\left( \alpha
\right) }\right) \right) ^{2}}{\left( \sum_{j}N_{j}\frac{K_{j}\left( \alpha
_{0}\right) }{K_{j}\left( \alpha \right) }\right) ^{2}}\text{ ,}
\end{equation*}%
\begin{equation*}
Q_{2}\left( \alpha \right) =\frac{Q_{2}^{num}\left( \alpha \right) }{%
Q_{2}^{den}\left( \alpha \right) }
\end{equation*}%
\begin{equation*}
=\frac{\left( \sum_{j}N_{j}\right) \left( \sum_{j}N_{j}\frac{K_{j}\left(
\alpha _{0}\right) }{K_{j}\left( \alpha \right) }\left( \frac{K_{j,1}\left(
\alpha \right) }{K_{j}\left( \alpha \right) }\right) ^{2}\right) -\left(
\sum_{j}N_{j}\frac{K_{j}\left( \alpha _{0}\right) }{K_{j}\left( \alpha
\right) }\right) \left( \sum_{j}N_{j}\left( \frac{K_{j,1}\left( \alpha
\right) }{K_{j}\left( \alpha \right) }\right) ^{2}\right) }{\left(
\sum_{j}N_{j}\frac{K_{j}\left( \alpha _{0}\right) }{K_{j}\left( \alpha
\right) }\right) \sum_{j}N_{j}}\text{ ,}
\end{equation*}%
\begin{equation*}
Q_{3}\left( \alpha \right) =\frac{Q_{3}^{num}\left( \alpha \right) }{%
Q_{2}^{den}\left( \alpha \right) }
\end{equation*}%
\begin{equation*}
=\frac{\left( \sum_{j}N_{j}\frac{K_{j}\left( \alpha _{0}\right) }{%
K_{j}\left( \alpha \right) }\right) \left( \sum_{j}N_{j}\frac{K_{j,2}\left(
\alpha \right) }{K_{j}\left( \alpha \right) }\right) -\left(
\sum_{j}N_{j}\right) \left( \sum_{j}N_{j}\frac{K_{j}\left( \alpha
_{0}\right) }{K_{j}\left( \alpha \right) }\frac{K_{j,2}\left( \alpha \right)
}{K_{j}\left( \alpha \right) }\right) }{\left( \sum_{j}N_{j}\frac{%
K_{j}\left( \alpha _{0}\right) }{K_{j}\left( \alpha \right) }\right)
\sum_{j}N_{j}}\text{ .}
\end{equation*}

The next step consists in showing that:
\begin{equation*}
Q_{2}\left( \alpha \right) +Q_{3}\left( \alpha \right) =o_{J_{L}}(1)\text{ .}
\end{equation*}

Using Corollary \ref{Kjcoroll}, $Q_{2}^{num}\left( \alpha \right) $ can be
written as:%
\begin{equation*}
Q_{2}^{num}\left( \alpha \right)
\end{equation*}%
\begin{eqnarray*}
&=&\left( \sum_{j}N_{j}\right) \left( \sum_{j}N_{j}\frac{K_{j}\left( \alpha
_{0}\right) }{K_{j}\left( \alpha \right) }\left( \log ^{2}B^{j}+2\frac{%
I_{1}\left( B\right) }{I_{0}\left( B\right) }\log B^{j}+\left( \frac{%
I_{1}\left( B\right) }{I_{0}\left( B\right) }\right)
^{2}+o_{J_{L}}(1)\right) \right) \\
&&-\left( \sum_{j}N_{j}\frac{K_{j}\left( \alpha _{0}\right) }{K_{j}\left(
\alpha \right) }\right) \left( \sum_{j}N_{j}\left( \log ^{2}B^{j}+2\frac{%
I_{1}\left( B\right) }{I_{0}\left( B\right) }\log B^{j}+\left( \frac{%
I_{1}\left( B\right) }{I_{0}\left( B\right) }\right)
^{2}+o_{J_{L}}(1)\right) \right) \text{ ,}
\end{eqnarray*}

while $Q_{3}^{num}\left( \alpha \right) \,$becomes:%
\begin{equation*}
Q_{3}^{num}\left( \alpha \right)
\end{equation*}%
\begin{eqnarray*}
&=&\left( \sum_{j}N_{j}\frac{K_{j}\left( \alpha _{0}\right) }{K_{j}\left(
\alpha \right) }\right) \left( \sum_{j}N_{j}\left( \log B^{2j}+2\frac{%
I_{1}\left( B\right) }{I_{0}\left( B\right) }\log B^{j}+\frac{I_{2}\left(
B\right) }{I_{0}\left( B\right) }+o_{J_{L}}(1)\right) \right) \\
&&-\left( \sum_{j}N_{j}\right) \left( \sum_{j}N_{j}\frac{K_{j}\left( \alpha
_{0}\right) }{K_{j}\left( \alpha \right) }\left( \log B^{2j}+2\frac{%
I_{1}\left( B\right) }{I_{0}\left( B\right) }\log B^{j}+\frac{I_{2}\left(
B\right) }{I_{0}\left( B\right) }+o_{J_{L}}(1)\right) \right) \text{ ,}
\end{eqnarray*}%
so that:%
\begin{equation*}
\frac{Q_{2}^{num}\left( \alpha \right) +Q_{3}^{num}\left( \alpha \right) }{%
Q_{2}^{den}\left( \alpha \right) }=o_{J_{L}}(1)\text{ .}
\end{equation*}

It remains to study $Q_{2}^{den}\left( \alpha \right) ;$ by using (\ref%
{propKj}) and (\ref{propsumB}), we have:%
\begin{equation*}
\lim_{J_{L}\rightarrow \infty }\frac{1}{B^{2\left( 2+\frac{\alpha -a_{0}}{2}%
\right) J_{L}}}\left( \sum_{j}N_{j}\frac{K_{j}\left( \alpha _{0}\right) }{%
K_{j}\left( \alpha \right) }\right) \left( \sum_{j}N_{j}\right)
\end{equation*}%
\begin{eqnarray*}
&=&\lim_{J_{L}\rightarrow \infty }\frac{c_{B}^{2}I\left( B,\alpha
_{0},\alpha \right) ^{2}}{B^{2\left( 2+\frac{\alpha -a_{0}}{2}\right) J_{L}}}%
\left( \sum_{j}B^{2j\left( 1+\frac{\alpha -\alpha _{0}}{2}\right) }\right)
\left( \sum_{j}B^{2j}\right) \\
&=&c_{B}^{2}I\left( B,\alpha _{0},\alpha \right) ^{2}\frac{B^{2\left( 1+%
\frac{\alpha -a_{0}}{2}\right) }}{B^{2\left( 1+\frac{\alpha -a_{0}}{2}%
\right) }-1}\frac{B^{2}}{B^{2}-1}>0\text{ .}
\end{eqnarray*}

Finally, we prove that $Q_{1}\left( \overline{\alpha }_{L}\right)
\rightarrow _{p}\frac{B^{2}\log ^{2}B}{\left( B^{2}-1\right) ^{2}}$. Using
Corollary \ref{Kjcoroll}, we write the numerator $Q_{1}^{num}\left( \alpha
\right) $ as:%
\begin{equation*}
Q_{1}^{num}\left( \alpha \right)
\end{equation*}%
\begin{eqnarray*}
&=&\left( \sum_{j}\frac{K_{j}\left( \alpha _{0}\right) }{K_{j}\left( \alpha
\right) }N_{j}\right) \left( \sum_{j}N_{j}\frac{K_{j}\left( \alpha
_{0}\right) }{K_{j}\left( \alpha \right) }\left( \log B^{j}+\frac{%
I_{1}\left( B\right) }{I_{0}\left( B\right) }\right) ^{2}\right) \\
&&-\left( \sum_{j}N_{j}\frac{K_{j}\left( \alpha _{0}\right) }{K_{j}\left(
\alpha \right) }\left( \log B^{j}+\frac{I_{1}\left( B\right) }{I_{0}\left(
B\right) }\right) \right) ^{2}
\end{eqnarray*}
\begin{equation*}
=\left( \sum_{j}\frac{K_{j}\left( \alpha _{0}\right) }{K_{j}\left( \alpha
\right) }N_{j}\left( \sum_{j}N_{j}\frac{K_{j}\left( \alpha _{0}\right) }{%
K_{j}\left( \alpha \right) }\log ^{2}B^{j}\right) \right) -\left(
\sum_{j}N_{j}\frac{K_{j}\left( \alpha _{0}\right) }{K_{j}\left( \alpha
\right) }\log B^{j}\right) ^{2}
\end{equation*}

Let $s=2\left( 1+\frac{\alpha -a_{0}}{2}\right) $; by applying (\ref{ZJLlim}%
) we have:%
\begin{eqnarray*}
\lim_{J_{L}\rightarrow \infty }\frac{1}{B^{2sJ_{L}}}Q_{1}^{num}\left( \alpha
\right) &=&\lim_{J_{L}\rightarrow \infty }\frac{c_{B}^{2}I\left( B,\alpha
_{0},\alpha \right) ^{2}}{B^{2sJ_{L}}}Z_{J_{L}}\left( s\right) \\
&=&\log ^{2}B\frac{B^{3s}}{(B^{s}-1)^{4}}c_{B}^{2}I\left( B,\alpha
_{0},\alpha \right) ^{2}\text{ . }
\end{eqnarray*}

It remains to study $Q_{1}^{den}\left( \alpha \right) ;$ by using again (\ref%
{K0ratio}) and (\ref{propsumB}):%
\begin{eqnarray*}
\lim_{J_{L}\rightarrow \infty }\frac{1}{B^{2sJ_{L}}}Q_{1}^{den}\left( \alpha
\right) &=&\lim_{J_{L}\rightarrow \infty }\frac{c_{B}^{2}I\left( B,\alpha
_{0},\alpha \right) ^{2}}{B^{2sJ_{L}}}\left( \sum_{j}B^{sj}\right) ^{2} \\
&=&c_{B}^{2}I\left( B,\alpha _{0},\alpha \right) ^{2}\left( \frac{B^{s}}{%
B^{s}-1}\right) ^{2}\text{ .}
\end{eqnarray*}%
Hence%
\begin{equation*}
\lim_{J_{L}\rightarrow \infty }Q\left( \alpha \right) =\frac{B^{2\left( 1+%
\frac{\alpha -a_{0}}{2}\right) }\log ^{2}B}{\left( B^{2\left( 1+\frac{\alpha
-a_{0}}{2}\right) }-1\right) ^{2}}\text{ .}
\end{equation*}%
For the consistency of $\widehat{\alpha }_{L}$, for $\overline{\alpha }%
_{L}\in \left[ \alpha _{0}-\widehat{\alpha }_{L},\alpha _{0}+\widehat{\alpha
}_{L}\right] $, we have
\begin{equation*}
Q\left( \overline{\alpha }_{L}\right) \longrightarrow _{p}\frac{B^{2}\log
^{2}B}{\left( B^{2}-1\right) ^{2}}\text{ .}
\end{equation*}

\begin{equation*}
\lim_{J_{L}\rightarrow \infty }Var\left( B^{J_{L}}S_{J_{L}}\left( \alpha
_{0}\right) \right) =\rho ^{2}(\alpha _{0},B)\frac{\log ^{2}B}{(B^{2}-1)}%
\text{ .}
\end{equation*}
\end{proof}

To investigate the efficiency of needlet estimates, fix $B^{J_{L}}=L/B,$ so
that the highest frequency covers the multipoles $l=\left[ L/B^{2}\right]
+1,...,L;$ observe that, under Condition \ref{REGULNEED3}%
\begin{equation*}
B^{J_{L}+1}\left( \widehat{\alpha }_{J_{L}}-\alpha _{0}\right) =L\left(
\widehat{\alpha }_{J_{L}}-\alpha _{0}\right) \overset{d}{\longrightarrow }%
\mathcal{N}\left( 0,B^{2}\times D\left( B,\alpha _{0}\right) \right) \text{ ,%
}
\end{equation*}%
while parametric estimates based upon standard Fourier analysis (see \cite%
{dlm}) yield
\begin{equation*}
L\left( \widehat{\alpha }_{L}-\alpha _{0}\right) \overset{d}{\longrightarrow
}\mathcal{N}\left( 0,8\right) \,\text{\ .}
\end{equation*}%
For any given value of $B,$ the asymptotic variance $D\left( B,\alpha
_{0}\right) $ can be evaluated numerically by means of (\ref{neve}) and a
plug-in method, where $\alpha _{0}$ is replaced by its consistent estimate $%
\widehat{\alpha }_{J_{L}}$. In practice, though, this is not really needed
for the values of $B$ which are commonly in use, i.e. $B\simeq 1.1/1.5.$ In
fact, using $\log x\simeq x-1+o(1)$ as $x\rightarrow 1$ we have%
\begin{eqnarray*}
\lim_{B\rightarrow 1}\frac{B^{2}\Psi \left( B\right) }{B-1}
&=&\lim_{B\rightarrow 1}\frac{1}{B-1}\frac{\left( B^{2}-1\right) ^{3}}{%
B^{2}\left( B-1\right) ^{2}} \\
&=&\lim_{B\rightarrow 1}\frac{1}{B-1}\frac{\left( B+1\right) ^{3}\left(
B-1\right) }{B^{2}}=8\text{ .}
\end{eqnarray*}%
A standard choice for the function $b(.)$ (see \cite{bkmpAoS}, \cite%
{marpecbook}) is provided by
\begin{eqnarray*}
b^{2}(x) &=&b^{2}(x;B)=0\text{ , for }x\notin (\frac{1}{B},B)\text{ ,} \\
b^{2}(x) &=&1-\frac{\int_{-1}^{(1-\frac{2B}{B-1}(x-\frac{1}{B}))}\exp (-%
\frac{1}{1-u^{2}})du}{\int_{-1}^{1}\exp (-\frac{1}{1-u^{2}})du}\text{ , for }%
\frac{1}{B}\leq x\leq 1\text{ ,} \\
\text{ }b^{2}(x) &=&\frac{\int_{-1}^{(1-\frac{2B}{B-1}(\frac{x}{B}-\frac{1}{B%
}))}\exp (-\frac{1}{1-u^{2}})du}{\int_{-1}^{1}\exp (-\frac{1}{1-u^{2}})du}%
\text{ , for }1\leq x\leq B\text{ .}
\end{eqnarray*}%
For this choice of $b(.),$ analytical and numerical approximations allow to
show that%
\begin{equation*}
\lim_{B\rightarrow 1}(B-1)\rho ^{2}\left( \alpha ;B\right) =1\text{ ,}
\end{equation*}%
whence
\begin{equation*}
\lim_{B\rightarrow 1}D\left( B,\alpha _{0}\right) =\lim_{B\rightarrow
1}\left( \frac{\left( B+1\right) ^{2}+2B\left( B+1\right) }{B^{3}}%
+o_{B}\left( 1\right) \right) =8\text{ .}
\end{equation*}%
Summing up, the variance of the needlet likelihood estimator is very close
to the "optimal" value (e.g. 8) which was found by \cite{dlm} for the
Fourier-based method. Some numerical results to validate this claim are
provided in Table 1 for a range of values of $B$ and $\alpha _{0}$. These
numerical results are confirmed with remarkable accuracy by the Monte Carlo
evidence reported in Section \ref{numerical} below.
\begin{equation*}
\begin{tabular}{|c|c|c|c|c|c|c|c|c|c|c|c|c|}
\hline
$B$ & \multicolumn{3}{|c|}{$\sqrt[8]{2}$} & \multicolumn{3}{|c|}{$\sqrt[4]{2}
$} & \multicolumn{3}{|c|}{$\sqrt[2]{2}$} & \multicolumn{3}{|c|}{$2$} \\
\hline
$\alpha _{0}$ & 2 & 3 & 4 & 2 & 3 & 4 & 2 & 3 & 4 & 2 & 3 & 4 \\ \hline
$\sigma ^{2}$ & 0.27 & 0.27 & 0.27 & 0.53 & 054 & 0.54 & 1.15 & 1.16 & 1.16
& 2.09 & 2.10 & 2.10 \\ \hline
$\tau ^{2}$ & 0.04 & 0.05 & 0.05 & 0.13 & 0.13 & 0.13 & 0.44 & 0.44 & 0.44 &
0.58 & 0.58 & 0.58 \\ \hline
$I_{0}$ & 0.17 & 0.17 & 0.17 & 0.35 & 0.35 & 0.35 & 0.70 & 0.70 & 0.70 & 1.39
& 1.39 & 1.39 \\ \hline
$\rho ^{2}$ & 8.46 & 8.48 & 8.50 & 5.00 & 5.04 & 5.09 & 2.58 & 2.61 & 2.63 &
1.40 & 1.41 & 1.43 \\ \hline
$\Psi $ & \multicolumn{3}{|c|}{0.75} & \multicolumn{3}{|c|}{1.18} &
\multicolumn{3}{|c|}{2.08} & \multicolumn{3}{|c|}{3.51} \\ \hline
$B^{2}D$ & 7.67 & 7.69 & 7.70 & 8.36 & 8.43 & 8.52 & 10.7 & 10.8 & 10.9 &
20.6 & 20.8 & 20.9 \\ \hline
\end{tabular}%
\end{equation*}

\begin{center}
Table 1: Some deterministic results for different values of $B$ and $\alpha
_{0}$.
\end{center}

\begin{remark}
It is shown in \cite{bkmpAoS} how needlet coefficients are asymptotically
unaffected by the presence of masked or unobserved regions, provide they are
centred outside the mask. It is then possible to argue that the asymptotic
results presented here remain unaltered in case of a partially observed
sphere, up to a normalization factor representing the so-called sky
fraction, i.e. the effective number of available observations. This is a
major advantage when compare to standard Fourier analysis techniques - in
the latter case, asymptotic theory can no longer be entertained in the case
of partial observations. Again, for brevity's sake we do not develop a
formal argument here; proofs, however, can be routinely performed starting
from the inequality%
\begin{equation*}
\frac{\mathbb{E}\left\{ \beta _{jk}-\beta _{jk}^{\ast }\right\} ^{2}}{%
\mathbb{E}\beta _{jk}^{2}}\leq \frac{C_{M}}{\left\{ 1+B^{j}d(\xi
_{jk},G)\right\} ^{M}}\text{ ,}
\end{equation*}%
valid for every integer $M>0,$ some constant $C_{M}>0,$ for $G$ denoting the
unobserved region, see again \cite{bkmpAoS},\cite{bkmpBer}.
\end{remark}

\section{Narrow-band estimates \label{narrow}}

As discussed in the previous Section, under Condition \ref{REGULNEED2},
asymptotic inference is made impossible by the presence of the nuisance
parameter $m.$ It is possible to get rid of this parameter, however, by
considering narrow-band estimates focussing only on the higher tail of the
power spectrum. The details are similar to the approach pursued in analogous
circumstances in \cite{dlm}. We start from the following

\begin{definition}
The \emph{Narrow-Band Needlet Whittle estimator }for the parameters $%
\vartheta =(\alpha ,G)$ is provided by%
\begin{equation*}
(\widehat{\alpha }_{J_{L};J_{1}},\widehat{G}_{J_{L};J_{1}}):=\arg
\min_{\alpha ,G}\sum_{j=J_{1}}^{J_{L}}\left[ \frac{\sum_{k}\beta _{jk}^{2}}{%
GK_{j}\left( \alpha \right) }-\sum_{k=1}^{N_{j}}\log \left( \frac{\beta
_{jk}^{2}}{GK_{j}\left( \alpha \right) }\right) \right] \text{ ,}
\end{equation*}%
or equivalently:%
\begin{eqnarray}
\widehat{\alpha }_{J_{L};J_{1}} &=&\arg \min_{\alpha }R_{J_{L};J_{1}}(\alpha
,\widehat{G}(\alpha )),  \label{narrowest} \\
R_{J_{L};J_{1}}(\alpha ) &=&\left( \log \widehat{G}_{J_{L};J_{1}}(\alpha )+%
\frac{1}{\sum_{j=J_{1}}^{J_{L}}N_{j}}\sum_{J_{1}=J_{1}}^{J_{L}}N_{j}\log
K_{j}\left( \alpha \right) \right) \text{ ,}  \notag
\end{eqnarray}%
where $J_{1}<J_{L}$ is chosen such that $B^{J_{L}}-B^{J_{1}}\rightarrow
\infty $ and%
\begin{equation}
B^{J_{1}}=B^{J_{L}}\left( 1-g\left( J_{L}\right) \right) \text{ , }%
J_{1}=J_{L}+\frac{\log \left( 1-g\left( J_{L}\right) \right) }{\log B}\text{
.}  \label{BJ1}
\end{equation}%
We choose $0<g\left( J_{L}\right) <1$ s.t. $\lim_{J_{L}\rightarrow \infty
}g\left( J_{L}\right) =0$ and $\lim_{J_{L}\rightarrow \infty
}J_{L}^{2}g^{3}\left( J_{L}\right) =0$ .
\end{definition}

For notational simplicity $B^{J_{1}}$ is defined as an integer (if this
isn't the case, modified arguments taking integer parts are completely
trivial). For definiteness, we can take for instance $g\left( J_{L}\right)
=J_{L}^{-3}$ .

\begin{theorem}
\label{theonarrowband}Let $\widehat{\alpha }_{J_{L};J_{1}}$ defined as in (%
\ref{narrowest}). Then under Condition \ref{REGULNEED2} we have
\end{theorem}

\begin{equation*}
g\left( J_{L}\right) ^{\frac{1}{2}}B^{J_{L}}\left( \widehat{\alpha }%
_{L;L_{1}}-\alpha _{0}\right) \overset{d}{\longrightarrow }\mathcal{N}\left(
0,\frac{\rho ^{2}(\alpha _{0},B)}{\Phi \left( B\right) }\right) \text{ ,}
\end{equation*}%
where%
\begin{equation*}
\Phi \left( B\right) =\log ^{2}B\frac{B^{2}}{\left( B^{2}-1\right) ^{2}}%
\left( \frac{4}{\left( B^{2}-1\right) }+2\left( \frac{\log B-1}{\log B}%
\right) \right) \text{ .}
\end{equation*}

\begin{proof}
Because the proof is very similar to the full band case, we put in evidence
here just the main differences. Consider:%
\begin{eqnarray*}
S_{J_{L};J_{1}}\left( \alpha \right)  &=&\frac{d}{d\alpha }%
R_{J_{L};J_{1}}\left( \alpha \right) \text{ ;} \\
Q_{J_{L};J_{1}}\left( \alpha \right)  &=&\frac{d^{2}}{d\alpha ^{2}}%
R_{J_{L};J_{1}}\left( \alpha \right) \text{ .}
\end{eqnarray*}%
Let
\begin{eqnarray*}
\overline{S}_{J_{L};J_{1}}\left( \alpha _{0}\right)
&=&S_{J_{L};J_{1}}\left( \alpha _{0}\right) \frac{G_{0}}{\widehat{G}\left(
\alpha _{0}\right) } \\
&=&\frac{-1}{\sum_{j=J_{1}}^{J_{L}}N_{j}}\sum_{j=J_{1}}^{J_{L}}\frac{%
K_{j,1}\left( \alpha _{0}\right) }{K_{j}\left( \alpha _{0}\right) }%
\sum_{k=1}^{N_{j}}\left( \frac{\beta _{jk}^{2}}{G\left( \alpha _{0}\right)
K_{j}\left( \alpha _{0}\right) }-\frac{\widehat{G}\left( \alpha _{0}\right)
}{G\left( \alpha _{0}\right) }\right) \text{ }.
\end{eqnarray*}%
We have:%
\begin{equation*}
\lim_{J_{L}\rightarrow \infty }\frac{B^{J_{L}}}{J_{L}g\left( J_{L}\right) }%
\mathbb{E}\left( \overline{S}_{J_{0},J_{L}}^{M}\left( \alpha _{0}\right)
\right)
\end{equation*}%
\begin{eqnarray*}
&=&\lim_{J_{L}\rightarrow \infty }\frac{B^{J_{L}}}{J_{L}g\left( J_{L}\right)
}I\left( B,\alpha _{0}+1,\alpha _{0}\right) \frac{\kappa }{%
\sum_{j=J_{1}}^{J_{L}}B^{2j}}\sum_{j=J_{1}}^{J_{L}}\log B^{j}\cdot
B^{2j}\left( B^{-j}-\frac{1}{\sum_{j=J_{1}}^{J_{L}}B^{2j}}%
\sum_{j=J_{1}}^{J_{L}}B^{j}\right) +o_{J_{L}}\left( 1\right)  \\
&=&\lim_{J_{L}\rightarrow \infty }\frac{B^{J_{L}}}{J_{L}g\left( J_{L}\right)
}\kappa I\left( B,\alpha _{0}+1,\alpha _{0}\right)
\sum_{j=J_{1}}^{J_{L}}\log B^{j}\cdot B^{2j}\left( B^{-j}-B^{-J_{L}}\left(
\frac{B-1}{B}+\frac{g\left( J_{L}\right) }{B}\right) \right)
+o_{J_{L}}\left( 1\right)  \\
&=&\lim_{J_{L}\rightarrow \infty }\frac{B^{-J_{L}}}{g\left( J_{L}\right) }%
\kappa I\left( B,\alpha _{0}+1,\alpha _{0}\right) \frac{B^{J_{L}}B\log B}{B-1%
}\left( J_{L}\left[ \left( \frac{B-1}{B}-\frac{B}{B+1}+\frac{1}{B\left(
B+1\right) }\right) +g\left( J_{L}\right) \left( \frac{1}{B}-\frac{2}{%
B\left( B+1\right) }\right) \right] \right)  \\
&=&\lim_{J_{L}\rightarrow \infty }\kappa I\left( B,\alpha _{0}+1,\alpha
_{0}\right) \frac{\log B}{B+1}+o_{J_{L}}\left( 1\right)
\end{eqnarray*}

Following (\ref{ZJLJ1}) and (\ref{VarSJ1}), we have
\begin{equation*}
Var\left( \overline{S}_{J_{L};J_{1}}\left( \alpha _{0}\right) \right) =\rho
^{2}(\alpha _{0},B)\frac{Z_{J_{L};J_{1}}\left( 2\right) }{\left(
\sum_{j=J_{1}}^{J_{L}}B^{2j}\right) ^{3}}\text{ .}
\end{equation*}

For (\ref{BJ1}), we have:%
\begin{equation*}
\frac{1}{B^{4J_{L}}}Z_{J_{L};J_{1}}\left( 2\right)
\end{equation*}%
\begin{eqnarray}
&=&\left( 1-\frac{\left( 1-g\left( J_{L}\right) \right) ^{2}}{B^{2}}\right)
^{2}-\frac{\left( B^{2}-1\right) ^{2}}{B^{4}}\left( 1-g\left( J_{L}\right)
\right) ^{2}\left( 1-\log _{B}\left( 1-g\left( J_{L}\right) \right) \right)
^{2}  \notag \\
&=&\left( \frac{B^{2}-1+2g\left( J_{L}\right) }{B^{2}}\right) ^{2}-\left(
\frac{B^{2}-1}{B^{2}}\right) ^{2}\left( 1-2g\left( J_{L}\right) \right)
\notag \\
&&\times \left( 1-\log _{B}\left( 1-g\left( J_{L}\right) \right) \right)
^{2}+O\left( g\left( J_{L}\right) ^{2}\right)  \notag \\
&=&\left( \frac{B^{2}-1}{B^{2}}\right) ^{2}\Delta Z_{J_{L};J_{1}}\left(
g\left( J_{L}\right) \right) +O\left( g\left( J_{L}\right) ^{2}\right)
\notag \\
&=&\left( \frac{B^{2}-1}{B^{2}}\right) ^{2}\left( \frac{4}{\left(
B^{2}-1\right) }+\left( 2-\frac{2}{\log B}\right) \right) g\left(
J_{L}\right) +O\left( B^{4J_{L}}g\left( J_{L}\right) ^{2}\right) \text{ ,}
\label{Zconti}
\end{eqnarray}

where%
\begin{eqnarray*}
\Delta Z_{J_{L};J_{1}}\left( g\left( J_{L}\right) \right) &=&\left( 1+\frac{%
4g\left( J_{L}\right) }{\left( B^{2}-1\right) }\right) -\left( 1-2g\left(
J_{L}\right) \right) \left( 1+\frac{1}{\log B}g\left( J_{L}\right) \right)
^{2} \\
&=&\left( 1+\frac{4g\left( J_{L}\right) }{\left( B^{2}-1\right) }\right)
-\left( 1-2g\left( J_{L}\right) \right) \left( 1+\frac{2}{\log B}g\left(
J_{L}\right) \right) \\
&=&\left( 1+\frac{4}{\left( B^{2}-1\right) }g\left( J_{L}\right) \right)
-\left( 1+\left( \frac{2}{\log B}-2\right) g\left( J_{L}\right) \right) \\
&=&\left( \frac{4}{\left( B^{2}-1\right) }+\left( 2-\frac{2}{\log B}\right)
\right) g\left( J_{L}\right)
\end{eqnarray*}

Thus we have%
\begin{equation*}
Z_{J_{L};J_{1}}\left( 2\right) =B^{4J_{L}}\Phi \left( B\right) g\left(
J_{L}\right) \text{ }+O\left( B^{4J_{L}}g\left( J_{L}\right) ^{2}\right)
\text{.}
\end{equation*}%
Note that $\Phi \left( B\right) >0$ for $B>1$.

On the other hand, simple calculations on Proposition \ref{propsumB} lead to%
\begin{eqnarray*}
\left( \sum_{j=J_{1}}^{J_{L}}B^{2j}\right) ^{3} &=&\frac{B^{6}}{\left(
B^{2}-1\right) ^{3}}\left( B^{2J_{L}}-B^{2\left( J_{1}-1\right) }\right)
^{3}+o\left( B^{6J_{L}}\right) \\
&=&\frac{B^{6}}{\left( B^{2}-1\right) ^{3}}B^{6J_{L}}\left( 1-B^{-2}\left(
1-g\left( J_{L}\right) \right) ^{2}\right) ^{3}+o\left( B^{6J_{L}}\right) \\
&=&\frac{B^{6}}{\left( B^{2}-1\right) ^{3}}B^{6J_{L}}\left( \frac{B^{2}-1}{%
B^{2}}\right) ^{3}+O\left( B^{6J_{L}}g\left( J_{L}\right) \right) \\
&=&B^{6J_{L}}\text{ }+O\left( B^{6J_{L}}g\left( J_{L}\right) \right) \text{ ,%
}
\end{eqnarray*}%
hence we have%
\begin{equation*}
Var\left( \overline{S}_{J_{L};J_{1}}\left( \alpha _{0}\right) \right) =\rho
^{2}(\alpha _{0},B)\Phi \left( B\right) g\left( J_{L}\right) B^{-2J_{L}}%
\text{ .}
\end{equation*}%
Consider now $Q_{J_{L};J_{1}}\left( \alpha \right) $, which we rewrite as%
\begin{equation*}
Q_{J_{L};J_{1}}\left( \alpha \right) =\frac{Q_{num}\left( \alpha \right) }{%
Q_{den}\left( \alpha \right) }\text{ .}
\end{equation*}%
Following a procedure similar to Lemma \ref{LemmaS1}, we have
\begin{equation*}
Q_{num}\left( \alpha \right) =c_{B}^{2}G_{0}^{2}I\left( B,\alpha _{0},\alpha
\right) Z_{J_{L};J_{1}}\left( s\right) \text{ ,}
\end{equation*}%
where $s=2\left( 1+\frac{\alpha -\alpha _{0}}{2}\right) $. Following (\ref%
{Zconti}), we obtain%
\begin{equation*}
Q_{num}\left( \alpha \right) =c_{B}^{2}G_{0}^{2}I\left( B,\alpha _{0},\alpha
\right) \Phi \left( B,s\right) B^{2sJ_{L}}g\left( J_{L}\right) +O\left(
B^{2sJ_{L}}g\left( J_{L}\right) ^{2}\right) \text{ ,}
\end{equation*}%
where%
\begin{equation*}
\Phi \left( B,s\right) =\log ^{2}B\frac{B^{s}}{\left( B^{s}-1\right) ^{2}}%
\left( \frac{2sg\left( J_{L}\right) }{B^{s}-1}+\frac{s\log B-2}{\log B}%
\right) \text{.}
\end{equation*}

Finally, we obtain
\begin{eqnarray*}
Q_{den}\left( \alpha \right)  &=&c_{B}^{2}G_{0}^{2}I\left( B,\alpha
_{0},\alpha \right) \left( \sum_{j=J_{1}}^{J_{L}}B^{sj}\right) ^{2} \\
&=&c_{B}^{2}G_{0}^{2}I\left( B,\alpha _{0},\alpha \right)
B^{2sJ_{L}}+o\left( B^{2sJ_{L}}\right) \text{ .}
\end{eqnarray*}%
Hence
\begin{equation*}
Q_{J_{L};J_{1}}\left( \alpha \right) =\Phi \left( B,s\right) g\left(
J_{L}\right) +O\left( B^{2sJ_{L}}g\left( J_{L}\right) ^{2}\right) \text{ ,}
\end{equation*}%
and, for the consistency of $\alpha ,$ we have
\begin{equation*}
Q_{J_{L};J_{1}}\left( \overline{\alpha }\right) \rightarrow _{p}\Phi \left(
B\right) g\left( J_{L}\right) \text{ .}
\end{equation*}%
Thus%
\begin{equation*}
\left( \frac{\rho ^{2}(\alpha _{0},B)}{\Phi \left( B\right) }\right) ^{-%
\frac{1}{2}}g\left( J_{L}\right) ^{\frac{1}{2}}B^{J_{L}}\frac{\overline{S}%
_{J_{L};J_{1}}\left( \alpha _{0}\right) }{Q_{J_{L};J_{1}}\left( \overline{%
\alpha }\right) }\overset{d}{\rightarrow }\mathcal{N}\left( 0,1\right) \text{
,}
\end{equation*}%
as claimed. Finally we can see that
\begin{equation*}
g\left( J_{L}\right) ^{\frac{1}{2}}B^{J_{L}}\mathbb{E}\frac{\overline{S}%
_{J_{L};J_{1}}\left( \alpha _{0}\right) }{Q_{J_{L};J_{1}}\left( \overline{%
\alpha }\right) }=O\left( J_{L}\cdot g\left( J_{L}\right) ^{\frac{3}{2}%
}\right) \underset{J_{L\rightarrow \infty }}{\rightarrow }0\text{ .}
\end{equation*}
\end{proof}

\begin{remark}
A careful inspection of the proof reveals that the asymptotic result could
be alternatively presented as
\begin{equation*}
B^{-J_{L}}\sqrt{Z_{J_{L};J_{1}}\left( 2\right) }\left( \widehat{\alpha }%
_{L;L_{1}(\delta )}-\alpha _{0}\right) \overset{d}{\longrightarrow }\mathcal{%
N}\left( 0,\rho ^{2}(\alpha _{0},B)\frac{\left( B^{2}-1\right) ^{3}}{B^{6}}%
\right) \text{ ,}
\end{equation*}%
where%
\begin{equation*}
Z_{J_{L};J_{1}}\left( 2\right) =\left( \sum_{j=J_{1}}^{J_{L}}B^{2j}\right)
\left( \sum_{j=J_{1}}^{J_{L}}B^{2j}j^{2}\log ^{2}B\right) -\left(
\sum_{j=J_{1}}^{J_{L}}B^{2j}j\log B\right) ^{2}.
\end{equation*}
\end{remark}

\section{Numerical Results\label{numerical}}

\begin{figure}[tbp]
\centering
\includegraphics[width=\textwidth]{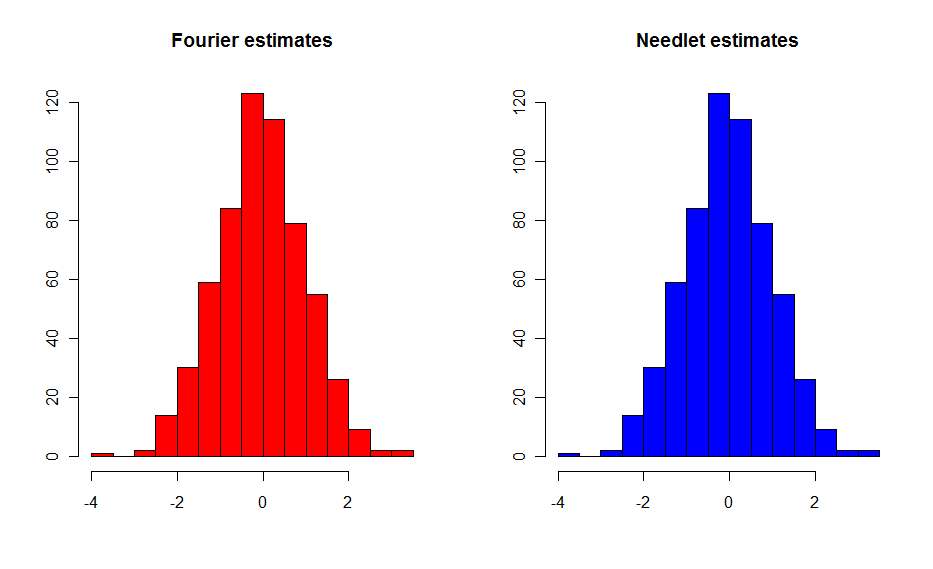}
\caption{Distribution of normalized $\left( \widehat{\protect\alpha}_{L}-%
\protect\alpha _{0}\right)$ and $\left( \widehat{\protect\alpha }_{J_{L}}-%
\protect\alpha _{0}\right) $, $L=2048$, $\protect\alpha _{0}=2$.}
\label{fig:1}
\end{figure}

In this Section we provide some numerical evidence to support the asymptotic
results discussed earlier. More precisely, using the statistical software R,
for given fixed values of $L$, $\alpha _{0}$, $B$ and $G_{0}$ and the
alternative conditions discussed in the previous Section, we sample random
values for the angular power spectra $\widehat{C}_{l}$, and evaluate the
corresponding needlet coefficients $\widehat{\beta}_{jk}$; we implement
standard and narrow-band estimates with both standard Fourier (as described
in \cite{dlm}) and needlet methods. We start by analyzing the simplest
model, i.e. the one corresponding to Condition \ref{REGULNEED3}. Here we
fixed $G_{0}=2$. In Figure 1, the first column reports the distribution of
Fourier estimates of $\widehat{\alpha }_{L}-\alpha _{0}$ normalized by a
factor $\sqrt{2}L/4$, while the second column reports the distribution of $%
\widehat{\alpha }_{J_{L}}-\alpha _{0}$ normalized by the factor $D(\alpha
_{0},B)^{-\frac{1}{2}}B^{J_{L}}$. In Table 2, we report the sample means and
variances for different values of $L$ and $\alpha _{0}$, while in Table 3 we
report the corresponding Shapiro-Wilk test of Gaussianity results. Figure 2
describes graphically the behavior of normalized distributions of estimates
of $\alpha _{0}$ in both classical Fourier and needlet analysis, full band
and narrow band, with $\kappa =1$, under Condition \ref{REGULNEED2}. Table 4
provides sample means and variances for different values of $\alpha _{0}\,$,
with $=\sqrt[8]{2}$, $L=1024$, $L_{1}=724$, and the results of the
corresponding Shapiro-Wilk test of Gaussianity. Overall, we believe this
numerical evidence to be very encouraging; in particular, we stress how the
asymptotic expression reported earlier provide extremely good approximations
for the Monte Carlo estimates of the standard deviation.%
\begin{equation*}
\begin{tabular}{|c|c|c|c|c|c|c|c|}
\cline{3-5}\cline{3-8}
\multicolumn{2}{c}{$B=2$} & \multicolumn{3}{|c}{$\widehat{\alpha }_{L}$} &
\multicolumn{3}{|c|}{$\widehat{\alpha }_{J_{L}}$} \\ \hline
$L$ & $\alpha _{0}$ & mean & sd & $\left( L\cdot sd\right) ^{2}/D$ & mean &
sd & $\left( L\cdot sd\right) ^{2}/D$ \\ \hline
& $2$ & 1.9984 & $1.12\cdot 10^{-2}$ & 1.03 & 1.9981 & $1.68\cdot 10^{-2}$ &
1.12 \\ \cline{2-8}
256 & $3$ & 2.9994 & $1.13\cdot 10^{-2}$ & 1.05 & 2.9997 & $1.84\cdot
10^{-2} $ & 1.07 \\ \cline{2-8}
& $4$ & 4.0009 & $1.10\cdot 10^{-2}$ & 0.99 & 3.9996 & $1.89\cdot 10^{-2}$ &
1.12 \\ \hline
& $2$ & 2.0005 & $5.79\cdot 10^{-3}$ & 1.09 & 2.0002 & $8.55\cdot 10^{-3}$ &
0.93 \\ \cline{2-8}
512 & $3$ & 2.9995 & $5.76\cdot 10^{-3}$ & 1.08 & 2.9999 & $8.50\cdot
10^{-3} $ & 0.91 \\ \cline{2-8}
& 4 & 3.9997 & $5.59\cdot 10^{-3}$ & 1.02 & 3.9997 & $9.35\cdot 10^{-3}$ &
1.09 \\ \hline
& $2$ & 2.0002 & $2.79\cdot 10^{-3}$ & 1.02 & 2.0002 & $4.42\cdot 10^{-3}$ &
0.99 \\ \cline{2-8}
1024 & $3$ & 2.9999 & $3.01\cdot 10^{-3}$ & 1.18 & 2.9998 & $4.40\cdot
10^{-3}$ & 0.98 \\ \cline{2-8}
& 4 & 3.9997 & $2.82\cdot 10^{-3}$ & 1.04 & 3.9998 & $4.39\cdot 10^{-3}$ &
0.97 \\ \hline
\end{tabular}%
\end{equation*}

\begin{center}
Table 2: Sample means and variances of $\widehat{\alpha }_{L}$ and $\widehat{%
\alpha }_{J_{L}}$, for different values of $L$ and $\alpha _{0}$, $B=2$.
\end{center}

\begin{equation*}
\begin{tabular}{cc|c|c|c|c|}
\cline{3-4}\cline{3-6}
&  & \multicolumn{2}{|c|}{$\widehat{\alpha }_{L}$ - Shapiro test} &
\multicolumn{2}{|c|}{$\widehat{\alpha }_{J_{L}}$ - Shapiro test} \\ \hline
\multicolumn{1}{|c}{$L$} & \multicolumn{1}{|c|}{$\alpha _{0}$} & $W$ &
p-value & $W$ & p-value \\ \hline
\multicolumn{1}{|c|}{} & \multicolumn{1}{|c|}{$2$} & 0.9921 & 0.35 & 0.9952
& 0.32 \\ \cline{2-6}
\multicolumn{1}{|c|}{256} & \multicolumn{1}{|c|}{$3$} & 0.9945 & 0.68 &
0.9939 & 0.59 \\ \cline{2-6}
\multicolumn{1}{|c|}{} & \multicolumn{1}{|c|}{$4$} & 0.9906 & 0.22 & 0.9909
& 0.24 \\ \hline
\multicolumn{1}{|c|}{} & \multicolumn{1}{|c|}{$2$} & 0.9956 & 0.83 & 0.9945
& 0.12 \\ \cline{2-6}
\multicolumn{1}{|c|}{512} & \multicolumn{1}{|c|}{$3$} & 0.9915 & 0.30 &
0.9943 & 0.65 \\ \cline{2-6}
\multicolumn{1}{|c|}{} & \multicolumn{1}{|c|}{4} & 0.9946 & 0.70 & 0.9895 &
0.15 \\ \hline
\multicolumn{1}{|c|}{} & \multicolumn{1}{|c|}{$2$} & 0.9885 & 0.11 & 0.9955
& 0.82 \\ \cline{2-6}
\multicolumn{1}{|c|}{1024} & \multicolumn{1}{|c|}{$3$} & 0.9967 & 0.94 &
0.9949 & 0.74 \\ \cline{2-6}
\multicolumn{1}{|c|}{} & \multicolumn{1}{|c|}{4} & 0.9915 & 0.29 & 0.9959 &
0.87 \\ \hline
\end{tabular}%
\end{equation*}

\begin{center}
Table 3: Shapiro-Wilk Gaussianity test of $\widehat{\alpha }_{L}$ and $%
\widehat{\alpha }_{J_{L}}$, for different values of $L$ and $\alpha _{0}$, $%
B=2$.

\begin{figure}[tbp]
\centering
\includegraphics[width=\textwidth]{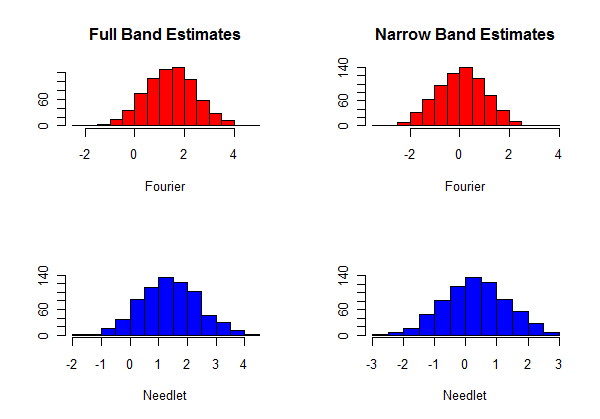}
\caption{Comparison among normalized distribution of Full Band and Narrow
band estimates - $\protect\alpha _{0}=2$, $L=1024$, $L_{1}=724$.}
\label{fig:2}
\end{figure}

\begin{equation*}
\begin{tabular}{c|c|c|c|c|c|c|c|c|}
\cline{2-5}\cline{2-9}
& \multicolumn{2}{|c|}{$\widehat{\alpha }_{L}$} & \multicolumn{2}{|c}{
Shapiro test} & \multicolumn{2}{|c|}{$\widehat{\alpha }_{J_{L}}$} &
\multicolumn{2}{|c|}{Shapiro test} \\ \hline
\multicolumn{1}{|c|}{$\alpha _{0}$} & mean & sd & W & p-val & mean & sd & W
& p-val \\ \hline
\multicolumn{1}{|c|}{$2$} & 2.004 & $2.68\cdot 10^{-3}$ & 0.9985 & 0.35 &
2.007 & $2.75\cdot 10^{-3}$ & 0.9989 & 0.32 \\ \hline
\multicolumn{1}{|c|}{$3$} & 3.004 & $2.76\cdot 10^{-3}$ & 0.9988 & 0.91 &
3.004 & $2.79\cdot 10^{-3}$ & 0.9988 & 0.94 \\ \hline
\multicolumn{1}{|c|}{$4$} & 4.004 & $2.88\cdot 10^{-3}$ & 0.9978 & 0.89 &
4.004 & $2.97\cdot 10^{-3}$ & 0.9954 & 0.28 \\ \hline
& \multicolumn{2}{|c|}{$\widehat{\alpha }_{L;L_{1}}$} & \multicolumn{2}{|c}{
Shapiro test} & \multicolumn{2}{|c|}{$\widehat{\alpha }_{J_{L};J_{1}}$} &
\multicolumn{2}{|c|}{Shapiro test} \\ \hline
\multicolumn{1}{|c|}{$\alpha _{0}$} & mean & sd & W & p-val & mean & sd & W
& p-val \\ \hline
\multicolumn{1}{|c|}{$2$} & 1.999 & $5.51\cdot 10^{-3}$ & 0.9907 & 0.72 &
2.002 & $6.69\cdot 10^{-3}$ & 0.9876 & 0.25 \\ \hline
\multicolumn{1}{|c|}{$3$} & 3.001 & $5.66\cdot 10^{-3}$ & 0.9989 & 0.95 &
3.001 & $6.98\cdot 10^{-3}$ & 0.9979 & 0.53 \\ \hline
\multicolumn{1}{|c|}{$4$} & 4.001 & $5.59\cdot 10^{-3}$ & 0.9960 & 0.42 &
4.001 & $6.40\cdot 10^{-3}$ & 0.9974 & 0.78 \\ \hline
\end{tabular}%
\end{equation*}%
Table 4: Sample means, variances and Shapiro Wilk Gaussianity test of $%
\widehat{\alpha }_{L}$ and $\widehat{\alpha }_{J_{L}}$, for different values
of $\alpha _{0}$
\end{center}

\appendix{}

\section{Auxiliary Results\label{sec: auxres}}

\bigskip

This Section presents some results, mainly focussed on the behaviour of the $%
K_{j}\left( \alpha \right) $, which will be useful to develop consistency
and asymptotic behaviour of the estimator (\ref{alphaneed}). We remark that
all these results hold under Condition \ref{REGULNEED}.

\begin{proposition}
\label{propKj}Let $K_{j}\left( \alpha \right) $, $K_{j,1}\left( \alpha
\right) $ and $K_{j,2}\left( \alpha \right) $ be as in Definition \ref%
{Kfunctions}. Then we have:
\begin{eqnarray}
K_{j}\left( \alpha \right) &=&B^{-\alpha j}\left( \frac{1}{c_{B}}I_{0}\left(
B,\alpha \right) +o_{j}\left( 1\right) \right) \text{ ;}  \label{K0j} \\
\frac{K_{j}\left( \alpha _{0}\right) }{K_{j}\left( \alpha \right) }
&=&B^{j\left( \alpha -\alpha _{0}\right) }\left( I\left( B,\alpha
_{0},\alpha \right) +o_{j}\left( 1\right) \right) \text{,}  \label{K0ratio}
\end{eqnarray}%
where
\begin{equation}
I_{0}\left( B,\alpha \right) =2\int_{B^{-1}}^{B}b^{2}\left( u\right)
u^{1-\alpha }du\text{ ,}  \label{I0_def}
\end{equation}%
\begin{equation}
I\left( B,\alpha _{0},\alpha \right) =\frac{I_{0}\left( B,\alpha _{0}\right)
}{I_{0}\left( B,\alpha \right) }  \label{I_def}
\end{equation}%
so that%
\begin{equation*}
0<c_{1}<I_{0}\left( B\right) <c_{2}<+\infty
\end{equation*}%
and%
\begin{equation}
\left[ I\left( \alpha _{0},\alpha \right) -1\right] =C_{I}\left\vert \alpha
-\alpha _{0}\right\vert  \label{I-ratio}
\end{equation}%
Moreover, for $K_{j,1}\left( \alpha \right) $ and $K_{j,2}\left( \alpha
\right) $ as above, we have:%
\begin{equation}
K_{j,1}\left( \alpha \right) +K_{j}\left( \alpha \right) \log B^{j}=-\frac{1%
}{c_{B}}B^{-\alpha j}\left\{ I_{1}\left( B\right) +o_{j}\left( 1\right)
\right\} \text{ ;}  \label{K1j}
\end{equation}%
\begin{equation}
K_{j,2}\left( \alpha \right) +K_{j}\left( \alpha \right) \log
^{2}B^{j}+2K_{j,1}\left( \alpha \right) \log B^{j}=\frac{1}{c_{B}}B^{-\alpha
j}\left\{ I_{2}\left( B\right) +o_{j}\left( 1\right) \right\}  \label{Kj2}
\end{equation}%
where%
\begin{equation*}
I_{1}\left( B\right) =2\int_{B^{-1}}^{B}b^{2}\left( u\right) u^{1-\alpha
}\left( \log u\right) du\text{ , }I_{2}\left( B\right)
=2\int_{B^{-1}}^{B}b^{2}\left( u\right) u^{1-\alpha }\left( \log u\right)
^{2}du\text{,}
\end{equation*}%
and%
\begin{equation*}
0<c_{3}<I_{2}\left( B\right) <c_{4}<+\infty \text{ .}
\end{equation*}
\end{proposition}

\begin{proof}
Recalling that $N_{j}=c_{B}B^{2j}$ from (\ref{Njdef}), simple calculations
lead to (\ref{K0j}):%
\begin{eqnarray*}
K_{j}\left( \alpha \right) &=&\frac{1}{c_{B}}B^{-\alpha
j}\sum_{l=B^{j-1}}^{B^{j+1}}b^{2}\left( \frac{l}{B^{j}}\right) \frac{1}{B^{j}%
}\frac{\left( 2l+1\right) }{B^{j}}\frac{l^{-\alpha }}{B^{-\alpha j}} \\
&=&B^{-\alpha j}\frac{1}{c_{B}}\left\{ 2\int_{B^{-1}}^{B}b^{2}\left(
u\right) u^{1-\alpha }du+o_{j}\left( 1\right) \right\} \\
&=&B^{-\alpha j}\frac{1}{c_{B}}\left\{ I_{0}\left( B\right) +o_{j}\left(
1\right) \right\} \text{;}
\end{eqnarray*}%
by applying the Lagrange Mean Value Theorem we obtain%
\begin{equation*}
I_{0}\left( B,\alpha \right) =2b^{2}\left( \xi \right) \xi ^{1-\alpha
}\left( B-B^{-1}\right) \text{ for }\xi \in \left[ B^{-1},B\right] \text{ ,}
\end{equation*}%
which is a non-zero, finite positive real number. Obviously:
\begin{equation*}
\frac{K_{j}\left( \alpha \right) }{K_{j}\left( \alpha _{0}\right) }=\frac{%
B^{-\alpha j}\left\{ I_{0}\left( B,\alpha \right) +o_{j}\left( 1\right)
\right\} }{B^{-\alpha _{0}j}\left\{ I_{0}\left( B,\alpha _{0}\right)
+o_{j}\left( 1\right) \right\} }=B^{j\left( \alpha _{0}-\alpha \right)
}\left\{ I\left( B,\alpha ,\alpha _{0}\right) +o_{j}\left( 1\right) \right\}
\text{ .}
\end{equation*}

Because by construction $0<C_{\min }\leq b^{2}\left( \xi \right) \leq
C_{\max }<\infty $, we have%
\begin{equation*}
C_{\min }B^{1-\alpha }\left( B-B^{-1}\right) \leq I_{0}\left( B,\alpha
\right) \leq C_{\max }B^{1-\alpha }\left( B-B^{-1}\right) \text{ ,}
\end{equation*}%
so that
\begin{equation*}
\left( \frac{C_{\max }}{C_{\min }}B^{\alpha -\alpha _{0}}\right) ^{-1}\leq
I\left( B,\alpha _{0},\alpha \right) \leq \frac{C_{\max }}{C_{\min }}%
B^{\alpha -\alpha _{0}}\text{ .}
\end{equation*}%
Hence, fixing $C_{\min }/C_{\max }\leq C_{I}\leq C_{\max }/C_{\min }$, we
obtain (\ref{I-ratio}). We recall moreover that $I\left( B,\alpha ,\alpha
\right) =1$.

As far as $K_{j,1}\left( \alpha \right) $ is concerned, we prove (\ref{K1j}%
). In fact:%
\begin{equation*}
K_{j,1}\left( \alpha \right) +K_{j}\left( \alpha \right) \log B^{j}=
\end{equation*}%
\begin{eqnarray*}
&=&\frac{-B^{-\alpha j}}{c_{B}}\sum_{l=B^{j-1}}^{B^{j+1}}b^{2}\left( \frac{l%
}{B^{j}}\right) \frac{1}{B^{j}}\frac{\left( 2l+1\right) }{B^{j}}\frac{%
l^{-\alpha }}{B^{-\alpha j}}\left( \log \frac{l}{B^{j}}\right) \\
&=&\frac{B^{-\alpha j}}{c_{B}}\left( -\int_{B^{-1}}^{B}2b^{2}\left( u\right)
u^{1-\alpha }\left( \log u\right) du+o_{j}\left( 1\right) \right) \\
&=&-B^{-\alpha j}\frac{1}{c_{B}}\left\{ I_{1}\left( B\right) +o_{j}\left(
1\right) \right\} \text{ .}
\end{eqnarray*}

Now, we have, by applying again Lagrange Mean Value Theorem:%
\begin{equation*}
I_{1}\left( B\right) =2b^{2}\left( \xi \right) \xi ^{1-\alpha }\log \left(
\xi \right) \left( B^{-1}-B\right) \text{ for }\xi \in \left[ B^{-1},B\right]
\text{ ,}
\end{equation*}%
where $-\infty <-c_{1}\leq I_{1}\left( B\right) \leq c_{2}<+\infty $, so
that
\begin{equation}
K_{j,1}\left( \alpha \right) =-B^{-\alpha j}\frac{1}{c_{B}}\left(
I_{0}\left( B\right) \log B^{j}-I_{1}\left( B\right) +o_{j}\left( 1\right)
\right) \text{ .}  \label{K1see}
\end{equation}%
Similarly, we obtain for (\ref{Kj2}):%
\begin{equation*}
K_{j,2}\left( \alpha \right) +K_{j}\left( \alpha \right) \log
^{2}B^{j}+2K_{j,1}\left( \alpha \right) \log B^{j}
\end{equation*}

\begin{eqnarray*}
&=&B^{-\alpha j}\frac{1}{c_{B}}\sum_{l=B^{j-1}}^{B^{j+1}}b^{2}\left( \frac{l%
}{B^{j}}\right) \frac{1}{B^{j}}\frac{\left( 2l+1\right) }{B^{j}}\frac{%
l^{-\alpha }}{B^{-\alpha j}}\log ^{2}\frac{l}{B^{j}} \\
&=&B^{-\alpha j}\frac{1}{c_{B}}\left\{ \int_{B^{-1}}^{B}2b\left( u\right)
u^{1-\alpha }\left( \log u\right) ^{2}du+o_{j}\left( 1\right) \right\} \\
&=&B^{-\alpha j}\frac{1}{c_{B}}\left\{ I_{2}\left( B\right) +o_{j}\left(
1\right) \right\} \text{ ,}
\end{eqnarray*}%
which is trivially strictly positive and bounded. Hence we have:%
\begin{equation}
K_{j,2}\left( \alpha \right) =B^{-\alpha j}\left( I_{0}\left( B\right)
\left( \log B^{j}\right) ^{2}+2I_{1}(B)\log B^{j}+I_{2}\left( B\right)
+o_{j}\left( 1\right) \right) \text{ .}  \label{K2see}
\end{equation}
\end{proof}

We now provide some further auxiliary results on the function $K_{j}\left(
\alpha \right) ;$ these results are exploited in the proofs for consistency
and elsewhere.

\begin{corollary}
\label{Kjcoroll}As $j\rightarrow \infty ,$ we have:%
\begin{equation}
\frac{K_{j,1}\left( \alpha \right) }{K_{j}\left( \alpha \right) }=-\log
B^{j}+\frac{I_{1}\left( B\right) }{I_{0}\left( 0\right) }=\log B^{j}\left(
-1+o_{j}\left( 1\right) \right) \text{ ;}  \label{K1lim}
\end{equation}%
\begin{equation*}
\left\{ 2\left( \frac{K_{j,1}\left( \alpha \right) }{K_{j}\left( \alpha
\right) }\right) ^{2}-\frac{K_{j,2}\left( \alpha \right) }{K_{j}\left(
\alpha \right) }\right\}
\end{equation*}%
\begin{eqnarray}
&=&\left( \log B^{j}\right) ^{2}+2\frac{I_{1}\left( B\right) }{I_{0}\left(
0\right) }\log B^{j}+2\left( \frac{I_{1}\left( B\right) }{I_{0}\left(
0\right) }\right) ^{2}+\frac{I_{2}\left( B\right) }{I_{0}\left( B\right) }%
\text{ }  \label{K2lim} \\
&=&\left( \log B^{j}\right) ^{2}\left( 1+o_{j}\left( 1\right) \right) \text{.%
}  \notag
\end{eqnarray}

\begin{proof}
From (\ref{K1see}), we obtain:%
\begin{equation*}
\frac{K_{j,1}\left( \alpha \right) }{K_{j}\left( \alpha \right) }=-\log
B^{j}(1+\frac{I_{1}\left( B\right) }{I_{0}\left( B\right) }+o\left( 1\right)
)\text{ .}
\end{equation*}%
Also, in view of (\ref{K2see}),
\begin{equation*}
\frac{K_{j,2}\left( \alpha \right) }{K_{j}\left( \alpha \right) }=\left(
\log B^{j}\right) ^{2}+2\frac{I_{1}\left( B\right) }{I_{0}\left( 0\right) }%
\log B^{j}+\frac{I_{2}\left( B\right) }{I_{0}\left( B\right) }+o_{j}\left(
1\right) \text{ .}
\end{equation*}%
Thus%
\begin{equation*}
\left\{ 2\left( \frac{K_{j,1}\left( \alpha \right) }{K_{j}\left( \alpha
\right) }\right) ^{2}-\frac{K_{j,2}\left( \alpha \right) }{K_{j}\left(
\alpha \right) }\right\}
\end{equation*}%
\begin{equation*}
=\left( \log B^{j}\right) ^{2}+2\left( \frac{I_{1}\left( B\right) }{%
I_{0}\left( 0\right) }\right) ^{2}+2\frac{I_{1}\left( B\right) }{I_{0}\left(
0\right) }\log B^{j}+\frac{I_{2}\left( B\right) }{I_{0}\left( B\right) }%
\text{ .}
\end{equation*}
\end{proof}
\end{corollary}

\begin{remark}
It is immediate to see that under Condition \ref{REGULNEED}, equation (\ref%
{expbeta2}) becomes (using Proposition \ref{propKj}):%
\begin{equation}
\mathbb{E}\left( \frac{1}{N_{j}}\frac{\sum_{k}\beta _{jk}^{2}}{%
G_{0}K_{j}\left( \alpha _{0}\right) }\right) =1+O\left( B^{-j}\right) \text{
,}  \label{expvaluebetabias}
\end{equation}%
while under Condition \ref{REGULNEED2} we have%
\begin{equation}
\mathbb{E}\left( \frac{1}{N_{j}}\frac{\sum_{k}\beta _{jk}^{2}}{%
G_{0}K_{j}\left( \alpha _{0}\right) }\right) =1+\kappa B^{-j}+o\left(
B^{-j}\right) \text{ .}  \label{expvaluebetak}
\end{equation}
\end{remark}

\begin{proposition}
\label{propsumB}Let $s>0$, $B>1$, $J_{1}<J_{L}$ Then:%
\begin{eqnarray}
\sum_{j=J_{1}}^{J_{L}}B^{sj} &=&\frac{B^{s}}{B^{s}-1}\left(
B^{sJ_{L}}-B^{sJ_{1}-1}\right) \text{ ,}  \label{Jint0} \\
\sum_{j=J_{1}}^{J_{L}}B^{sj}\log B^{j} &=&\frac{B^{s}}{B^{s}-1}\log B\left(
B^{sJ_{L}}\left( J_{L}-\frac{1}{B^{s}-1}\right) \right.  \label{Jint1} \\
&&\left. -B^{s\left( J_{1}-1\right) }\left( \left( J_{1}-1\right) -\frac{1}{%
B^{s}-1}\right) \right) \text{ ,}  \notag \\
\sum_{j=J_{1}}^{J_{L}}B^{sj}j^{2}\log ^{2}B &=&\frac{B^{s}}{B^{s}-1}\log
^{2}B\left( B^{sJ_{L}}\left( \left( J_{L}-\frac{1}{B^{s}-1}\right) ^{2}+%
\frac{B^{s}}{\left( B^{s}-1\right) ^{2}}\right) \right.  \notag \\
&&-\left. B^{s\left( J_{1}-1\right) }\left( \left( \left( J_{1}-1\right) -%
\frac{1}{B^{s}-1}\right) ^{2}+\frac{B^{s}}{\left( B^{s}-1\right) ^{2}}%
\right) \right)  \label{Jint2b}
\end{eqnarray}
\end{proposition}

\begin{proof}
The first result is trivial:%
\begin{eqnarray*}
\sum_{j=J_{1}}^{J_{L}}B^{sj}
&=&\sum_{j=0}^{J_{L}}B^{sj}-\sum_{j=0}^{J_{1}}B^{sj}=\frac{B^{sJ_{L}+1}-1}{%
B^{s}-1}-\frac{B^{sJ_{1}}-1}{B^{s}-1} \\
&=&\frac{B^{s}}{B^{s}-1}\left( B^{sJ_{L}}-B^{s\left( J_{1}-1\right) }\right)
\text{.}
\end{eqnarray*}%
Likewise, we obtain:%
\begin{equation*}
\frac{d}{ds}\frac{B^{s}}{B^{s}-1}=\frac{-B^{s}\log B}{\left( B^{s}-1\right)
^{2}}
\end{equation*}%
\begin{eqnarray*}
\sum_{j=J_{1}}^{J_{L}}B^{sj}\log B^{j} &=&\frac{d}{ds}\left[
\sum_{j=J_{1}}^{J_{L}}\exp \left\{ sj\log B\right\} \right] =\frac{d}{ds}%
\left[ \frac{B^{s}}{B^{s}-1}\left( B^{sJ_{L}}-B^{s\left( J_{1}-1\right)
}\right) \right] \\
&=&\frac{-B^{s}\log B}{\left( B^{s}-1\right) ^{2}}\left(
B^{sJ_{L}}-B^{s\left( J_{1}-1\right) }\right) +\frac{B^{s}\log B}{B^{s}-1}%
\left( J_{L}B^{sJ_{L}}-J_{1}B^{s\left( J_{1}-1\right) }\right) \\
&=&\frac{B^{s}}{B^{s}-1}\log B\left( B^{sJ_{L}}\left( J_{L}-\frac{1}{B^{s}-1}%
\right) -B^{s\left( J_{1}-1\right) }\left( \left( J_{1}-1\right) -\frac{1}{%
B^{s}-1}\right) \right) \text{ .}
\end{eqnarray*}%
Finally, we have:%
\begin{equation*}
\sum_{j=J_{1}}^{J_{L}}B^{sj}j^{2}\left( \log B\right) ^{2}=\frac{d^{2}}{%
ds^{2}}\left\{ \frac{B^{s}}{B^{s}-1}\left\{ B^{sJ_{L}}-B^{s\left(
J_{1}-1\right) }\right\} \right\}
\end{equation*}%
\begin{equation*}
=\frac{d}{ds}\left\{ \frac{B^{s}\log B}{B^{s}-1}\left( B^{sJ_{L}}\left(
J_{L}-\frac{1}{B^{s}-1}\right) -B^{s\left( J_{1}-1\right) }\left( \left(
J_{1}-1\right) -\frac{1}{B^{s}-1}\right) \right) \right\}
\end{equation*}%
\begin{eqnarray*}
&=&\frac{-B^{s}\log ^{2}B}{\left( B^{s}-1\right) ^{2}}\left(
B^{sJ_{L}}\left( J_{L}-\frac{1}{B^{s}-1}\right) -B^{s\left( J_{1}-1\right)
}\left( \left( J_{1}-1\right) -\frac{1}{B^{s}-1}\right) \right) \\
&&+\frac{B^{s}\log ^{2}B}{B^{s}-1}\left( B^{sJ_{L}}\left( J_{L}^{2}-\frac{1}{%
B^{s}-1}J_{L}+\frac{B^{s}}{\left( B^{s}-1\right) ^{2}}\right) \right) \\
&&-\frac{B^{s}\log ^{2}B}{B^{s}-1}\left( B^{sJ_{1}}\left( \left(
J_{1}-1\right) ^{2}-\frac{1}{B^{s}-1}\left( J_{1}-1\right) +\frac{B^{s}}{%
\left( B^{s}-1\right) ^{2}}\right) \right)
\end{eqnarray*}%
\begin{eqnarray*}
&=&\frac{B^{s}\log ^{2}B}{B^{s}-1}\left( \left( B^{sJ_{L}}\left( \left(
J_{L}-\frac{1}{B^{s}-1}\right) ^{2}+\frac{B^{s}}{\left( B^{s}-1\right) ^{2}}%
\right) \right) \right. \\
&&-\left. \left( B^{s\left( J_{1}-1\right) }\left( \left( \left(
J_{1}-1\right) -\frac{1}{B^{s}-1}\right) ^{2}+\frac{B^{s}}{\left(
B^{s}-1\right) ^{2}}\right) \right) \right) \text{ .}
\end{eqnarray*}
\end{proof}

The next result combines (\ref{Jint0}), (\ref{Jint1}), (\ref{Jint2b}).

\begin{corollary}
\label{sumcorollary}Let
\begin{equation*}
Z_{J_{L};J_{1}}\left( s\right) =\left( \sum_{j=J_{1}}^{J_{L}}B^{sj}\right)
\left( \sum_{j=J_{1}}^{J_{L}}B^{sj}j^{2}\log ^{2}B\right) -\left(
\sum_{j=J_{1}}^{J_{L}}B^{sj}j\log B\right) ^{2}\text{ .}
\end{equation*}%
Then we have:%
\begin{equation}
Z_{J_{L};J_{1}}\left( s\right) =\left( \frac{B^{s}\log B}{B^{s}-1}\right)
^{2}\left[ \frac{B^{s}}{\left( B^{s}-1\right) ^{2}}\left(
B^{sJ_{L}}-B^{s\left( J_{1}-1\right) }\right) ^{2}-B^{s\left(
J_{L}+J_{1}-1\right) }\left( J_{L}-\left( J_{1}-1\right) \right) ^{2}\right]
\label{ZJLJ1}
\end{equation}%
Moreover if $J_{1}=1\,$, we have%
\begin{equation*}
Z_{J_{L}}\left( s\right) =\left( \sum_{j=1}^{J_{L}}B^{sj}\right) \left(
\sum_{j=1}^{J_{L}}B^{sj}j^{2}\log ^{2}B\right) -\left(
\sum_{j=1}^{J_{L}}B^{sj}j\log B\right) ^{2}\text{ ,}
\end{equation*}%
so that
\begin{equation}
\lim_{J_{L}\rightarrow \infty }B^{-2sJ_{L}}Z_{J_{L}}\left( s\right) =\log
^{2}B\frac{B^{3s}}{(B^{s}-1)^{4}}\text{ .}  \label{ZJLlim}
\end{equation}
\end{corollary}

\begin{proof}
Recalling (\ref{Jint0}), (\ref{Jint1}) and (\ref{Jint2b}), we have:%
\begin{equation*}
\left( \sum_{j=1}^{J_{L}}B^{sj}\right) \left(
\sum_{j=1}^{J_{L}}B^{sj}j^{2}\log ^{2}B\right)
\end{equation*}%
\begin{eqnarray*}
&=&\left( \frac{B^{s}\log B}{B^{s}-1}\right) ^{2}\left(
B^{sJ_{L}}-B^{s\left( J_{1}-1\right) }\right) \times \left( \left(
B^{sJ_{L}}\left( \left( J_{L}-\frac{1}{B^{s}-1}\right) ^{2}+\frac{B^{s}}{%
\left( B^{s}-1\right) ^{2}}\right) \right) \right. \\
&&-\left. \left( B^{s\left( J_{1}-1\right) }\left( \left( \left(
J_{1}-1\right) -\frac{1}{B^{s}-1}\right) ^{2}+\frac{B^{s}}{\left(
B^{s}-1\right) ^{2}}\right) \right) \right) \\
&=&\left( \frac{B^{s}\log B}{B^{s}-1}\right) ^{2}\times \text{ }\left[
B^{2sJ_{L}}\left( \left( J_{L}-\frac{1}{B^{s}-1}\right) ^{2}+\frac{B^{s}}{%
\left( B^{s}-1\right) ^{2}}\right) \right. \\
&&+B^{2s\left( J_{1}-1\right) }\left( \left( \left( J_{1}-1\right) -\frac{1}{%
B^{s}-1}\right) ^{2}+\frac{B^{s}}{\left( B^{s}-1\right) ^{2}}\right) \\
&&-\left. B^{s\left( J_{L}+J_{1}-1\right) }\left( \left( J_{L}-\frac{1}{%
B^{s}-1}\right) ^{2}+\left( \left( J_{1}-1\right) -\frac{1}{B^{s}-1}\right)
^{2}+\frac{2B^{s}}{\left( B^{s}-1\right) ^{2}}\right) \right] \text{ ;}
\end{eqnarray*}%
while, on the other hand:%
\begin{equation*}
\left( \sum_{j=1}^{J_{L}}B^{sj}j\log B\right) ^{2}
\end{equation*}%
\begin{eqnarray*}
&=&\left( \frac{B^{s}\log B}{B^{s}-1}\right) ^{2}\left[ B^{2sJ_{L}}\left(
J_{L}-\frac{1}{B^{s}-1}\right) ^{2}+B^{2s\left( J_{1}-1\right) }\left(
\left( J_{1}-1\right) -\frac{1}{B^{s}-1}\right) ^{2}\right. \\
&&-\left. 2B^{s\left( J_{L}+J_{1}-1\right) }\left( J_{L}-\frac{1}{B^{s}-1}%
\right) \left( \left( J_{1}-1\right) -\frac{1}{B^{s}-1}\right) \right] \text{
,}
\end{eqnarray*}

so that:%
\begin{equation*}
Z_{J_{L};J_{1}}\left( s\right) =\left( \frac{B^{s}\log B}{B^{s}-1}\right)
^{2}\left[ \frac{B^{s}}{\left( B^{s}-1\right) ^{2}}\left(
B^{sJ_{L}}-B^{sJ_{1}}\right) ^{2}-B^{s\left( J_{L}+J_{1}\right) }\left(
J_{L}-J_{1}\right) ^{2}\right] \text{ .}
\end{equation*}

Clearly if $J_{1}=1$%
\begin{equation*}
Z_{J_{L}}\left( s\right) =B^{2sJ_{L}}\log ^{2}B\frac{B^{3s}}{(B^{s}-1)^{4}}%
+o\left( B^{2sJ_{L}}\right) \text{, }
\end{equation*}%
as claimed.
\end{proof}

\begin{lemma}
\label{cumulants}Let $A_{j}$ and $B_{j}$ be defined as in (\ref{Aj_cum}) and
(\ref{Bj_cum}). As $J_{L}\rightarrow \infty $%
\begin{equation*}
\frac{1}{B^{4J_{L}}}cum\left\{
\sum_{l_{1}}(A_{j_{1}}+B_{j_{1}}),\sum_{l_{2}}(A_{j_{2}}+B_{j_{2}}),%
\sum_{l_{3}}(A_{j_{3}}+B_{j_{3}}),\sum_{l_{4}}(A_{j_{4}}+B_{j_{4}})\right\}
\end{equation*}%
\begin{equation*}
=O_{J_{L}}\left( \frac{J_{L}^{4}\log ^{4}B}{B^{2J_{L}}}\right) \text{ .}
\end{equation*}
\end{lemma}

\begin{proof}
It is readily checked (see also \cite{dlm}) that%
\begin{equation*}
cum\left\{ \widehat{C}_{l},\widehat{C}_{l},\widehat{C}_{l},\widehat{C}%
_{l}\right\} =O\left( l^{-3}l^{-4\alpha _{0}}\right) \text{ .}
\end{equation*}%
Let us compute:%
\begin{equation*}
C_{j_{1},j_{2},j_{3}j_{4}}^{4}=cum\left( \frac{\sum_{k}\beta _{j_{1}k}^{2}}{%
N_{j_{1}}G_{0}K_{j_{1}}\left( \alpha _{0}\right) },\frac{\sum_{k}\beta
_{j_{2}k}^{2}}{N_{j_{2}}G_{0}K_{j_{2}}\left( \alpha _{0}\right) },\frac{%
\sum_{k}\beta _{j_{3}k}^{2}}{N_{j_{3}}G_{0}K_{j_{3}}\left( \alpha
_{0}\right) },\frac{\sum_{k}\beta _{j_{4}k}^{2}}{N_{j_{4}}G_{0}K_{j_{4}}%
\left( \alpha _{0}\right) }\right)
\end{equation*}%
\begin{equation*}
=\left( \prod_{i=1}^{4}\frac{1}{N_{j_{i}}G_{0}K_{j_{i}}\left( \alpha
_{0}\right) }\right) cum\left( \sum_{k}\beta _{j_{1}k}^{2},\sum_{k}\beta
_{j_{2}k}^{2},\sum_{k}\beta _{j_{3}k}^{2},\sum_{k}\beta _{j_{4}k}^{2}\right)
\end{equation*}%
\begin{equation*}
=\sum_{l_{1},l_{2},l_{3},l_{4}}\left( \prod_{i=1}^{4}\frac{b^{2}\left( \frac{%
l_{i}}{B^{j_{i}}}\right) \left( \frac{2l_{i}+1}{4\pi }\right) }{%
N_{j_{i}}G_{0}K_{j_{i}}\left( \alpha _{0}\right) }\right) cum\left( \widehat{%
C}_{l_{1}},\widehat{C}_{l_{2}},\widehat{C}_{l_{3}},\widehat{C}_{l_{4}}\right)
\end{equation*}%
\begin{equation*}
=\sum_{l}\left( \frac{2l+1}{4\pi }\right) ^{4}\left( \prod_{i=1}^{4}\frac{%
b^{2}\left( \frac{l}{B^{j_{i}}}\right) }{N_{j_{i}}G_{0}K_{j_{i}}\left(
\alpha _{0}\right) }\right) cum\left( \widehat{C}_{l},\widehat{C}_{l},%
\widehat{C}_{l},\widehat{C}_{l}\right) +o\left( B^{-4j}\right)
\end{equation*}%
\begin{equation*}
=O\left( \sum_{l}\left( \prod_{i=1}^{4}B^{\left( \alpha _{0-}2\right)
j_{i}}b^{2}\left( \frac{l}{B^{j_{i}}}\right) \right) B^{\left( 2-4\alpha
_{0}\right) j}\left( l^{1-4\alpha _{0}}\right) \right)
\end{equation*}%
\begin{equation*}
=O\left( B^{-6j}\prod_{i=1}^{4}\delta _{j}^{j_{i}}\right) \text{ .}
\end{equation*}%
Then we have%
\begin{eqnarray*}
&&cum\left\{ \frac{\widehat{G}_{J_{L}}(\alpha _{0})}{G_{0}},\frac{\widehat{G}%
_{J_{L}}(\alpha _{0})}{G_{0}},\frac{\widehat{G}_{J_{L}}(\alpha _{0})}{G_{0}},%
\frac{\widehat{G}_{J_{L}}(\alpha _{0})}{G_{0}}\right\} \\
&=&O\left( \frac{1}{B^{8J_{L}}}%
\sum_{j_{1}j_{2}j_{3}j_{4}}N_{j_{1}}N_{j_{2}}N_{j_{3}}N_{j_{4}}C_{j_{1},j_{2},j_{3}j_{4}}^{4}\right)
\\
&=&O\left( \frac{1}{B^{8J_{L}}}\sum_{j}B^{2j}\right) =O\left(
B^{-6J_{L}}\right) \text{ .}
\end{eqnarray*}%
As in \cite{dlm}, the proof can be divided into 5 cases, corresponding
respectively to
\begin{equation*}
\frac{1}{B^{4J_{L}}}cum\left\{
\sum_{j_{1}}A_{j_{1}},\sum_{j_{2}}A_{j_{2}},\sum_{j_{3}}A_{j_{3}},%
\sum_{j_{4}}A_{j_{4}}\right\} ,\frac{1}{B^{4J_{L}}}cum\left\{
\sum_{j_{1}}B_{j_{1}},\sum_{j_{2}}B_{j_{2}},\sum_{j_{3}}B_{j_{3}},%
\sum_{j_{4}}B_{j_{4}}\right\}
\end{equation*}%
\begin{equation*}
\frac{1}{B^{4J_{L}}}cum\left\{
\sum_{j_{1}}A_{j_{1}},\sum_{j_{2}}B_{j_{2}},\sum_{j_{3}}B_{j_{3}},%
\sum_{j_{4}}B_{j_{4}}\right\} ,\frac{1}{B^{4J_{L}}}cum\left\{
\sum_{j_{1}}A_{j_{1}},\sum_{j_{2}}A_{j_{2}},\sum_{j_{3}}B_{j_{3}},%
\sum_{j_{4}}B_{j_{4}}\right\}
\end{equation*}%
and%
\begin{equation*}
\frac{1}{B^{4J_{L}}}cum\left\{
\sum_{j_{1}}A_{j_{1}},\sum_{j_{2}}A_{j_{2}},\sum_{j_{3}}A_{j_{3}},%
\sum_{j_{4}}B_{j_{4}}\right\} \text{ ,}
\end{equation*}%
where we have used \ref{Aj_cum}, \ref{Bj_cum}. We have for instance%
\begin{eqnarray*}
&&\frac{1}{B^{4J_{L}}}cum\left\{
\sum_{j_{1}}A_{j_{1}},\sum_{j_{2}}A_{j_{2}},\sum_{j_{3}}A_{j_{3}},%
\sum_{j_{4}}A_{j_{4}}\right\} \\
&=&O\left( \frac{1}{B^{4J_{L}}}\sum_{j_{1},j_{2}j_{3},j_{4}}\prod_{i=1}^{4}%
\left( B^{2j_{i}}\log B^{j_{i}}\right) C_{j_{1},j_{2},j_{3}j_{4}}^{4}\right)
\\
&=&O\left( \frac{1}{B^{4J_{L}}}\sum_{j}B^{8j}\log ^{4}B^{j}B^{-6j}\right) =O(%
\frac{1}{B^{4J_{L}}}\sum_{j}\log ^{4}B^{j}B^{2j})=O(\frac{\log ^{4}B^{J_{L}}%
}{B^{2J_{L}}})\text{ ;}
\end{eqnarray*}%
and%
\begin{eqnarray*}
&&\frac{1}{B^{4J_{L}}}cum\left\{
\sum_{j_{1}}B_{j_{1}},\sum_{j_{2}}B_{j_{2}},\sum_{j_{3}}B_{j_{3}},%
\sum_{j_{4}}B_{j_{4}}\right\} \\
&=&\frac{1}{B^{4J_{L}}}\left\{ \sum_{j}B^{2j}\log B^{j}\right\}
^{4}cum\left\{ \frac{\widehat{G}_{J_{L}}(\alpha _{0})}{G_{0}},\frac{\widehat{%
G}_{J_{L}}(\alpha _{0})}{G_{0}},\frac{\widehat{G}_{J_{L}}(\alpha _{0})}{G_{0}%
},\frac{\widehat{G}_{J_{L}}(\alpha _{0})}{G_{0}}\right\} \\
&=&O\left( \log ^{4}B^{J_{L}}B^{-2J_{L}}\right) \text{ ;}
\end{eqnarray*}%
The proof for the remaining terms is entirely analogous, and hence
omitted.\bigskip
\end{proof}

\end{document}